\documentclass[twoside]{ut-thesis}

\usepackage{amsmath, amsthm, amssymb}
\usepackage{graphicx}
\usepackage{subfigure}
\usepackage{multirow}

\degree{Doctor of Philosophy}
\department{Mathematics}
\gradyear{2009}
\author{Derek Krepski}
\title{Pre-quantization of the Moduli Space of Flat $G$-Bundles}

\input{diagrams.tex}
\begin{document}



\newcommand{\N}{\mathbb{N}}
\newcommand{\Z}{\mathbb{Z}}
\newcommand{\Zp}{\Z_p}
\newcommand{\Zn}{\Z_n}
\newcommand{\R}{\mathbb{R}}
\newcommand{\Q}{\mathbb{Q}}
\newcommand{\RP}{\R P}
\newcommand{\half}{\frac{1}{2}}
\newcommand{\Ztwo}{\Z_2}
\newcommand{\toby}[1]{\stackrel{#1}{\longrightarrow}}
\newcommand{\Lg}{L\mathfrak{g}}
\newcommand{\Lghat}{\widehat{\Lg}}
\newcommand{\LGhat}{\widehat{LG}}
\newcommand{\Hol}{\mathrm{Hol}}
\newcommand{\pr}{\mathrm{pr}}
\newcommand{\id}{\mathrm{id}}

\newcommand{\C}{\mathcal{C}}
\newcommand{\SU}{\mathrm{SU}}
\newcommand{\PU}{\mathrm{PU}}
\newcommand{\SO}{\mathrm{SO}}
\newcommand{\nobtwomat}[4]{ \begin{array}{cc} #1 & #2 \\
                                             #3 & #4 \\ \end{array} }

\newarrow{Equals}{=}{=}{=}{=}{=}


\theoremstyle{plain}
\newtheorem{thm}{Theorem}[section]
\newtheorem{lemma}[thm]{Lemma}
\newtheorem{fact}[thm]{Fact}
\newtheorem{prop}[thm]{Proposition}
\newtheorem{cor}[thm]{Corollary}
\newtheorem*{thmnonumber}{Theorem}
\newtheorem*{propnonumber}{Proposition}
\newtheorem{appthm}{Theorem}[chapter]

\theoremstyle{definition}
\newtheorem{defn}[thm]{Definition}
\newtheorem{eg}[thm]{Example}

\theoremstyle{remark}
\newtheorem{remark}[thm]{Remark}
\newtheorem{case}{Case}

\newtheorem*{mainthmone}{\emph{\textbf{Theorem \ref{thm:coho-calc-equiv}}}}
\newtheorem*{mainthmtwo}{\emph{\textbf{Theorem \ref{thm:preqofPM}}}}

\begin{preliminary}

\maketitle



\begin{abstract}


This thesis studies the pre-quantization of quasi-Hamiltonian group actions from a cohomological viewpoint.  The compatibility of pre-quantization with symplectic reduction and the fusion product are established, and are used to understand the necessary and sufficient conditions for the pre-quantization of $M_G(\Sigma)$, the moduli space of flat $G$-bundles over a closed surface $\Sigma$.

For a simply connected, compact, simple Lie group $G$, $M_G(\Sigma)$ is known to be pre-quantizable at integer levels.  For non-simply connected $G$, however, integrality of the level is not sufficient for pre-quantization, and this thesis determines the obstruction---namely a certain cohomology class in $H^3(G\times G;\Z)$---that places further restrictions on the underlying level. The levels that admit a pre-quantization of the moduli space are determined explicitly for all non-simply connected, compact, simple Lie groups $G$. Partial results are obtained for the case of a surface $\Sigma$ with marked points.

Also, it is shown that via the bijective correspondence between quasi-Hamiltonian group actions and Hamiltonian loop group actions, the corresponding notions of pre-quantization coincide. 

\end{abstract}



\begin{dedication}
\center{
I dedicate this work to my parents, Walter and Cecilia.  }
\end{dedication}
\newpage


\begin{acknowledgements}

Having written the last word of my thesis, I would like to take a moment to express my gratitude to those who have shaped my experience in graduate school.  

First, I thank my advisors, Eckhard Meinrenken and Paul Selick.  I am grateful for their guidance and support, in all its forms.  To have been able to work so closely with such talented, motivating and enthusiastic mathematicians has been a great pleasure.  

I would also like to thank Lisa Jeffrey and Velimir Jurdjevic for the many helpful  conversations about mathematics, and for their insights into academic life.   Thanks as well to Reyer Sjamaar for his feedback and comments on this work.

My experience as a graduate student would not have been possible without generous financial support from various sources.  In addition to my advisors,  I thank the Natural Sciences and Engineering Research Council of Canada, the Ontario Graduate Scholarship Program, the University of Toronto, and Prof. George Elliott. 

Many thanks as well to the always welcoming and genial Ida Bulat.  Her commitment to the Department of Mathematics and its graduate students is admirable. 

Notwithstanding our many conversations about mathematics, I am also thankful to have some very good friends who helped distract me along the way.  Thank you Gregson, Gadi, Brian, and James for your friendship.

Finally, I would like to thank Heather, whose love and encouragement are truly uplifting.  Every day I consider myself lucky to have you by my side.
\end{acknowledgements}

\tableofcontents



\end{preliminary}



\chapter{Setting the Stage}

\section{Introduction \& Summary of Results}

\subsection*{Geometric quantization}

Pioneered by Souriau \cite{Sou}, Kostant \cite{Kos}, and Kirillov \cite{Ki} in the  1960's, \textsl{geometric quantization} is  a construction for Hamiltonian group actions that was originally aimed at understanding the relationship between classical and quantum physics.  Recall that Hamiltonian group actions (and more generally, symplectic geometry) provide a natural mathematical framework to describe classical physics, whereas the story of quantum physics is told in the language of Hilbert spaces and operator algebras.  Intuitively, quantization is a kind of inverse process to ``taking the classical limit''---realizing a classical theory as a limiting case of a given quantum theory.  Geometric quantization provides a (multi-step) recipe for this inverse process, and this thesis is concerned with the first step in this recipe, known as \textsl{pre-quantization}.  More precisely, this thesis studies pre-quantization in the realm of \textsl{quasi-Hamiltonian} group actions.  

Developed in 1998 by Alekseev, Malkin, and Meinrenken \cite{AMM}, the theory of quasi-Hamiltonian group actions provides an equivalent yet finite dimensional treatment of Hamiltonian loop group actions with proper moment map (see Chapter \ref{chapter:loopgroups}).   Hamiltonian loop group actions behave like ordinary Hamiltonian actions by a compact Lie group $G$, except of course that the underlying Lie group (i.e. the loop group) and the topological space on which it acts are infinite dimensional.  Deciding to work in the framework of quasi-Hamiltonian actions sidesteps any infinite dimensional quandaries that may arise, although it may require a small investment of adjusting one's intuition. 

There is a similarity between Hamiltonian group actions and their quasi-Hamiltonian counterparts that is noteworthy for the purposes of investigating pre-quantization.  Analogous to the way in which a Hamiltonian $G$-action on a symplectic manifold $(X,\omega)$ determines a $G$-equivariant cohomology class $[\omega_G]$  in $H^2_G(X;\R)$, a quasi-Hamiltonian $G$-action on a manifold $M$ determines a \emph{relative} $G$-equivariant cohomology class $ [(\omega,\eta_G)]$ in $H^3_G(\phi;\R)$, where $\phi:M\to G$ is the so-called \textsl{group-valued moment map} (see Chapter \ref{chapter:quasi}).  This similarity helps motivate the definition of pre-quantization in the quasi-Hamiltonian setting.

A pre-quantization of a quasi-Hamiltonian $G$-action on $M$ is defined to be an \textsl{integral lift}---a pre-image via the coefficient homomorphism $H_G^3(\phi;\Z) \to H^3_G(\phi;\R)$---of the distinguished relative cohomology class $[(\omega,\eta_G)] \in H^3_G(\phi;\R)$.   This is analogous to the Hamiltonian setting, where a pre-quantization of a Hamiltonian $G$-action on $X$ may be defined as an integral lift of $[\omega_G] \in H^2(X;\R)$ (see Section \ref{sec:forsymplectics}).  (The special case where $G$ is the trivial group may be used to define pre-quantization of symplectic manifolds.) In either setting, the question of whether a given $G$-action admits a pre-quantization is a topological one, and describing conditions for which such a pre-quantization exists is one of the main objectives of this thesis.
More specifically, this thesis aims to describe the obstructions to pre-quantization for a particular quasi-Hamiltonian $G$-space related to the moduli space of flat $G$-bundles over a closed orientable surface $\Sigma$, where $G$ is a simple compact Lie group.

\subsection*{The moduli space of flat $G$-bundles}

The moduli space  of flat $G$-bundles over a closed orientable surface $\Sigma$, denoted $M_G(\Sigma)$, is an intensely studied object in symplectic geometry, which appears in other areas of mathematics as well.  In algebraic geometry, for example, it is the space of semi-stable holomorphic vector bundles (of prescribed rank and degree) over $\Sigma$ \cite{NS}; and in topological quantum field theory, it is the phase space of Chern-Simons gauge theory on a 3-manifold $\Sigma \times \mathbb{R}$ \cite{Witten}.  The interest in $M_G(\Sigma)$ from the symplectic geometry viewpoint stems from the work of  Atiyah and Bott \cite{AB} in the 1980's, who showed that the moduli space $M_G(\Sigma)$ carries a natural symplectic structure.  

Perhaps the simplest description of the moduli space is as the (usually singular) orbit space $M_G(\Sigma)=\mathrm{Hom}(\pi_1(\Sigma),G)/G$, 
where $\mathrm{Hom}(\pi_1(\Sigma),G)$ is the representation space of the fundamental group $\pi_1(\Sigma)$ on which $G$ acts by  conjugation.

Notice that if $\Sigma$  has genus $g$, a choice of generators for $\pi_1(\Sigma)$ determines an identification 
$$
M_G(\Sigma)\cong \Big\{(a_1, b_1, \ldots, a_g, b_g) \in G^{2g} \, \Big| \,  \prod_{i=1}^g a_ib_ia_i^{-1}b_i^{-1}=1 \Big\} \big/ G
$$
where $G$ acts diagonally  by conjugation on each factor of $G^{2g}$.  Equivalently, if $\phi:G^{2g}\to G$ denotes the product of commutators map, then $M_G(\Sigma)\cong \phi^{-1}(1)/G$, the orbit space of the identity level set $\phi^{-1}(1)$. This  exhibits $M_G(\Sigma)$ as a \textsl{symplectic quotient} of a quasi-Hamiltonian $G$-space, namely $G^{2g}$ (see Chapter \ref{chapter:quasi}).  

Analogous to Meyer-Marsden-Weinstein reduction for Hamiltonian group actions, the orbit space $\phi^{-1}(1)/G$, known as the symplectic quotient,  inherits a symplectic structure from the quasi-Hamiltonian $G$-space $M$ with group-valued moment map $\phi:M\to G$.  Furthermore, as shown in Chapter \ref{chapter:prequantization} (Proposition \ref{prop:qeqpreq=reducedpreq}), a pre-quantization of $M$ induces a pre-quantization of $\phi^{-1}(1)/G$; therefore, one way to analyze the pre-quantization of $M_G(\Sigma)$ is to study its ``parent'' quasi-Hamiltonian $G$-space $G^{2g}$.

One of the significant achievements of Alekseev, Malkin, and Meinrenken's work \cite{AMM} is that the symplectic structure on $M_G(\Sigma)$ coming from the above description as a symplectic quotient of a quasi-Hamiltonian $G$-space, coincides with the symplectic structure given by Atiyah and Bott in \cite{AB}.  In either case, this symplectic structure depends on a choice of \textsl{level} $l>0$, which specifies an inner product on the Lie algebra of $G$ (see Chapter \ref{chapter:quasi}).  The main result of this thesis shows that the quasi-Hamiltonian $G$-space $G^{2g}$ admits a pre-quantization if and only if the underlying level $l$ is a multiple of an integer $l_0(G)$ that depends (only) on the Lie group $G$.  Moreover, $l_0(G)$ is determined for each compact simple Lie group $G$.  

For Lie groups $G$ that are simply connected, it is easy to see that $l_0(G)=1$ (see Remark \ref{remark:simplyconnectedG}).  If $G$ is not simply connected, however, the situation is more complicated; therefore, the main accomplishment of this thesis is the determination of $l_0(G)$ for non-simply connected $G$.

\subsection*{Main result}

The main result of this thesis is the following Theorem from Chapter \ref{chapter:prequantization}.  

\begin{mainthmone} 
Let $G$ be a non-simply connected compact simple
Lie group with universal covering group $\tilde{G}$.  The quasi-Hamiltonian $\tilde{G}$-space $M_G(\Sigma_1^g,b)=G^{2g}$ admits a pre-quantization if and only if the underlying level $l=ml_0(G)$ for some $m\in \N$, where $l_0(G)$ is given in Table \ref{table:main}.
\end{mainthmone}

\setcounter{chapter}{5}

\begin{table}[!h] 
\centering 
\begin{tabular}{|c||c|c|c|c|c|c|c|c|}
\hline
\multirow{2}{*}{$G$} & $PU(n)$ & $SU(n)/\mathbb{Z}_k$ & $PSp(n)$ & $SO(n)$ & $PO(2n)$ & $Ss(4n)$ & $PE_6$ & $PE_7$ \\
  & $n\geq 2$  &$n\geq 2$ &$n\geq 1$ &$n\geq 7$ &$n\geq 4$ &$n\geq 2$ & &\\
\hline
\multirow{2}{*}{$l_0(G)$} & \multirow{2}{*}{$n$} & \multirow{2}{*}{$\mathrm{ord}_k(\frac{n}{k})$} & 1, $n$ even & \multirow{2}{*}{1} & 2, $n$ even & 1, $n$ even & \multirow{2}{*}{3} & \multirow{2}{*}{2} \\
 & & & 2, $n$ odd & & 4, $n$ odd & 2, $n$ odd &  &    \\
\hline
\end{tabular}

\caption{\small{The integer $l_0(G)$. \underbar{Notation}: $\mathrm{ord}_k(x)$ denotes the order of $x$ mod $k$ in $\Z_k$.}}

\end{table}
\setcounter{chapter}{1}

Note that by the Cartan-Killing classification of compact simple Lie groups, all possible non-simply connected, compact simple Lie groups appear in Table \ref{table:main}.

As explained in Chapter \ref{chapter:quasi}, the quasi-Hamiltonian $G$-space $G^{2g}$ may be considered as the moduli space $M_G(\Sigma_1^g,b)$ of based flat $G$-bundles over a surface $\Sigma_1^g$ of genus $g$ with one boundary component.  The case of a surface with $r>1$ boundary components requires some extra care, and is discussed in Section \ref{sec:punconj}. 

An interesting curiosity is that Table \ref{table:main} appears in \cite{TL}, in a different context.  In \cite{TL}, Toledano-Laredo classifies irreducible positive energy representations of loop groups $LG$, by using an integer denoted $l_b(G)$ which happens to equal $l_0(G)$ for each $G$.  The reason for the coincidence is not yet understood.

\subsection*{Other results}

Leading up to the main result in Theorem  \ref{thm:coho-calc-equiv} are a pair of Propositions that describe how pre-quantization is compatible with symplectic quotients (Proposition \ref{prop:qeqpreq=reducedpreq}, as mentioned above), and the fusion product (Proposition \ref{prop:fusion}).  As reviewed in Chapter \ref{chapter:quasi}, the Cartesian product of two quasi-Hamiltonian $G$-spaces $M_1$ and $M_2$ is naturally a quasi-Hamiltonian $G$-space, called the \textsl{fusion product} and is denoted $M_1\circledast M_2$.  Proposition \ref{prop:fusion} shows that a pre-quantization of each factor $M_i$  induces a pre-quantization of the fusion product $M_1\circledast M_2$.  

This last result shows that it suffices to consider the pre-quantization of a fusion product ``factor by factor.''  In particular, $G^{2g}$ is the $g$-fold fusion product of the \textsl{double} $\mathbf{D}(G):=G\times G$ (see Example \ref{eg:double}), and hence it suffices to consider the case of genus $g=1$.  

As indicated earlier, the theory of quasi-Hamiltonian group actions is equivalent to Hamiltonian loop group actions with proper moment map.  This equivalence is reviewed in Chapter \ref{chapter:loopgroups}, and it is shown that, under this equivalence, the two notions of pre-quantization coincide. 
 
\section{Comments for topologists}\label{sec:fortopologists}

Since the pre-quantization of a quasi-Hamiltonian group action is defined as an integral lift of a certain relative cohomology class, it is not surprising that the results of this thesis have a strong topological flavour.  In particular, the main result that determines the obstruction to the existence of a pre-quantization of $G^{2g}$ amounts to computing the induced map $\tilde\phi^*:H^3(\tilde{G};\Z) \to H^3(G\times G;\Z)$, where $\tilde\phi:G\times G\to \tilde{G}$ is the canonical lift of $\phi:G\times G\to G$, the commutator map in the simple compact Lie group $G$, to the universal covering group $\tilde{G}$.  
The motivation for wanting to compute this map comes from symplectic geometry, where the result is cast as a condition on the underlying level of the quasi-Hamiltonian $G$-space $G^{2g}$.  

Using the Cartan-Killing classification of compact simple Lie groups, and known results regarding their cohomology (reprinted here in the Appendix), the determination of $l_0(G)$  becomes an algebra computation.  For most Lie groups, the calculation goes through without any real surprises.   However, for $G=PU(2)\cong SO(3)$ a more delicate analysis is required since the algebraic approach leaves an ambiguity.  

The resolution of this ambiguity uses more subtle techniques from homotopy theory.  Using the fact that $SO(3)$ is homeomorphic to $\mathbb{R}P^3$, the obstruction turns up in the suspension of the projective plane $\Sigma \mathbb{R}P^2$.  The resolution of this algebraic ambiguity uses some results regarding the homotopy group $\pi_3(\Sigma\mathbb{R}P^2)$, described by Wu \cite{W}.

\section{Comments for symplectic geometers} \label{sec:forsymplectics}

In this thesis, pre-quantization is defined as an integral lift of a certain cohomology class.  As illustrated by the calculations of Chapter \ref{chapter:calc}, this point of view helps cast the obstruction to the existence of a pre-quantization in terms of an algebraic problem.  Although very useful, this point of view pushes the differential geometric picture aside.

Recall that in the framework of Hamiltonian $G$-actions on a symplectic manifold $(X,\omega)$, a pre-quantization may also be defined as a $G$-equivariant principal $U(1)$-bundle whose equivariant curvature class is $[\omega_G]$ in $H^2_G(X;\R)$.  The classification of $G$-equivariant principal $U(1)$-bundles over a space $X$ by $H^2_G(X;\Z)$ helps to show that the two approaches (i.e. integral lifts of $[\omega_G]$ or equivariant $U(1)$-bundles with equivariant curvature $\omega_G$) are equivalent \cite{GGK}.  

In the quasi-Hamiltonian setting, there is an analogous bundle-theoretic approach to pre-quantization due to Shahbazi \cite{Sh}.  What is required is a geometric understanding of the relative cohomology group $H^3_G(\phi;\Z)$ (where $\phi$ is a smooth map), which is shown to classify \textsl{relative gerbes} \cite{Sh2}.  Pre-quantization may also be studied in the language of \textsl{pre-quasi-symplectic groupoids}---see \cite{LX} for such a treatment. 

Although this will not be discussed in this work, one may continue beyond pre-quantization and study the quantization of quasi-Hamiltonian group actions.  In this context, quantization has only recently been defined using twisted equivariant K-theory (see \cite{AM} and \cite{M-quant}).

\section{Outline of thesis}

The contents of this thesis are organized into two parts. The first part reviews some of the techniques and tools used from homotopy theory, and collects some facts about Lie groups that appear throughout.  The second part is the main body of the thesis, introducing the setting of quasi-Hamiltonian group actions, and containing the study of pre-quantization in the quasi-Hamiltonian framework.  

The following is a Chapter-by-Chapter summary/commentary of the contents of this thesis.
\\

\noindent\textbf{Chapter 2}.  This chapter reviews the homotopy theoretic methods that are used in this work.  It assumes some familiarity with algebraic topology, and principal bundles; however, it is meant to be accessible and does not assume any prior experience with homotopy theory.  References are given should the reader desire further study. 
\\

\noindent\textbf{Chapter 3}.  Various facts about Lie groups are collected here.  Besides being a convenient reference, and establishing notation that is used later on, this chapter discusses $H^3(G)$ which figures prominently throughout this work.  It also reviews some Lie theory that is necessary to understand Section \ref{sec:punconj} in Chapter \ref{chapter:prequantization}.
\\

\noindent\textbf{Chapter 4}.  This is the first chapter in the main body of the thesis.  It reviews the definition of quasi-Hamiltonian group actions, and some fundamental properties and examples.  In particular, the moduli space of flat $G$-bundles is reviewed extensively as it is the main example whose pre-quantization is being analyzed in Chapter \ref{chapter:prequantization}. Also, a pullback construction of quasi-Hamiltonian $G$-spaces is reviewed.
\\

\noindent\textbf{Chapter 5}. This chapter defines pre-quantization of quasi-Hamiltonian group actions, and establishes some compatibility properties with regards to symplectic reduction and fusion product.  The pre-quantization of the moduli space of flat $G$-bundles is studied, and the main theorem (whose proof uses the calculations appearing in Chapter \ref{chapter:calc}) is stated.  Some partial results in the case where the underlying surface has $r>1$ boundary components are also obtained.
\\

\noindent\textbf{Chapter 6}. The calculations that support Theorem \ref{thm:coho-calc-equiv} appear in this chapter.  By far the most interesting case is when the underlying group $G=PU(n)$ or $SU(n)/\Z_k$, where $k$ divides $n$ (see Section \ref{sec:interestingcase}).  Consequently, this section contains the most detail, as the techniques used there appear in other cases as well.  Of particular interest is the case $G=PU(2)$, as the algebraic methods used in other cases do not completely solve the problem and homotopy theoretic methods are applied.  
\\

\noindent\textbf{Chapter 7}. The final chapter of this thesis discusses the correspondence between Hamiltonian loop group actions and quasi-Hamiltonian group actions.  In particular, each framework admits a notion of pre-quantization, and it is shown that, under the aforementioned correspondence, these notions coincide.

\part{Background}
\chapter{Elementary Homotopy Theory}

This chapter recalls some elementary topics used throughout the thesis, as well as some tools and methods from homotopy theory that appear in Chapters \ref{chapter:liegroups} and \ref{chapter:calc}.  Most (if not all) of the material is well established in the literature, and is included here for the convenience of the reader.  More thorough treatments can be found in \cite{S}, \cite{May}, \cite{Ma}, \cite{Mc}, and \cite{Weibel}, among others.

\section{Fibrations and cofibrations} \label{sec:fibcofib}

The homotopy theoretical techniques used in this thesis can be loosely described as manipulating maps of spaces that are well behaved.  Here, ``well behaved'' means that the maps facilitate the capture of certain topological information, such as cohomology and homotopy groups.  A homotopy equivalence, for example, is a very well behaved map in this respect, and there are two other kinds of maps that are necessary to do homotopy theory: fibrations, and cofibrations. 

Note that all topological spaces in this Chapter will be assumed to have the homotopy type of a connected $CW$-complex. 

\subsection*{Fibrations}

Recall that a surjective map $\pi:E\to B$ is called a \emph{fibration} if it satisfies the \emph{homotopy lifting property} for any space $Y$;
\begin{diagram}[h=3em,w=3em]
Y  & \rTo^{h} & E \\
\dTo^{\iota_0} & \ruDashto<{\tilde{H}} & \dTo>\pi \\
Y\times I & \rTo^H& B \\
\end{diagram}
That is,  given a homotopy $H:Y\times I \to B$, and a map $h:Y\to E$, there exists a lift $\tilde{H}:Y\times I \to E$ such that the two triangles in the above diagram commute, where $\iota_0$ includes $Y$ as $Y\times 0$.

Important examples of fibrations are covering projections, and more generally principal $G$-bundles, where $G$ is a topological group.  

Every map (that is surjective on path components) is a fibration up to homotopy.  That is, given a map $f: X\to Y$, it can be factored as a composition,
$$ 
X \toby{\approx} X' \toby{f'} Y
$$
of a homotopy equivalence $X\toby{\approx} X'$ and a fibration $f':X' \to Y$.  Indeed, let $X'$ be  the \emph{mapping path space}, $P^f=\{(x,\omega)\in X\times Y^I|\, \omega(0)=f(x) \}$.  The homotopy equivalence $X\to P^f$ given by $x\mapsto (x,\mathrm{const}_{f(x)})$, and the fibration $f':P^f \to Y$ sending $(x,\omega) \mapsto \omega(1)$ provide the required factorization of $f$.  Therefore, in order to study topological invariants such as cohomology groups, nothing is lost by studying the fibration $f':X'\to Y$ in place of $f:X\to Y$.

\subsection*{Homotopy fibres}

Recall that for principal $G$-bundles $\pi:E\to B$, the fibre $\pi^{-1}(b)$ over any point $b$ in $B$ is homeomorphic to $G$.  For a fibration  $\pi:E\to B$, with connected base $B$, the fibres $F_b=\pi^{-1}(b)$ are only homotopy equivalent.  Nevertheless, the \emph{fibre} of a fibration is defined to be the fibre $F_b$ over some chosen base point $b$ in the base $B$.  

For an arbitrary map $f:X\to Y$,  the \emph{homotopy fibre} $F^f$ of $f$ is the fibre of the fibration $f'$ where $f$ has been factored as $f:X\toby{\approx}P^f\toby{f'} Y$.  This exhibits $F^f$ as the pullback of $f$ with the fibration $p:PY\to Y$, where $PY=\{\omega \in Y^I \, | \, \omega(1)=*\}$ is the (contractible) \emph{path space} of $Y$, and $p(\omega)=\omega(0)$.

The homotopy type of $F^f$ is uniquely determined by the homotopy class of $f$.  
In fact, the homotopy fibre of a map $f:X\to Y$ may also be obtained by forming the pullback of $f$ with any fibration $p:E\to Y$, where $E$ is contractible.  
To see this, factor $f$ as the composition $f:X\toby{\approx}P^f \toby{f'} Y$, and consider the following diagram of successive pullbacks, which defines $F$.
\begin{diagram}
F & \rTo& F' & \rTo^k & E \\
\dTo & & \dTo &  & \dTo>p \\
X & \rTo^\approx & P^f & \rTo^{f'} & Y \\
\end{diagram}
Since $f'$ is a fibration, the map $k:F'\to E$ is also a fibration.  And since the diagram is a pullback, the fibres of $k$ and $f'$ are equal.  Note that $F'$ is homotopy equivalent to the fibre of $k$ because the base $E$ is contractible. That is, $F'\approx F^f$. Since $F$ is the pullback of $F'$ over a homotopy equivalence, $F\approx F'$, and hence $F\approx F^f$.

\subsection*{Cofibrations}

The other kind of well behaved map is a cofibration. Recall that an inclusion $j:A \hookrightarrow X$ is called a \emph{cofibration} if it satisfies the \emph{homotopy extension property} for any space $Y$;
\begin{diagram}
A & & \rTo^{j} & & X \\
    &  &                     & \ldTo^{\iota_0} & \\
\dTo<{\iota_0} & & X\times I &  &  \dTo>h \\
   & \ruTo<{j\times \mathrm{id}} &  &   \rdDashto>{H'} & \\
 A\times I &  & \rTo^H  &  & Y\\
\end{diagram}
that is, if given a homotopy $H:A\times I \to Y$, and a map $h:X\to Y$ such that $h\circ j=H\circ \iota_0$, there exists an extension $H':X\times I \to Y$ of $H$ such that $H'\circ \iota_0=h$.
The inclusion of a nonempty subcomplex $A\hookrightarrow X$ is an example of a cofibration.  

Every map $X\to Y$ is also a cofibration up to homotopy.  That is, any map $f:X\to Y$ can be factored as a composition
$$
X \toby{j} Y' \toby{\approx} Y
$$
of a cofibration $j:X\to Y'$ and a homotopy equivalence $Y'\to Y$. Indeed, let $Y'$ be the \emph{mapping cylinder} $M_f$, which is the quotient space 
 $$ 
 M_f=\left( (X\times I) \coprod Y \right) \big/ \sim \qquad\text{where }(x,0)\sim f(x). 
$$
The cofibration $X\hookrightarrow M_f$ that includes $X$ as $X\times 1$, and the homotopy equivalence $r:M_f\to Y$ defined by $r(y)=y$ for $y\in Y$ and 
$r(x,t)= f(x)$ otherwise provide the required factorization of $f$.

\subsection*{Homotopy cofibres}

The \emph{cofibre} $X/A$ of a cofibration $A\hookrightarrow X$ is the quotient space obtained from $X$ by identifying $A$ to a single point.  For a map $f:X\to Y$, the cofibre $C_f$ of $X\to M_f$ is called the \emph{mapping cone} of $f$.  More generally, the cofibre of the cofibration $f':X\to Y'$ in a factorization of $f$ as above is called the \emph{homotopy cofibre} of $f$, well defined up to homotopy.  Observe that in the special case where $f:S^{n-1} \to Y$, the mapping cone $C_f$ is simply $Y\cup_f e^{n}$, the space obtained  by attaching an $n$-cell $e^{n}$ to $Y$ with attaching map $f$.

\section{Loops, suspensions, and adjoints} \label{sec:loopsetc}

Recall the following elementary constructions for spaces $X$ and $Y$ with base points.  

\begin{itemize}
\item[$\cdot$] The \emph{wedge product} $X\vee Y=\{(x,y)\in X\times Y | x=* ,\text{ or } y=*\}$.
\item[$\cdot$] The \emph{smash product} $X\wedge Y=(X\times Y)/(X\vee Y)$.
\item[$\cdot$] The (\emph{reduced}) \emph{suspension} of $X$ is $\Sigma X = X\wedge S^1$, which may be represented by points $(x,t)\in X\times I$ with the understanding that points of the form $(x,0)$, $(x,1)$, and $(*,t)$ all represent the same point, the base point of $\Sigma X$. 
\item[$\cdot$] The (\emph{based}) \emph {mapping space} $\mathrm{Map}_*(X,Y)$ is the set of based continuous maps $X\to Y$ (with compact-open topology).
\item[$\cdot$] The (\emph{based}) \emph{loop space} $\Omega X=\mathrm{Map}_*(S^1,X)$.
\item[$\cdot$] The set of \emph{homotopy classes of maps} $[X,Y]=\pi_0\mathrm{Map}_*(X,Y)$.
\end{itemize}

The operations ``suspension'' $\Sigma$ and ``loop'' $\Omega$ are functors. That is, for any map $f:X\to Y$ there are maps $\Sigma f:\Sigma X\to \Sigma Y$ 
and $\Omega f: \Omega X \to \Omega Y$ defined by
$$
\Sigma f (x,t)=(f(x),t), \quad \text{and} \quad \Omega f ( \gamma)(t)=f(\gamma(t)).
$$
Since a map $h:\Sigma X\to Y$ is of the form  $h(x,t)$, it also describes a map $\mathrm{ad} h:X\to \Omega Y$, which assigns $x\mapsto h(x,t)$,  describing a loop in $Y$ as $t$ varies.  This correspondence is bijective,
$$
\mathrm{ad}:[\Sigma X,Y]\toby{\cong}[X,\Omega Y],
$$
and the suspension and loop functors, $\Sigma(\quad)$ and $\Omega(\quad)$ are adjoint to each other.  

Let $\alpha:X\to \Omega \Sigma X$ denote the adjoint of the identity map $\Sigma X\to \Sigma X$, so that $\alpha(x)\in \Omega\Sigma X$ is the loop $\alpha(x)(t)=(x,t)$.  The the adjoint $\mathrm{ad} h: X\to \Omega Y$ of any map $h:\Sigma X\to Y$ may be computed by the composition $\mathrm{ad} h:X\toby{\alpha} \Omega\Sigma X \toby{\Omega h}\Omega Y$.  Indeed, $\Omega h (\alpha(x))(t)=h(\alpha(x)(t))=h(x,t)$, which shows that the loop described by the adjoint of $h$ is the loop $\Omega h(\alpha(x))$. Similarly, if $\beta: \Sigma \Omega X\to X$ denotes the adjoint of the identity map $\Omega X\to \Omega X$, then the adjoint  $\mathrm{ad}^{-1} g$ of a map $g:X\to \Omega Y$ may be computed by the composition $\mathrm{ad}^{-1} g:\Sigma X \toby{\Sigma g} \Sigma\Omega X \toby{\beta} Y$.

Note that the adjunction of $\Sigma$ and $\Omega$ yields the identification $[S^{n+1},X]=\pi_{n+1}(X)\cong \pi_n(\Omega X)=[S^n,\Omega X]$, since $\Sigma S^n=S^{n+1}$.

\section{Exact sequences}

Fibrations and cofibrations can be used to generate long exact sequences in homotopy groups and cohomology groups, respectively,  by using fibration and cofibration sequences.

\subsection*{Fibration sequences}

Recall that a principal $G$-bundle $E\to X$ is classified by a map $f:X\to BG$ (see Section \ref{sec:classifyingspaces}). Therefore, up to bundle isomorphism,  $E$ fits in the pullback diagram
\begin{diagram}
E & \rTo & EG \\
\dTo & & \dTo \\
X & \rTo^f & BG \\
\end{diagram}
where $EG$ is contractible.  This identifies $E=F^f$, the homotopy fibre of the classifying map $f:X\to BG$.  Also, recall that the fibre of the principal bundle $E\to X$ is $G$.  Therefore in this setting, the homotopy fibre of $F^f\to X$ is $G\approx \Omega(BG)$.  This generalizes to homotopy fibres of arbitrary maps  $f: X\to B$ as follows.

Let $f:X\to B$ be a map, and recall that the homotopy fibre $F^f$ may be obtained by taking the pullback of $f$ with path space fibration of $B$:
\begin{diagram}
F^f &\rTo& PB \\
\dTo & & \dTo \\
X & \rTo^f & B \\
\end{diagram}  
Since $PB\to B$ is a fibration, the pullback $F^f \to X$ is also fibration, and these maps have equal fibres, namely $\Omega B$.  In other words, the homotopy fibre of $F^f \to X$ is $\Omega B$. 

This procedure of taking homotopy fibres may be continued indefinitely.  That is, starting with a map $f:X\to B$, with homotopy fibre $F^f$, the previous paragraph shows that the homotopy fibre of $F^f\to X$ is $\Omega B$. Therefore, the sequence of maps $\Omega B\to F^f \to X$ is (up to homotopy) the inclusion of a fibre followed by a fibration, in the same sense that the sequence $F^f \to X \to B$ is the inclusion of a fibre followed by a fibration.

 Repeating this procedure with the map $F^f \to X$, with homotopy fibre $\Omega B$, shows that the homotopy fibre of $\Omega B\to F^f$ is $\Omega X$.  Continuing, there is a sequence 
\begin{equation} \label{eqn:fibseq}
 \cdots\to \Omega^2 X \to \Omega^2B \to \Omega F^f \to \Omega X \to \Omega B \to F^f \to X \to B
\end{equation}
in which each pair of successive maps is (up to homotopy) the inclusion of a fibre followed by a fibration $F\hookrightarrow E\toby{\mathrm{fib.}} B$.  Such a sequence of maps is called a \emph{fibration sequence}. 

The  fibration sequence (\ref{eqn:fibseq}) may be used to produce the long exact sequence of homotopy groups associated to a fibration sequence $F\to X\to B$.  The following proposition follows directly from the homotopy lifting property enjoyed by fibrations.

\begin{prop}\label{prop:lift} Let $F\toby{j}X\toby\pi B$ be a fibration sequence. For any space $Y$, the sequence of pointed sets 
$$
[Y,F]\toby{j_\sharp} [Y,X]\toby{\pi_\sharp} [Y,B]
$$ is exact.  That is, if $f:Y\to X$ is any map such that  $\pi\circ f$ is null homotopic, then there exists a lift $\tilde{f}:Y\to F$ such that $f=\tilde{f}\circ j$.
\end{prop}

Applying this proposition repeatedly to the fibration sequence (\ref{eqn:fibseq}), and using the adjunction of $\Sigma$ and $\Omega$, as in the previous section gives the familiar long exact sequence of homotopy groups 
$$
\cdots \pi_n(F) \to \pi_n(X) \to \pi_n(B) \to \pi_{n-1}(F) \to \cdots
$$

\subsection*{Cofibration sequences}

One of the defining characteristics of singular cohomology $H^*(-)$ is that to a pair of spaces $(X,A)$, where $A$ is a subspace of $X$, there is a long exact sequence of cohomology groups (with any coefficient ring):
$$
\cdots \to H^p(X) \to H^p(A) \to H^{p+1}(X,A) \to H^{p+1}(X) \to \cdots
$$ 
Furthermore, if the inclusion $A\hookrightarrow X$ is a cofibration and $A$ is closed in $X$, then $\tilde{H}^*(X,A)\cong \tilde{H}^*(X/A)$.  In other words, a cofibration $A\to X$ gives rise to a long exact sequence in cohomology:
$$
\cdots \to \tilde{H}^p(X) \to \tilde{H}^p(A) \to \tilde{H}^{p+1}(X/A) \to \tilde{H}^{p+1}(X) \to \cdots
$$ 

 Long exact sequences in cohomology may also be obtained from an arbitrary map $f:X \to Y$ by replacing $f$ by a cofibration, as discussed in Section \ref{sec:fibcofib}. Using the mapping cylinder, for example, the cofibration $X\to M_f $ yields the long exact sequence:
\begin{equation} \label{eqn:cohomseq}
\cdots \to H^p(Y) \to H^p(X) \to H^{p+1}(C_f) \to H^{p+1}(Y) \to \cdots
\end{equation}
where the homotopy equivalence $M_f\approx Y$ has been used.

These exact sequences may also be described using cofibration sequences, in a way analogous to homotopy groups and fibration sequences.  Let $f:X\to Y$ be a map, and consider the homotopy cofibre $C_f$, the mapping cone of $f$.  Notice that $Y\to C_f$ is a cofibration, and that the cofibre $Y/C_f\approx \Sigma X$.  Therefore, the homotopy cofibre of the induced map into the cofibre $Y\to C_f$  is $\Sigma X$. This procedure may be repeated indefinitely to yield the Barrat-Puppe sequence,
\begin{equation} \label{eqn:cofibseq}
X \to Y \to C_f \to \Sigma X \to \Sigma Y \to \Sigma C_f \to \Sigma^2 X \to \Sigma^2 Y \to \cdots
\end{equation}
in which every pair of successive maps is (up to homotopy) a cofibration followed by the induced map into the cofibre $A\toby{\mathrm{cofib.}} X \to C$.  Such a sequence is called a \emph{cofibration sequence}.

The following proposition follows immediately from the homotopy extension property enjoyed by cofibrations.

\begin{prop}\label{prop:extend} Let $X\toby{f}Y\toby{q} C_f$ be a cofibration sequence. For any space $K$, the sequence of pointed sets 
$$
[C_f,K]\toby{q^*} [Y,K]\toby{f^*} [X,K]
$$ is exact.  That is, if $h:Y\to K$ is any map whose restriction $h \circ f$ to $X$ is null homotopic, then there exists an extension $\bar{h}:C_f\to K$ such that $h=\bar{h}\circ q$.
\end{prop}

Applying this proposition repeatedly to the Barrat-Puppe sequence (\ref{eqn:cofibseq}), yields the long exact sequence
$$
\cdots  [\Sigma^2 X,K]\to [\Sigma C_f,K] \to [\Sigma Y,K]\to [\Sigma X,K] \to [C_f,K]\to[Y,K]\to [X,K] \cdots
$$

To complete the analogy with fibration sequences and homotopy groups, this sequence of pointed sets may be used to yield long exact sequences in cohomology. Using the fact that cohomology is representable by Eilenberg-MacLane spaces $H^p(X;R)\cong[X,K(p,R)]$, applying the above long exact sequence to $K=K(p,R)$, and using the suspension isomorphism $H^{p+1}(\Sigma X;R) \cong H^p(X;R)$ gives the familiar long exact sequence (\ref{eqn:cohomseq}). (This description of the long exact sequence in cohomology is not used in this work, and the reader is left to consult \cite{S}, or \cite{May} for details.)

\subsection*{The algebraic mapping cone}

The above discussion of long exact sequences in cohomology may also be modeled on the level of cochain complexes, using the algebraic mapping cone, which is a model for $S^*(C_f)$ that is built from the cochain complexes $S^*(Y)$ and $S^*(X)$.  

Let $\varphi:A^* \to B^*$ be a map of cochain complexes.  The \emph{algebraic mapping cone of $\varphi$} is the cochain complex $C^*(\varphi)$ defined by:
$$
C^p(\varphi) := B^{p-1}\oplus A^p
$$ 
with differential  $d_\varphi(\beta,\alpha)=(d\beta+\varphi \alpha, -d\alpha)$.  The homology of this cochain complex fits in a long exact sequence
$$
\cdots \to H^p(B) \to H^p(C) \to H^{p+1}(C(\varphi)) \to H^{p+1}(B) \to \cdots
$$
where the connecting map is the map induced by $\varphi$.  (Details may be found in \cite{Ma}.)

Applying this construction to the induced map of $f:X\to Y$ on cochain complexes $f^*:S^*(Y)\to S^*(X)$ gives the \emph{relative cohomology} group $H^*(f):=H(C^*(f^*),d_{f^*})$, which fits in the long exact sequence: 
$$
\cdots \to H^p(Y) \to H^p(X) \to H^{p+1}(f) \to H^{p+1}(Y) \to \cdots
$$
which will be used extensively in later chapters.

\section{Samelson and Whitehead products} \label{sec:wheadprod}

An \emph{$H$-space} is a based space $G$ equipped with a map $\mu:G\times G \to G$ such that $\mu\circ \iota_1\approx \mathrm{id}_G \approx \mu\circ\iota_2$, where $\iota_1(x)=(x,*)$ and $\iota_2(x)=(*,x)$.  The map $\mu$ is referred to as the multiplication map, since it a homotopy theoretic generalization of a topological group's multiplication map. A map $f:G\to H$ of $H$-spaces is an \emph{$H$-map} if it is a homomorphism up to homotopy; that is, if $\mu_H \circ (f\times f)\approx f\circ \mu_G$.

An \emph{H-group} is an $H$-space that is \emph{homotopy associative} (i.e.~$\mu\circ(\mathrm{id}_G \times \mu) \approx \mu \circ (\mu\times \mathrm{id}_G)$), and has a \emph{homotopy inverse} (a map $c:G\to G$ satisfying $\mu\circ(\mathrm{id}_G\times c) \circ \Delta \approx * \approx \mu\circ(c\times \mathrm{id}_G) \circ \Delta$, where $\Delta$ is the diagonal map).  Clearly, topological groups are $H$-spaces. As with topological groups, $\mu(x,y)$ may be written as $xy$, while $c(x)$ may be written as $x^{-1}$.
\\

Suppose $G$ is an $H$-group, and let $f:X\to G$ and $g:Y\to G$ be continuous maps.  The map $(x,y) \mapsto f(x)g(y)f(x)^{-1}g(y)^{-1}$ is null homotopic when restricted to $X\vee Y$.  Therefore, by Proposition \ref{prop:extend} there is an induced map 
$$
\langle f, g\rangle: X\wedge Y\to G,
$$
 called the \emph{Samelson product} of $f$ and $g$.  To see that $\langle f,g\rangle$ is well defined (up to homotopy), note that the sequence $[X\wedge Y,G] \to [X\times Y,G] \to [X\vee Y,G]$ splits, because any map $f_1\vee f_2: X\vee Y\to G$ can be extended to $X\times Y$ by  $\mu\circ(f_1\times f_2):X\times Y \to G\times G \to G$.  Therefore, the homotopy class of $\langle f,g\rangle$ depends only on the homotopy classes of $f$ and $g$.

For  $f:\Sigma X\to B$, and $g:\Sigma Y \to B$, the \emph{Whitehead product} of $f$ and $g$,  
$$
[f,g]:\Sigma(X\wedge Y) \to B, \quad [f,g]=\mathrm{ad}^{-1}\,\langle \mathrm{ad}\, f,\mathrm{ad}\,g\rangle
$$ 
where  recall $\mathrm{ad} \,h: A\to \Omega B$ denotes the adjoint of $h:\Sigma A \to B$. (Note that the loop space $\Omega B$ is an $H$-group, with multiplication defined by concatenation of paths.)

In the special case where $X$ and $Y$ are spheres, the Whitehead product 
induces a product on homotopy groups.  Indeed, if $f:S^n=\Sigma S^{n-1}\to B$ and $g:S^m=\Sigma S^{m-1}\to B$ represent elements in $\pi_n(B)$ and $\pi_m(B)$ respectively, then their Whitehead product $[f,g]:S^{n+m-1}=\Sigma(S^{n-1}\wedge S^{m-1}) \to B$ represents an element in $\pi_{n+m-1}(B)$.  For example, the Whitehead product of the identity map $\iota:S^2\to S^2$ with itself yields a map $[\iota,\iota]:S^3\to S^2$, which represents some element in $\pi_3(S^2)$.  

In fact, from the long exact sequence of homotopy groups associated to the Hopf fibration sequence $S^1\to S^3 \toby{\eta} S^2$, it is easy to see that $\pi_3(S^2)\cong \pi_3(S^3)\cong \Z$, and that $\eta$ represents a generator.  It will be of crucial importance in Chapter \ref{chapter:calc} to recognize that $[\iota,\iota]$ is homotopy equivalent to $2\eta$ in $\pi_3(S^2)$, a well known fact that is reviewed next.

\begin{eg} $[\iota,\iota]=2\eta$ \label{eg:whiteheadistwicehopf}
\end{eg}

This exercise uses the Hopf invariant $h(f)$ of a map $f:S^{2n-1}\to S^n$ for $n\geq 2$.  Recall that such a map may be viewed as describing the attaching map of a $2n$ cell to $S^n$, producing a $CW$-complex $X$ with exactly two cells: one in dimension $n$, and one in dimension $2n$.  Since there are no cells in adjacent dimensions, the cohomology $X$ is easy to calculate, namely $\tilde{H}^p(X;\Z)\cong \Z$ when $p=n$ or $p=2n$, and $0$ otherwise.  If $x\in H^n(X;\Z)$ and $y\in H^{2n}(X;\Z)$ denote generators, then the \emph{Hopf invariant} $h(f)$ is the integer (defined up to sign) satisfying $x^2=h(f)y$.  It is an easy exercise to check that the Hopf invariant defines a homomorphism $h:\pi_{2n-1}(X)\to \Z$, that depends only on the homotopy type of $f$ \cite{Ma}.

Recall that the complex projective plane $\mathbb{C}P^2$ is obtained from $\mathbb{C}P^1=S^2$ by attaching a $4$-cell, with attaching map $\eta:S^3\to S^2$, the  Hopf fibration.  Therefore,  $h(\eta)=1$, since $H^*(\mathbb{C}P^2;\Z)\cong \Z[x]/(x^3)$, and the homomorphism $h:\pi_3(S^2)\to \Z$ is an isomorphism able to detect the homotopy class of a given map $S^3\to S^2$.  To check that $[\iota,\iota]=2\eta$, it suffices to verify that the Hopf invariant $h([\iota,\iota])=2$.

The Whitehead product $[\iota,\iota]$ may also be described in the following way (see \cite{S}).  The product $S^2 \times S^2$ has a cell decomposition with two cells in dimension 2 and one cell in dimension 4.  More precisely, it can be obtained by attaching a 4-cell to the 3-skeleton $S^2\vee S^2$, via an attaching map $w:S^3\to S^2 \vee S^2$.  The composition of $w$ with the fold map $f:S^2\vee S^2 \to S^2$ (which is the identity on each copy of $S^2$) is homotopy equivalent to the Whitehead product $[\iota,\iota]$.

Continuing with the computation of the Hopf invariant $h([\iota,\iota])$, let $J=C_{[\iota,\iota]}$ be the homotopy cofibre of the Whitehead product, and let $x\in H^2(J;\Z)$, and $y\in H^4(J;\Z)$ be generators.  Consider the following homotopy commutative diagram of cofibration sequences.
\begin{diagram}
S^3 &\rTo^w & S^2\vee S^2 &\rTo & S^2 \times S^2   \\
\dEquals & & \dTo>f & & \dTo>g   \\
S^3 & \rTo^{[\iota,\iota]} & S^2 & \rTo & J   \\
\end{diagram} 
Recall that $H^*(S^2\times S^2;\Z)$ is generated by $a$ and $b$ in $H^2(S^2\times S^2;\Z)\cong H^2(S^2\vee S^2;\Z)$, and $ab\in H^4(S^2\times S^2;\Z)$, with $a^2=b^2=0$.  Applying $H^*(\quad;\Z)$ to the above diagram gives a diagram with exact rows,
\begin{diagram}
H^{q-1}(S^3) & \rTo & H^q(J)& \rTo & H^q(S^2) & \rTo & H^q(S^3)  \\
\dEquals & & \dTo>{g^*} & & \dTo>{f^*} & & \dEquals \\
H^{q-1}(S^3) &\rTo & H^q(S^2\times S^2) &\rTo & H^q(S^2 \vee S^2) & \rTo & H^q(S^3)   \\
\end{diagram}
showing that the induced map $g^*:H^4(J;\Z) \to H^4(S^2 \times S^2;\Z)$ is an isomorphism, and $g^*(y)=ab$ (up to a sign).  The exactness of the top row  shows that $x\in H^2(J;\Z)$ may be identified with the generator of $H^2(S^2;\Z)$.  And since $f^*(x)=a+b$, $g^*(x^2)=(g^*(x))^2=(a+b)^2=2ab$, which shows that $h([\iota,\iota])=2$. \hfill $\qed$

\section{De Rham theory} \label{sec:classifyingspaces}

 Let $S^*(X;R)$ denote the singular cochain complex of  a topological space $X$ with coefficients in the commutative ring $R$.  If in addition $X$ is a smooth manifold, recall that the subcomplex $S^*_{sm}(X;R)$ of smooth cochains on $X$ is chain homotopy equivalent to $S^*(X;R)$ for any coefficient ring $R$ (see Theorem 2.1 in Appendix A of \cite{Ma}).

Let $(\Omega^*(X),d)$ be the de Rham cochain complex of differential forms on a smooth oriented manifold $X$, with differential $d$, the exterior derivative.   Integration of differential forms defines a chain map $\Omega^*(X)\to S_{sm}^*(X;\mathbb{R})$, and the de Rham theorem states that this map is a chain homotopy equivalence.  This yields a natural isomorphism $H(\Omega^*(X),d)\cong H^*(X;\mathbb{R})$, which throughout this work will serve as an identification.  (See \cite{Ma} for more details.)

For a $G$-space $X$, where $G$ is a compact, connected Lie group, there is a $G$-equivariant version of the de Rham theorem.  It is defined using the Borel construction, which will be reviewed next.

To begin, recall that for a topological group $G$ there is a contractible space $EG$ on which $G$ acts freely.  The orbit space $BG=EG/G$ is called the \emph{classifying space} of $G$, and is unique up to homotopy type. The first construction for any topological group $G$ is due to Milnor \cite{Milnor2}.  (Constructions of classifying spaces also exist for associative $H$-spaces, as in  \cite{Milgram} for example.)

An important feature of classifying spaces is that they classify principal bundles.  That is, there is a one-to-one correspondence  between isomorphism classes of principal $G$-bundles over $X$ and homotopy classes of maps $[X,BG]$.  Given a principal $G$-bundle $E\to X$, the corresponding map $\rho:X\to BG$ is called the \emph{classifying map} of the bundle.  Furthermore, $E\to X$ is isomorphic to the pullback of $EG\to BG$ along $\rho$.  

Let $X_G$ denote the \emph{homotopy quotient}---the orbit space $(X\times EG)/G$ under the diagonal $G$-action.  Then the $G$-equivariant cohomology of $X$ with coefficients in $R$ is
$$
H^*_G(X; R):=H^*(X_G; R).
$$  It is well known that this definition does not depend on the choice of $EG$ (see \cite{GS}).  

Observe that the pullback diagram 
\begin{diagram}
X \times EG & \rTo & EG \\
\dTo & & \dTo \\
X_G & \rTo^\rho & BG \\
\end{diagram}
identifies $X$ as the homotopy fibre of the classifying map $\rho:X_G\to BG$.  The fibre bundle $X_G\to BG$ is sometimes called the \emph{Borel construction}.

As with the de Rham complex above, the $G$-equivariant cohomology of $X$ with real coefficients may be modeled at the chain level geometrically.  First, let $S\mathfrak{g}^*$ denote the symmetric algebra on $\mathfrak{g}^*$, the dual of the Lie algebra $\mathfrak{g}$ of $G$, with $G$-action determined by the co-adjoint action of $G$ on $\mathfrak{g}^*$.  The Cartan complex $(\Omega_G^*(X), d_G)$ of $G$-equivariant differential forms on $X$ is the $G$-invariant subcomplex
$$
\Omega_G^*(X):=(S\mathfrak{g}^*\otimes \Omega^*(X))^G, \qquad
d_G=1\otimes d - \sum \xi_i \otimes \iota_{\xi_i^\sharp}
$$
where $\xi_1, \ldots, \xi_{\mathrm{dim}G}$ is a choice of basis for $\mathfrak{g}$,  $\zeta^\sharp$ denotes the generating vector field of $\zeta\in \mathfrak{g}$ on $X$ induced by the $G$-action, and $\iota_v$ denotes contraction by the vector field $v$.  Note that elements $\alpha\otimes \beta \in \Omega_G^*(X)$ are assigned grading $2|\alpha| + |\beta|$.

An element in $\Omega_G^*(X)$ may be viewed as a $G$-equivariant polynomial $\alpha:\mathfrak{g}\to \Omega^*(X)$, in which case the differential $d_G$ may be described by $(d_G\alpha)(\zeta)=d(\alpha(\zeta)) - \iota_\zeta\alpha(\zeta)$. In any case, the equivariant de Rham theorem states that there is a canonical isomorphism $H(\Omega_G^*(X),d_G)\cong H_G^*(X;\mathbb{R})$, which will also serve as an identification throughout this work. (See \cite{GS} for more details.)
 
The natural map $ev_0:\Omega_G^*(X) \to \Omega^*(X)$ that sends $\alpha:\mathfrak{g} \to \Omega^*(X)$ to $\alpha(0)$ induces a map $H_G^*(X;\R)\to H^*(X;\R)$, which may be regarded as the induced map on cohomology of $X\to X_G$.  A differential form $\omega=ev_0(\omega_G)$ is said to have an \emph{equivariant extension} $\omega_G$.  Notice that in order to have an equivariant extension, $\omega$ must be $G$-invariant.

\section{Spectral sequences}\label{sec:spectralsequences}

There are two kinds of spectral sequences used in this thesis: the Serre spectral sequence, and the mod $p$ Bockstein spectral sequence.  The main theorems about each of these will be outlined in this section.  Although it will be assumed that the reader is familiar with the basic definitions concerning spectral sequences (i.e.~filtrations, convergence, etc.), the discussion of the Bockstein spectral sequence will assume no prior exposure.  Consult \cite{Mc}, \cite{S}, or \cite{Weibel} for more thorough treatments. 
\\

First, recall how spectral sequences arise in the framework of exact couples (due to Massey \cite{Ma-exco}).  A pair $(D,E)$ of modules is an \emph{exact couple} if there is a diagram:
\begin{diagram}[height=1.5em]
D & & \\
    & \rdTo>i & \\
  \uTo<k  &	& D \\
  &\ldTo>j & \\
E & & \\
\end{diagram}
that is exact at each vertex.  The \emph{derived couple} $(D', E')$ is the exact couple obtained from $(D,E)$ by setting $D'=i(D)$, $E'=\mathrm{ker}\, d/\mathrm{im}\, d$, where $d=jk$.  The corresponding diagram for $(D', E')$
\begin{diagram}[height=1.5em]
D' & & \\
    & \rdTo>{i'} & \\
  \uTo<{k'}  &	& D' \\
  &\ldTo>{j'} & \\
E' & & \\
\end{diagram}
has maps $i'=i|_{D'}$, $j'(iy)=[j(y)]$, and $k'([x])=k(x)$, which one can show are well defined.  Clearly, this construction may be repeated ad infinitum. 

Therefore, from an exact couple $(D,E)$ arises a sequence $(E_r, d_r)$  of modules with differentials defined by 
$$
(E_1,d_1)=(E, d),\quad (E_2,d_2)=(E',d'), \quad (E_3, d_3)=(E'', d''),\quad \text{and so on}.
$$
Such a sequence of differential modules $(E_r, d_r)$ with $E_{r+1}=H(E_r,d_r)$ is called a \emph{spectral sequence}. 

Typically, a bigrading $E_*^{p,q}$ is introduced to recover the familiar form of spectral sequences, as in the Serre spectral sequence discussed below. (Note, however, that the Bockstein spectral sequence is not naturally bigraded, and in that case it is not necessary to introduce the extra bookkeeping.)  More specifically, it is assumed that the modules $D=D^{p,q}$ and $E=E^{p,q}$ are bigraded, and that the maps $i$, $j$, and $k$ have bidegrees
$$
|i|=(-1,1), \quad |j|=(0,0), \quad \text{and}\quad |k|=(1,0),
$$
which give $d$ a bidegree $|d|=(1,0)$.  The resulting spectral sequence  $(E_r^{p,q}, d_r^{p,q})$ assumes a bigrading, with differentials $d_r^{*,*}$ of bidegree $(r,1-r)$. 

The spectral sequences considered below are \emph{bounded} spectral sequences---that is, for each $n$ there are only finitely many nonzero terms of total degree $n$ in $E_1^{*,*}$.  Therefore, for each $p$ and $q$ the differentials $d_r^{p,q}: E_r^{p,q}\to E_r^{p+r,q-r+1}$ eventually become zero because there can only be finitely many nonzero terms in total degree $p+q+1$.  Similarly for the differentials $d_r^{p-r,q+r-1}:E_r^{p-r,q+r-1}\to E_r^{p,q}$, hence there is an index $s$ for which $E_r^{p,q}=E_{r+1}^{p,q}$ for $r\geq s$.  This stable value is denoted $E_\infty^{p,q}$.

A bounded spectral sequence $E_*^{p,q}$  \emph{converges} to $H^*$ if there is a finite degree preserving filtration
$$
0=F^sH^* \subset F^{s-1}H^* \subset \cdots \subset F^{t+1}H^* \subset F^tH^n=H^*
$$
and isomorphisms $E_\infty^{p,q}\cong F^pH^{p+q}/F^{p+1}H^{p+q}$.  If the spectral sequence, and the graded module $H^*$ are also algebras, then the spectral sequence may converge to $H^*$ as algebras, which means the above isomorphisms respect the induced algebra structures (see \cite{Mc}).

\subsection*{The Serre spectral sequence}

Given a fibration $E\toby{\pi} X$ with fibre  $F$, the Serre spectral sequence is an algebraic tool that relates $H^*(X)$ and $H^*(F)$ to $H^*(E)$, using any coefficient ring. To see how the spectral sequence arises,  consider the filtration of $E$ by subspaces $F^p:=\pi^{-1}(X^{(p)})$, where $X^{(p)}$ is the $p$-skeleton of the connected $CW$-complex $X$.  The inclusions $F^{p-1}\hookrightarrow F^p$ each induce long exact sequences in cohomology, 
$$
\cdots \to H^s(F^p) \to H^s(F^{p-1}) \to H^{s+1}(F^p, F^{p-1}) \to H^{s+1}(F^p) \to \cdots
$$
which can be spliced together to yield  the bigraded exact couple $(D,E)$:
$$
D^{p,q}=H^{p+q-1}(F^p), \qquad E^{p,q}=H^{p+q}(F^p, F^{p-1}).
$$
The filtration on $E$ induces a filtration on the cochain complex $S^*(E)$ by setting $F^pS^*(E)=\ker\{S^*(E)\to S^*(F^{p-1})\}$.  It follows that the resulting spectral sequence with $E_1^{p,q}=H^{p+q}(F^p, F^{p-1})$ converges to $H^*(E)$.

In his thesis, Serre \cite{Serre} computed the $E_2$-term of this spectral sequence, which is given in the following theorem. 

\begin{thmnonumber}[Serre] Let $\pi:E\to X$ be a fibration with fibre $F$, where $X$ is connected.  The $E_2$-term of the above spectral sequence is
$E_2^{p,q} \cong H^p(X; \mathcal{H}^q(F))$, and $E_*^{p,q}$ converges to $H^*(E)$ as an algebra.
\end{thmnonumber}

Here, $H^*(X; \mathcal{H}^q(F))$ denotes the cohomology of $X$ with coefficients in the local coefficient system $\mathcal{H}^q(F)$ (see \cite{S}).  When the base $X$ is simply connected, the local coefficient system is trivial, and $H^*(X;\mathcal{H}^q(F)) \cong  H^*(X; H^q(F))$.  Moreover, if in addition the underlying coefficient ring $R$ is a field, then $E_2^{p,q} \cong H^p(X; R) \otimes_R H^q(F; R)$.

\subsection*{The Cartan-Leray spectral sequence}

Let $\pi:E\to X$ be  a covering projection, and let $\Gamma=\pi_1(X)$.  Consider the classifying map (see Section \ref{sec:classifyingspaces}) $\rho: X\to B\Gamma$ of the covering projection, with homotopy fibre $E$.  As discussed above, there is a spectral sequence with $E_2^{p,q}\cong H^p(B\Gamma; \mathcal{H}^q(E))$ that converges to $H^*(X)$.  For a discrete group such as $\Gamma$, it is well known that $H^*(B\Gamma)\cong H_{alg}^*(\Gamma)$, where the right side denotes the algebraically defined cohomology of $\Gamma$.  The Cartan-Leray spectral sequence, which will be introduced shortly, is a re-working of the Serre spectral sequence under this identification.

Recall that the homology and cohomology of a discrete group $\Gamma$ with coefficients in the $\Gamma$-module $A$ (i.e.~the $\Z \Gamma$-module $A$) are defined as
$$
H^{alg}_*(\Gamma;A) = \mathrm{Tor}_*^{\Z\Gamma}(\Z,A), \quad \text{and} \quad
H_{alg}^*(\Gamma;A) = \mathrm{Ext}^*_{\Z\Gamma}(\Z,A) .
$$
For a covering projection $E\to X$, the fundamental group $\Gamma$ acts on $E$ by deck transformations, and hence on $H^*(E)$ and $H_*(E)$, which may then be viewed as $\Gamma$-modules.  

\begin{thmnonumber}[Cartan-Leray] Let $E\to X$ be a covering projection, and let $\Gamma=\pi_1(X)$.  There is a spectral sequence converging to $H^*(X)$ with $E_2^{p,q}\cong H_{alg}^p(\Gamma; H^q(E))$.
\end{thmnonumber}

The following Proposition, which will be used in Chapter \ref{chapter:prequantization}, is a direct application of the Cartan-Leray spectral sequence of a covering.  It may be viewed as the Serre exact sequence for the fibration $X\to B\pi_1(X)$ that classifies the universal covering projection $\tilde{X}\to X$.

\begin{prop} \label{prop:serreexactsequence}
Let $\tilde{X}\to X$ be a universal covering projection, where $H^q(\tilde{X};\Z)=0$ for $1\leq q < N$.  Then there is an exact sequence
\[ 0\to H^N_{alg}(\pi;\Z) \to H^N(X;\Z) \to H^N(\tilde{X};\Z)^\pi \to H^{N+1}_{alg}(\pi;\Z)  \]
where $\pi=\pi_1(X)$.
\end{prop}
\begin{proof}
Consider the Cartan-Leray spectral sequence for the covering $\tilde{X}\to X$ with $E_2^{p,q}=H^p_{alg}(\pi; H^q(\tilde{X};\Z))$, which is zero when $1\leq q < N$. Therefore, the first non-trivial differential is $d:E_N^{0,N}\to E_N^{N+1,0}$. The result follows from the family of exact sequences
$$
 0 \to E_\infty^{0,n} \to E_n^{0,n} \to E_n^{n+1,0}\to E_\infty^{n+1,0} \to 0 
 $$
for $n\leq N$, and in particular for $n=N-1$ and $n=N$, as well as the exact sequence
$$
 0\to E_\infty^{n,0} \to H^n(X;\Z) \to E_\infty^{0,n} \to 0  
 $$
in the same range. Indeed, splicing these two sequences together gives the desired sequence, since $H^p_{alg}(\pi; H^0(\tilde{X};\Z))=H^p_{alg}(\pi;\Z)$ and $H^0_{alg}(\pi;H^q(\tilde{X};\Z))\cong H^q(\tilde{X};\Z)^\pi$ (see \cite{Weibel}). 
\end{proof}

\subsection*{The Bockstein spectral sequence}

In Chapters \ref{chapter:liegroups} and \ref{chapter:calc} it will be necessary to extract information about cohomology with integer coefficients using cohomology with $\Z_p$ coefficients for every prime $p$.   The main tool used in this setting is the Bockstein spectral sequence, which will be reviewed here.

Let $p$ be prime, and consider the long exact sequence in cohomology 
$$
\cdots \to H^q(X;\Z) \toby{p} H^q(X;\Z) \toby{\mathrm{red}} H^q(X;\Z_p) \toby{\tilde\beta} H^{q+1}(X;\Z) \to H^{q+1}(X;\Z) \to \cdots
$$
corresponding to the short exact sequence of coefficient groups $0\to \Z \toby{p} \Z \toby{\mathrm{red}} \Z_p \to 0$, where $p:\Z\to \Z$ denotes multiplication by $p$, and $\mathrm{red}$ denotes reduction modulo $p$.   This long exact sequence rolls up into an exact couple,
\begin{diagram}[h=1.8em]
D^{q} &&&D^{q+1}&&&D^{q+2}&& \\
    &\rdTo(5,2)^p&&&\rdTo(5,2)^p &&&&\\
    &&D^{q-1}&&&D^{q}&&&D^{q+1}\\
\uTo<{\tilde\beta} &\ldTo^{\mathrm{red}}&&\uTo<{\tilde\beta}&\ldTo{\mathrm{red}}&&\uTo<{\tilde\beta} &\ldTo^{\mathrm{red}}&\\
E^{q-1} &&&E^q&&&E^{q+1}&& \\
\end{diagram}
with $E^*=H^*(X;\Z_p)$ and $D^*=H^*(X;\Z)$, which gives rise to a spectral sequence called the \emph{Bockstein spectral sequence}.  

From the diagram of short exact sequences,
\begin{diagram}
0 	& \rTo	&  \Z	& \rTo^p 	& \Z	& \rTo	& \Z_p	& \rTo	& 0 \\	
	&		& \dTo>{\mathrm{red}}	& 		& \dTo	&		&\dEquals&	& \\	
0 	& \rTo	&  \Z_p 	& \rTo 	& \Z_{p^2}	& \rTo	& \Z_p	& \rTo	& 0 \\	
\end{diagram}
it is clear that the $d^1$ differential is  $d^1=\mathrm{red}\circ\tilde\beta=\beta:H^q(X;\Z_p) \to H^{q+1}(X;\Z_p)$, the connecting homomorphism corresponding to the long exact sequence in cohomology of the bottom row.  The $d^r$ differential is denoted $\beta^{(r)}$.

The Bockstein spectral sequence is singly graded and has the effect of filtering out $p^r$-torsion in $H^*(X;\Z)$.  The following proposition summarizes its main features. See \cite{S}, \cite{Mc}, or \cite{Weibel} for more precise statements and details.

\begin{propnonumber}
If $H^*(X;\Z)$ is finitely generated in each dimension, then the Bockstein spectral sequence converges, and 
$$
E_\infty \cong (H^*(X;\Z))_{\mathrm{free}}\otimes \Z_p,
$$
where $A_\mathrm{free}=A/A_{\mathrm{torsion}}$ denotes the free part of the finitely generated abelian group $A$. Moreover, let  $y=\mathrm{red}(z)\in H^q(X;\Z_p)$ be non-zero for some $z \in H^q(X;\Z)$.  If $y$ survives to $E_r$ and becomes $0$ in $E_{r+1}$ (because it is in the image of $\beta^{(r)}$), then (ignoring torsion prime to $p$) $z$ generates a $\Z_{p^r}$-summand in $H^q(X;\Z)$. 
\end{propnonumber}

\begin{remark}
The statement regarding convergence says that if a non-zero cohomology class $x\in H^q(X;\Z_p)$ survives to $E_\infty$ (i.e.~$\beta^{(j)}(x)=0$ for all $j$ and $x$ is not in the image of $\beta^{(r)}$ for any $r$), then $x$ is the reduction mod $p$ of a torsion free class $z\in H^q(X;\Z)$ that is not divisible by $p$.  
\end{remark}

\begin{remark} \label{remark:bockstein}
The above proposition shows how to verify that  a given element $z\in H^q(X;\Z)$ generates a $\Z_m$ summand.  Write $m=p_1^{r_1} \cdots p_s^{r_s}$, so that 
$$
\Z_m\cong \Z_{p_1^{r_1}} \oplus \cdots \oplus \Z_{p_s^{r_s}},
$$
and let $y_i=\mathrm{red}(z)\in H^q(X;\Z_{p_i})$  for $i=1, \ldots, s$.  According to the proposition, it would suffice to verify $y_i$ survives to $E_{r_i}$ but lies in the image of $\beta^{(r_i)}$ in the mod $p_i$ Bockstein spectral sequence, for each prime $p_i$.  
\end{remark}

The following example computes the mod 2 Bockstein spectral sequence  for the Moore space $P^n(4)$ directly from the definition, and verifies the above proposition in that case. 

\begin{eg} The Moore space $P^n(4)$
\end{eg}

Let $P^n(4)$ be the CW-complex obtained by attaching an $n$-cell to the sphere $S^{n-1}$ via the attaching map $4:S^{n-1} \to S^{n-1}$ of degree 4.  In other words, $P^n(4)$ is the cofibre of $4:S^{n-1} \to S^{n-1}$. The long exact sequence in cohomology easily gives the cohomology of $P^n(4)$; namely, $\tilde{H}^q(P^n(4);\Z)\cong\Z_4$ for $q=n$, and $0$ otherwise. In fact, the same long exact sequence with $\Z_m$  coefficients, where $m=2$ or $4$, yields
$$ \tilde{H}^q(P^n(4);\Z_m) \cong \left\{ \begin{array}{cl} 
\Z_m & \text{if }q=n-1\text{, or }n \\
0 & \text{otherwise}.
\end{array} \right.
$$

The Bockstein spectral sequence has $E_1^*=H^*(P^n(4);\Z_2)$, with the only possible non-trivial differential $\beta:H^{n-1}(P^n(4);\Z_2)\to H^{n}(P^n(4);\Z_2)$.  To compute $\beta$, recall the exact couple formalism:
\begin{diagram}[h=2em]
H^{n}(P^n(4);\Z) &&&0&&&    && \\
    &\rdTo(5,2)^2&&&    &\\
    &&0&&&H^{n}(P^n(4);\Z)     \\
\uTo<{\tilde\beta} &\ldTo^{\mathrm{red}}&&\uTo&\ldTo{\mathrm{red}}&\\
H^{n-1}(P^n(4);\Z_2) &&&H^{n}(P^n(4);\Z_2)&& \\
\end{diagram}
It is clear that $\tilde\beta$ is non-trivial, since the above sequence is exact.  Therefore, $\tilde\beta$ sends the non-trivial element in $H^{n-1}(P^n(4);\Z_2)\cong \Z_2$ to the element of order 2 in $H^{n}(P^n(4);\Z)\cong \Z_4$, from which it follows that $\beta=\mathrm{red}\tilde\beta=0$.  
Therefore, $E_2^*=E_1^*$.  

To compute $\beta^{(2)}$, note that $D'=2 H^*(P^n(4);\Z)$, and recall that $\mathrm{red}':D'\to E'=E_2$ is defined by $\mathrm{red}'(2x)=\mathrm{red}(x)$.  Therefore $\mathrm{red}'$ is an isomorphism.  And since $\tilde\beta$ sends a generator of $H^{n-1}(P^n(4);\Z_2)$ to twice a generator of $H^{n}(P^n(4);\Z)$, the induced map $\tilde\beta'$ is an isomorphism.  Therefore, $\beta^{(2)}:H^{n-1}(P^n(4);\Z_2) \to H^{n}(P^n(4);\Z_2)$ is an isomorphism as well, and the spectral sequence collapses with $E_\infty=E_3=0$.  

To verify the proposition, note that the diagram of long exact sequences
\begin{diagram}
H^{n-1}(S^{n-1};\Z) & \rTo^4 & H^{n-1}(S^{n-1};\Z) & \rTo & H^{n}(P^n(4);\Z) & \rTo & 0 \\
\dTo & & \dTo & & \dTo & & \\
H^{n-1}(S^{n-1};\Z_2) & \rTo^0 & H^{n-1}(S^{n-1};\Z_2) & \rTo & H^{n}(P^n(4);\Z_2) & \rTo & 0 \\
\end{diagram}
shows that the generator $y\in H^{n}(P^n(4);\Z_2)$ is the reduction mod $2$ of a generator in $H^{n}(P^n(4);\Z) \cong \Z_4$.  And if $x\in H^{n-1}(P^n(4);\Z_2)\cong \Z_2$ denotes the non-trivial element, the above shows that $\beta^{(2)}(x)=y$, as required. \hfill $\qed$
\\

The following elementary example illustrates the underlying method of the calculations appearing in Chapter \ref{chapter:calc}.

\begin{eg} The pinch map $P^n(4)\to S^{n+1}$
\end{eg}

Recall from the above example that the $n$-skeleton of  $P^n(4)$ is the sphere $S^n$.  Collapsing the $n$-skeleton of $P^n(4)$ to a point clearly yields the sphere $S^{n+1}$, and the resulting map $f:P^n(4)\to S^{n+1}$ is called the \emph{pinch map}.  This example computes the induced map $f^*$ with integer coefficients using the tools described above.  Note that this is a trivial computation, which essentially appeared in the previous example; however, the discussion that follows is meant to illustrate the methods appearing in Chapter \ref{chapter:calc}.  

First, observe that $f$ is the map which sends $P^n(4)$ to the cofibre of the inclusion of the $n$-skeleton $S^n \hookrightarrow P^n(4)$.  And from the corresponding long exact sequence,
$$
0 \to H^n(P^n(4);\Z_2) \to H^n(S^n;\Z_2) \to H^{n+1}(S^{n+1};\Z_2) \toby{f^*} H^{n+1}(P^n(4);\Z_2) \to 0
$$
simply knowing that each of the groups is $\Z_2$ (by the previous example) shows that $f^*$ is an isomorphism with $\Z_2$ coefficients.   

To compute $f^*$ over the integers, notice that the generator $a\in H^{n+1}(S^{n+1};\Z_2)$ is the reduction mod $2$ of a generator in $H^{n+1}(S^{n+1};\Z)$.  The above verifies that $f^*(a)=y$, using the notation of the previous example (which also showed that $y$ is in the image of $\beta^{(2)}$).  Therefore, by the commutative diagram
\begin{diagram}
H^*(S^{n+1};\Z) & \rTo^{f^*} & H^*(P^n(4);\Z) \\
\dTo>{\mathrm{red}} & & \dTo>{\mathrm{red}} \\
H^*(S^{n+1};\Z_2) & \rTo^{f^*} & H^*(P^n(4);\Z_2) \\
\end{diagram}
$f^*:H^{n+1}(S^{n+1};\Z) \to H^{n+1}(P^n(4);\Z)$ is reduction mod $4$. 
\hfill $\qed$

\chapter{Topology of Lie Groups} \label{chapter:liegroups}

This chapter recalls some basic definitions and facts regarding  Lie groups and their Lie algebras, mainly to establish notation and for convenient reference. In particular, the classification of compact simple Lie groups, and results concerning their cohomology are reviewed.  The reader may find it useful to consult \cite{MT}, \cite {Ka}, \cite{BtD}, or \cite{H}  for details.

\section{Simple Lie groups} \label{sec:liegroupsprelim}

Let $G$ be a compact Lie group with Lie algebra $\mathfrak{g}$.  Recall that a non-abelian (i.e. $[\, , \,]\neq0$) Lie algebra $\mathfrak{g}$ is \emph{simple} if $\mathfrak{g}$ is the only non-trivial ideal, and that a Lie group $G$ is simple if its Lie algebra $\mathfrak{g}$ is simple.  Much of this work deals with compact  simple Lie groups, for which one has the well known Cartan-Killing classification. An exhaustive list of such Lie groups  appears in Table \ref{table:list} below.  

The classification of compact simple Lie groups follows from the classification of   (finite dimensional) simple Lie algebras.  Each such Lie algebra determines a connected, simply connected Lie group (Lie's third theorem), and hence the classification of simple Lie algebras determines all possible simply connected compact simple Lie groups $\tilde{G}$.  The remaining compact simple Lie groups are necessarily quotients $G=\tilde{G}/Z$, where $Z\cong \pi_1(G)$ is a central subgroup of $\tilde{G}$. 


Table \ref{table:list} below lists each simply connected compact simple Lie group $\tilde{G}$ with center $Z(\tilde{G})$ and the possible quotients $G=\tilde{G}/Z$ by central subgroups $Z \subset Z(\tilde{G})$.  
Note that the center of the spinor group $Spin(4n)$ has two central subgroups of order 2. If $a$ denotes the generator of the kernel of the double cover $Spin(4n)\to SO(4n)$, and $b$ denotes another generator of the center $Z(Spin(4n))\cong\Ztwo \oplus \Ztwo$, then the  \emph{semi-spinor group} $Ss(4n)$ is the quotient $Spin(4n)/\langle b \rangle$ \cite{IKT}.

\begin{table}[h] \label{table:list}
\begin{center}
\caption{Compact simple Lie groups}
\begin{tabular}{|c|c|c|c|}
\hline
 $\tilde{G}$ & $Z(\tilde{G})$ & $Z\subset Z(\tilde{G})$ & $G=\tilde{G}/Z$   \\
  \hline\hline
  \multirow{2}{*}{$SU(n)$}   & \multirow{2}{*}{$\Zn$}  & $\Zn$ & $PU(n)$  \\
  		   &  &$\Z_k$, $k|n$& $SU(n)/\Z_k$   \\
  \hline
  $Sp(n)$  &$ \Ztwo$  & $\Ztwo$ & $PSp(n)$  \\
  \hline
    $Spin(2n+1)$   & $\Ztwo$ &  $\Ztwo$ & $SO(2n+1)$   \\
  \hline
  $Spin(4n+2)$ & \multirow{2}{*}{$\Z_4$} & $\Z_4$ & $PO(4n+2)$   \\
	 &	& $\Ztwo$ & $SO(4n+2)$    \\
\hline
 \multirow{3}{*}{$Spin(4n)$}   & \multirow{3}{*}{$\Ztwo \times \Ztwo$} & $ \Ztwo \times \Ztwo$ & $PO(4n) $  \\
  	  & &  $\Ztwo=\langle a\rangle$ & $SO(4n)$   \\
	  & & $\Ztwo=\langle b\rangle$ & $Ss(4n)$ \\
\hline
  $E_6$  & $\Z_3$ & $\Z_3$ &$PE_6$   \\
  \hline
  $E_7$  &  $\Ztwo$& $\Ztwo$ & $PE_7$   \\
\hline
$E_8$, $F_4$, $G_2$ & 0 & \multicolumn{1}{c}{} & \multicolumn{1}{c}{} \\
\cline{1-2}
\end{tabular}
\end{center}
\end{table}

\section{Cohomology of Lie groups}\label{sec:coholiegroup}

The calculations appearing in Chapter \ref{chapter:calc} use known results regarding the cohomology of Lie groups.  This section reviews some elementary  results surrounding the cohomology of Lie groups that will be used later.

The structure maps of a topological group endow the cohomology ring with some extra algebraic structure.  Recall that the multiplication map $\mu:G\times G\to G$ of a topological group, or more generally an $H$-group (see Section \ref{sec:wheadprod}), endows the cohomology ring $H^*(G;\mathbb{F})$, where $\mathbb{F}$ is a field,  with a coproduct
$$
\mu^*: H^*(G;\mathbb{F}) \to H^*(G\times G;\mathbb{F}) \cong H^*(G;\mathbb{F})\otimes H^*(G;\mathbb{F}).
$$
For $x\in H^n(G;\mathbb{F})$, write $\mu^*x=x\otimes 1 + 1 \otimes x + \sum x_i' \otimes x_j''$, for some $x_k'$ and $x_k''$ of degree $k<n$.  Often, one writes $\bar\mu^*x= \sum x_i' \otimes x_j'' $, which is known as the \emph{reduced coproduct}.  
The diagonal map $\Delta:G\to G\times G$ induces the multiplication,
$$
\Delta^*:H^*(G;\mathbb{F})\otimes H^*(G;\mathbb{F}) \cong H^*(G\times G;\mathbb{F}) \to H^*(G;\mathbb{F}),
$$
which makes $H^*(G;\mathbb{F})$  a \emph{Hopf algebra} (see \cite{MM}). 
Appendix \ref{chapter:app}  lists the cohomology Hopf algebras of all non-simply connected compact simple Lie groups.

The induced homomorphism on cohomology for the inversion map $c(g)=g^{-1}$ of a topological group may be computed from the relation $\mu \circ(\mathrm{id} \times c) \circ \Delta \approx *$.  Indeed, suppose first that $x\in H^*(G;\mathbb{F})$ is \emph{primitive} (i.e. $\mu^*x=x\otimes 1 + 1 \otimes x$).  Then,
\begin{align*}
0&=\Delta^*(\mathrm{id}\times c)^*\mu^* x \\
&= \Delta ^*(\mathrm{id}\times c)^*(x\otimes 1 + 1\otimes x) \\
&= \Delta ^*(x\otimes 1 + 1\otimes c^*x) \\
&=x+c^*x
\end{align*}
and hence $c^*x=-x$.  The non-zero cohomology classes of smallest degree are necessarily primitive.  Therefore, if $x$ is not primitive, $c^*x$ may be computed recursively from the knowledge of $c^*x'$ for all $x'$ of smaller degree.

\subsection*{Degree 3 cohomology}

It will be helpful in subsequent chapters to be familiar with some aspects of degree 3 cohomology of Lie groups, and this section recalls some important features of $H^3(G;R)$ with coefficient ring $R=\Z$ and $R=\R$.

For a Lie group $G$ with Lie algebra $\mathfrak{g}$, recall that the left-invariant, and right invariant Maurer-Cartan forms, denoted $\theta^L$ and $\theta^R$ respectively,  are the $\mathfrak{g}$-valued 1-forms defined by 
$$
\theta^L_g(v)= L_{g^{-1}*}(v), \quad \text{and} \quad \theta ^R_g(v)=R_{g^{-1}*}(v),
$$
for all $v \in T_gG$, where $L_h$ and $R_h$ denote left and right multiplication by $h$, respectively.   Let $(-,-)$ denote an $Ad$-invariant inner product on $\mathfrak{g}$.  

For any space $X$, recall that the Lie bracket $[-,-]:\mathfrak{g} \otimes \mathfrak{g} \to \mathfrak{g}$ induces  a map  $[-,-]:\Omega^*(X;\mathfrak{g} \otimes \mathfrak{g}) \to \Omega^*(X;\mathfrak{g})$, which in turn induces a map $[-,-]:\Omega^*(X;\mathfrak{g}) \otimes \Omega^*(X;\mathfrak{g}) \to \Omega^*(X;\mathfrak{g})$ given by the composition with the wedge product:
$$
\Omega^*(X;\mathfrak{g}) \otimes \Omega^*(X;\mathfrak{g}) \toby{\wedge} \Omega^*(X;\mathfrak{g} \otimes \mathfrak{g}) \toby{[-,-]} \Omega^*(X;\mathfrak{g}).
$$
Similarly, the inner product on $\mathfrak{g}$ induces a map $(-,-):\Omega^*(X;\mathfrak{g})\otimes \Omega^*(X;\mathfrak{g}) \to \Omega^*(X)$.  

The canonical 3-form $\eta \in \Omega^3(G)$ is defined using the $Ad$-invariant inner product by
$$
\eta=\frac{1}{12}(\theta^L,[\theta^L,\theta^L])=\frac{1}{12}(\theta^R,[\theta^R,\theta^R]);
$$
therefore, $\eta$ is bi-invariant and hence closed (by Proposition 12.6 in  \cite{Bredon}, for example),  defining  a (de Rham) cohomology class $[\eta] \in H^3(G;\R)$.  Recall that $H^*(G;\R)$ may be calculated from the complex of bi-invariant forms on $G$ with trivial differential, and hence $H^*(G;\R)\cong (\wedge^*\mathfrak{g}^*)^G$, from which it follows that $H^3(G;\R)\cong\R$ when $G$ is simple (see \cite{Bredon}).

For a compact simple Lie group $G$ whose fundamental group $Z=\pi_1(G)$ is cyclic, $H^3(G;\Z)\cong \Z$.  Only $G=PO(4n)$ in Table \ref{table:list} has a non-cyclic fundamental group, in which case $H^3(PO(4n);\Z) \cong \Z\oplus \Z_2$.   One way to see this is as follows.  Let $\pi:\tilde{G}\to G$ be the universal covering homomorphism, with kernel $Z\subset \tilde{G}$.  Since $Z$ acts on $\tilde{G}$ by translation, it induces the identity map on cohomology.  Indeed, for any $z\in Z$, choosing a path in the connected group $\tilde{G}$ yields a homotopy $L_z\sim 1$, where $1:G\to G$ denotes the identity map.  In other words, $Z$ acts trivially on $H^*(\tilde{G};\Z)$, which implies that the local coefficient system $\mathcal{H}^*(\tilde{G};\Z)$ on $BZ$ (see Section \ref{sec:classifyingspaces}) is trivial.  Hence, in the Serre spectral sequence for the fibration $G\to BZ$ with homotopy fibre $\tilde{G}$,  $E_2  \cong H^*(BZ; H^*(\tilde{G};\Z))$. Since $\tilde{G}$ is $2$-connected, $E_2^{p,q}=0$ for $q=1,2$, and by the Universal Coefficient theorem and the Hurewicz theorem $H^3(\tilde{G};\Z) \cong H_3(\tilde{G};\Z)\cong \pi_3(\tilde{G}) \cong \Z$ (see \cite{MT}).  Therefore, $E_2^{0,3}=\Z$.  If $Z$ is cyclic, $H^3(BZ;\Z)=0$, and  $H^3(G;\Z)\cong \Z$ in this case.  If $Z\cong \Z_2\oplus \Z_2$, then $H^3(BZ;\Z)\cong \Z_2$, which shows that $H^3(PO(4n);\Z)\cong \Z\oplus \Z_2$.  

\begin{remark} The familiar fact that $H^3(G;\R)\cong \R$ may now be recovered by the Universal Coefficient theorem, and the discussion in the preceding paragraph.  In fact, since $\tilde{H}^*(BZ;\R)=0$, the above shows that the induced map $\pi^*:H^*(G;\R) \to H^*(\tilde{G};\R)$ is an isomorphism.
\end{remark}

Note also that the covering $\pi:\tilde{G} \to G$ induces an isomorphism of Lie algebras, which will actually be used as an identification.  In other words, by abuse of notation, $\pi^*\eta=\eta$.  
By the above discussion, it follows that some multiple of $[\eta]$ lies in the image of the coefficient homomorphism $\iota_\R:H^*(\tilde{G};\Z) \to H^*(\tilde{G};\R)$, or equivalently, for a certain choice of $Ad$-invariant inner product on $\mathfrak{g}$, $[\eta]=\iota_\R(1)$, where $1\in H^3(\tilde{G};\Z)$ is a generator.  It is well known (see \cite{PS}) that the basic inner product (see Remark \ref{remark:innerproducts} in Section \ref{sec:liethy}) has this property.

Let $G$ act on itself by conjugation.  
There is an equivariant extension $\eta_G$ of the canonical 3-form $\eta$ with similar cohomological properties.  In the Cartan model of equivariant differential forms, $\eta_G$ is given by
$$
\eta_G(\xi)=\eta+\half(\theta^L+\theta^R,\xi).
$$
It is easy to check that $d_G\eta_G=0$, and hence defines a cohomology class $[\eta_G]\in H^3_G(G;\R)$.  Proposition \ref{lemma:simplyG2} below shows that $H^3_{\tilde{G}}(\tilde{G};\Z)\cong \Z$, and that $[\eta_G]=\iota_\R(1)$, where $1\in H^3_{\tilde{G}}(\tilde{G};\Z)$ is a generator.

\begin{prop} \label{lemma:simplyG2} If $G$ is simply
connected, then the canonical map $H_G^3(G) \to  H^3(G)$ is an isomorphism (with any coefficients).
\end{prop}
\begin{proof}  It will suffice to show that the inclusion $p:G \to G_G$ of the fiber in the Borel construction $G_G:=(G\times EG)/G \to BG$ induces an isomorphism  $p_\sharp:\pi_k(G)\toby{\cong} \pi_k(G_G)$ for $k=1$, $2$, and $3$.  Indeed, since $\pi_1(G)=\pi_2(G)=0$, $p_*:H_3(G;\Z) \to H_3(G_G;\Z)$ is an isomorphism by the Hurewicz Theorem.  Therefore the same is true for $p^*$ as well, which is the statement of the lemma.  

To see that the induced homomorphism $p_\sharp$ is an isomorphism on $\pi_k(-)$ for $k=1$, $2$, and $3$, consider the long exact sequence of homotopy groups for the fibration  $G \toby{p} G_G \to BG$.  For $k=1$ and $2$ this is clear, since $BG$ is $3$-connected, and for $k=3$, $p_\sharp: \pi_3(G) \to \pi_3(G_G)$ is surjective for the same reason.  To check that $p_\sharp$ is also injective, it is equivalent to check that the connecting homomorphism $\partial: \pi_4(BG) \to \pi_3(G)$ is zero.  This homomorphism may be factored as $\pi_4(BG) \cong \pi_3(G)  \toby{j_\sharp} \pi_3(G\times EG) \cong \pi_3(G)$, where  $j:G\to G\times EG$ includes $G$ as the fiber of $p':G\times EG \to G_G$. But $j$ is null homotopic, as it factors through the contractible space $EG$, sending $g\in G$ to $g\cdot(e,*)=(geg^{-1},g\cdot *)=(e,g\cdot *)$.  Hence $\partial=0$.
\end{proof}

\section{Alcoves and the Weyl group} \label{sec:liethy}

The center $Z(G)$ of a compact simple Lie group $G$ acts on $G$ by translation.  Since this action commutes with the conjugation action of $G$, $Z(G)$ acts on the set of conjugacy classes  in $G$.   As will be recalled below, if $G$ is simply connected, there is a bijection between the set of conjugacy classes in $G$ and an alcove of $G$, and the main purpose of this section is to describe the corresponding $Z(G)$-action on the alcove.  Such a description will take some time to develop, since there are various definitions and facts that will need to  be recalled from Lie theory.  
\\

Let $G$  be a compact simple Lie group.  Choose a maximal torus $T\subset G$---a subgroup $T$  isomorphic to $S^1\times \cdots \times S^1$  that is not contained in any other such subgroup---and let $\mathfrak{t}$ be its Lie algebra.  Choose an  inner product $(-,-)$ on $\mathfrak{g}$ that is invariant under the adjoint action $Ad:G\times \mathfrak{g} \to \mathfrak{g}$. 

\begin{remark} \label{remark:innerproducts} Since $\mathfrak{g}$ is simple, all $Ad$-invariant inner products are (negative) multiples of the \emph{Killing form} $\kappa(\xi,\zeta)=\mathrm{Trace}(\mathrm{ad}\xi \circ \mathrm{ad} \zeta)$ \cite{BtD}. The \emph{basic inner product} $B(-,-)$ is the invariant inner product on $\mathfrak{g}$ normalized to make long roots (see Remark \ref{remark:longroots} below) have length $\sqrt{2}$.
\end{remark}

A non-zero element $\alpha \in \mathfrak{t}^*$ is a \emph{root}, if the subspace $$
V_\alpha = \{v\in \mathfrak{g}_\mathbb{C} \, | \, [x,v]=2\pi i \alpha(x)v , \text{for all }x\in \mathfrak{t}\}
$$ is non-zero, where $\mathfrak{g}_\mathbb{C}=\mathfrak{g}\otimes \mathbb{C}$.


\begin{remark} \label{remark:longroots}
The roots of $G$ form an \emph{irreducible root system} for the vector space $\mathfrak{t}^*$.   Using the inner product to identify $\mathfrak{t}\cong \mathfrak{t}^*$, it can be shown (see \cite{BtD}, for example) that there are at most two different lengths for a root.  Consequently, roots are either \emph{long} or \emph{short}.
\end{remark}

Each root $\alpha$ determines hyperplanes $L_\alpha(q)=\alpha^{-1}(q) \subset \mathfrak{t}$, where $q\in \Z$.   Let $s_\alpha:\mathfrak{t}\to \mathfrak{t}$ denote reflection in the hyperplane $L_\alpha(0)$.  Then there is a unique vector $\alpha^\vee \in \mathfrak{t}$ called the \emph{co-root} (or \emph{inverse root}) characterized by the equation $s_\alpha(x)=x-\alpha(x)\alpha^\vee$ (Proposition V-2.13 in \cite{BtD}).   Using the inner product to identify $\mathfrak{t}^*\cong \mathfrak{t}$,  $\alpha^\vee=\frac{2}{(\alpha,\alpha)}\alpha$ so that the reflections $s_\alpha$ take the form
$$
s_\alpha(x)=x-2\frac{(\alpha,x)}{(\alpha,\alpha)}\alpha.
$$ 
The group $W$ generated by the reflections $s_\alpha$ is called the \emph{Weyl group} of $G$.  

Co-roots lie in $I=\ker \exp_T$, the integer lattice (Proposition V-2.16 in \cite{BtD}), and therefore the group $\Gamma$ generated by co-roots is a subgroup of $I$. Furthermore, $\pi_1(G) \cong I/\Gamma$ (Theorem V-7.1 in \cite{BtD}).  (Therefore the inclusion $\Gamma\subseteq I$ is an equality for simply connected Lie groups.)

The \emph{Weyl chambers} of $G$ are the closures of the connected components of $\mathfrak{t}\setminus \bigcup_{\alpha} L_\alpha(0)$, and the  \emph{alcoves}  are the closures of the connected components of  $\mathfrak{t}\setminus \bigcup_{\alpha,q} L_\alpha(q)$.  

An important fact that will be used later is that for simply connected Lie groups, points in an alcove are in one-to-one correspondence with the set of conjugacy classes in the group.  The main reason for this correspondence is that each element of the group is conjugate to exactly one element of the maximal torus that lies in $\exp(\Delta) \subset T$.  The following theorem summarizes the relevant facts.  (See Section 2 in Chapter IV of \cite{BtD} for proofs.)
Let $W_{ext}$ be the group generated by the Weyl group $W$ and the group $\Gamma$ of translations by co-roots.  Similarly, let $W'$ be the group generated by $W$ and the group of translations $I$ generated by the integral lattice.

\begin{thmnonumber} Let $G$ be a compact Lie group, and let $\Delta\subset \mathfrak{t}$ be an alcove.  
\begin{enumerate}
\item $W_{ext}=\Gamma\ltimes W$ and $W'=I\ltimes W$.
\item $\Delta$ is a fundamental domain for the $W_{ext}$-action on $\mathfrak{t}$.
\item There is a one-to-one correspondence between $\mathfrak{t}/W'$ and $Con(G)$, the set of conjugacy classes in $G$, given by $\xi + W' \mapsto Ad_G\cdot \exp \xi$.
\end{enumerate}
\end{thmnonumber}

\begin{remark}\label{remark:alcovecorrespondence}
 For a simply connected Lie group $G$, $I=\Gamma$.  Therefore, $W'=W_{ext}$, and $\mathfrak{t}/W' = \Delta$, since $\Delta$ is a fundamental domain.  For non-simply connected Lie groups $G$, the correspondence $\Delta \to Con(G)$ 
is merely a surjection.
\end{remark}

\subsection*{The $Z(G)$-action on the alcove}

Assume that the simple compact Lie group $G$ is simply connected.  
Since the action of the center $Z(G)$ by translation in $G$ commutes with the conjugation action, $Z(G)$ acts on $Con(G)$, the set of conjugacy classes in $G$.  Since $Con(G)\cong \Delta$, this induces an action of $Z(G)$ on $\Delta$.  

Let $z=\exp \zeta \in Z(G)\subset T$, where $T$ is the maximal torus.  Since $z$ acts by translation, it sends the conjugacy class of $\exp\xi$ to $z\exp\xi=\exp\zeta\exp\xi=\exp(\zeta+\xi)$.  Therefore, the induced action on $\xi \in\Delta$ can be described by declaring $z\cdot \xi$ to be the unique element in $\Delta$ satisfying $w(\zeta + \xi)=z\cdot \xi$ for some $w\in W_{ext}=W'$.  

Let $\Lambda=\exp^{-1} Z(G)$ be the \emph{central elements} in $\mathfrak{t}$, so that  $Z(G)\cong \Lambda/I$.  Since $Z(G)$ acts trivially on $\mathfrak{g}$, for any root $\alpha\in \mathfrak{t}^*$, $\alpha(\Lambda)\subset \Z$. That is, $\Lambda$ is contained in the family of hyperplanes $\{L_\alpha(q) | q\in \Z\}$ for each root $\alpha$.
 Conversely, any element that meets every family of hyperplanes $\{L_\alpha(q) | q\in \Z\}$ must act trivially on $\mathfrak{g}$, and is therefore contained in the center $Z(G)$.  

\begin{eg} $G=SU(n)$ \label{eg:conj}
\end{eg}

The alcove $\Delta$ for $SU(n)$ is an $(n-1)$-simplex, and it will be verified next that the center $Z(SU(n))\cong \Z_n$ acts by shifting the vertices.   That is, if $v_0, \ldots, v_{n-1}$ denote the vertices of the simplex $\Delta$ (with usual labelling---see below), the generator of $\Z_n$ sends the vertex $v_i$ to $v_{i+1}$ for $i=0,\ldots, n-2$ and $v_{n-1}$ to $v_0$.  

To begin, recall that the maximal torus $T$ consists of diagonal matrices in $SU(n)$, and hence $\mathfrak{t}$ may be identified with the subspace $\{x=\sum x_ie_i|\sum x_i=0\} \subset \R^n$ where $e_i$ denotes the standard basis vector.  The inner product inherited from $\R^n$ is $Ad$-invariant, and is readily checked to be the basic inner product.

The roots are $\theta_i-\theta_j$ ($i\neq j$) restricted to $\mathfrak{t}$, where $\theta_i:\R^n \to \R$ is given by $\theta_i(\sum x_j e_j)=x_i$.  Therefore, under the identification $\mathfrak{t}\cong\mathfrak{t}^*$, the roots are simply the vectors $e_i-e_j$, which all have squared length 2, and the co-roots coincide with the roots. 

An alcove is given by (see \cite{BtD} and \cite{MT})
$$
\Delta =\left\{ \sum x_i e_i \in \mathfrak{t}\, \big|\,  x_i\geq x_{i+1}, \text{ and } x_1 -x_n\leq 1\right\},
$$
from which it is easy to see that the alcove is indeed an $(n-1)$-simplex.  Indeed, each vertex is the solution to $n-1$ of the $n$ equalities 
$$
x_i=x_{i+1}, \, \text{for }\, i=1,\ldots, n-1,\,\text{and }\,x_1-x_n=1.
$$
This gives vertices 
\begin{align*}
v_0 & = 0 \\
v_1 & =\tfrac{1}{n} ((n-1)e_1 - e_2  -\cdots - e_{n-1} -e_n) \\
v_2 & =\tfrac{1}{n} ((n-2)  e_1 + (n-2)e_2 - \cdots - 2e_{n-1} + 2 e_n)\\
& \, \, \, \vdots \\
v_{n-1} & =\tfrac{1}{n} ( e_1 + e_2 + \cdots + e_{n-1} -(n-1) e_n).
\end{align*}

The vertices of the alcove are also central elements.  To describe the action of the center $Z(SU(n))\cong \Z_n$ it suffices to describe translation by a central element $\zeta \in \Lambda$ of order $n$ in $I$, such as $\zeta=v_1$.  

To that end, consider $\displaystyle v_k=\tfrac{n-k}{n} \sum_{j=1}^k e_j - \tfrac{k}{n}\sum_{j=k+1}^n e_j$, and compute $\zeta+v_k$:
\begin{align*}
\zeta + v_k & = \frac{1}{n}\left[ (n-1)e_1-\sum_{j=2}^ne_j + (n-k)\sum_{j=1}^k e_j   -k\sum_{j=k+1}^ne_j\right] \\
&= \frac{1}{n}\left[ (2n-k-1)e_1 + (n-k+1)\sum_{j=2}^k e_j -(k+1)\sum_{j=k+1}^n e_j \right]
\end{align*}
Notice that $\zeta+v_k-v_{k+1} = e_1-e_{k+1}$, a co-root.  Therefore, these elements are in the same $W_{ext}=\Gamma \ltimes W$-orbit, and the generator $z\in Z(SU(n))$ acts by $z\cdot v_k=v_{k+1}$ for $k=0, \ldots, n-2$.  A similar calculation verifies that $ z\cdot v_{n-1}=v_0$.

\begin{remark} \label{remark:specialconjclass}
Notice that since the center $Z(SU(n))$ acts by shifting the vertices, there is exactly one fixed point in $\Delta$, the barycenter.  Indeed, writing $x\in \Delta$ in barycentric coordinates $x=\sum t_iv_i$ with $\sum t_i=1$ shows that $x$ is fixed by a generator precisely when $t_i=\tfrac{1}{n}$, since in this case, the barycentric coordinates of $x$ are unique.  In other words, the conjugacy class corresponding to the barycenter $\zeta_0=\tfrac{1}{n}\sum v_i$ in $\Delta$ is invariant under the action of  $Z(SU(n))$ on $SU(n)$.
\end{remark}

\part{Foreground}
\chapter{Quasi-Hamiltonian $G$-Spaces} \label{chapter:quasi}

The theory of quasi-Hamiltonian $G$-spaces, introduced in 1998 by Alekseev, Malkin and Meinrenken \cite{AMM}, provides a natural framework in which to study the moduli space of flat $G$-bundles over a compact oriented surface (see Section \ref{eg:modulispace} and the following chapter).  This chapter reviews the definition of a quasi-Hamiltonian $G$-space, and surveys some key examples and theorems of \cite{AMM}.

\section{Definitions} 
\label{sec:quasi-intro}

Let $G$ be a compact Lie group, and $\mathfrak{g}$ its corresponding Lie algebra, with a chosen $Ad$-invariant positive definite inner product $(-,-)$, which may be used to identify $\mathfrak{g}\cong\mathfrak{g}^*$.  Let $\eta \in \Omega^3(G)$ be the canonical $3$-form on $G$. (See Section \ref{sec:coholiegroup}.)

On any $G$-manifold $X$, let $\xi^\sharp$, where $\xi\in \mathfrak{g}$, denote the generating vector field on $X$ induced by the action of $G$, following the convention:
\[ \xi^\sharp(x) =\frac{d}{dt}\Big|_{t=0}\exp(-t\xi)\cdot x.  \]
The group $G$ itself is considered as a $G$-manifold, acting by conjugation. 

\begin{defn}  
\label{defn:qhspace} A quasi-Hamiltonian $G$-space
is a triple $(M,\omega,\phi)$ consisting of a $G$-manifold $M$,
an invariant $2$-form $\omega$ in $\Omega^2(M)^G$, and an equivariant map
$\phi:M\to G$ (called the moment map) satisfying the following three properties:
\begin{itemize}
\item[$\mathrm{Q}1$.] $d\omega +\phi^*\eta =0$

\item[$\mathrm{Q}2$.] $\iota_{\xi^\sharp}\omega = -\half
\phi^*(\theta^L+\theta^R,\xi)$ for all $\xi\in\mathfrak{g}$

\item[$\mathrm{Q}3$.] At every point $p\in M$,
$\ker\omega_p \cap \ker{\mathrm{d}\phi|_p} = \{0\}$

\end{itemize}
\end{defn}

Notice that condition Q1 of Definition \ref{defn:qhspace} says
that the pair $(\omega,\eta)$ defines a cocycle of dimension $3$
in $\Omega^*(\phi)$, the algebraic mapping cone of
$\phi^*:\Omega^*(G)\to\Omega^*(M)$, and hence determines a
cohomology class $[(\omega,\eta)]\in H^3(\phi\, ;\R)$.
In fact, conditions Q1 and Q2 together can be re-expressed in
terms of the Cartan model for equivariant differential forms on
$M$ (see \cite{GS}) . Specifically, Q1 and Q2 may be replaced by the single
relation
\[d_G\omega+\phi^*\eta_G =0 \] where $\omega$ is viewed as an
equivariant differential form and $\eta_G$ is the equivariant extension given by
$\eta_G(\xi)=\eta+\half(\theta^L+\theta^R,\xi)$ in the Cartan
model. Therefore,  Q1 and Q2  give an
equivariant cohomology class $[(\omega,\eta_G)]\in H^3_G(\phi\,
;\R)$.  This is the salient feature of the above definition, for the purposes of this work. 

Before proceeding further, it may be useful to contrast the above definition
with that of a Hamiltonian $G$-space, the classical counterpart of a quasi-Hamiltonian $G$-space.  To that end, recall the following definition.

\begin{defn} 
\label{defn:hspace} A Hamiltonian $G$-space is a triple $(M, \sigma, \Phi)$ consisting of a $G$-manifold $M$, an invariant $2$-form $\sigma$ in $\Omega^2(M)^G$, and an equivariant map $\Phi:M\to \mathfrak{g}^*$, satisfying the following three properties:
\begin{itemize}
\item[$\mathrm{H}1$.] $d\sigma =0$
\item[$\mathrm{H}2$.] $ \iota_{\xi^\sharp}\sigma =-d(\Phi,\xi)$ for all $\xi \in \mathfrak{g}$
\item[$\mathrm{H}3$.] $\sigma$ is non-degenerate
\end{itemize}
\end{defn}

Despite not playing a major role in the chapters that follow, condition Q3 in Definition \ref{defn:qhspace} highlights a significant difference between quasi-Hamiltonian $G$-spaces and their classical counterparts.  In particular, the 2-form $\omega$ of a quasi-Hamiltonian $G$-space is sometimes degenerate, and hence not symplectic.  In fact, conditions Q1, Q2, and Q3 imply that $\ker \omega_m=\{ \xi^\sharp(m) \,|\, Ad_{\phi(m)}\xi=-\xi\}$ \cite{AMM}. 
Nevertheless, Example \ref{eg:exp} shows that a Hamiltonian $G$-space naturally provides all of the ingredients in Definition \ref{defn:qhspace}, except possibly condition Q3.  Also, both definitions single out cohomology classes.  Indeed, condition H1 of Definition \ref{defn:hspace} says that the 2-form $\sigma$ is closed, and hence defines a cohomology class $[\sigma]\in H^2(M;\R)$.  When coupled with condition H2, this gives a closed equivariant extension $\sigma_G(\xi)=\sigma+(\Phi,\xi)$ in the Cartan model, and hence a cohomology class $[\sigma_G]\in H^2_G(M;\R)$.

In subsequent chapters, the pre-quantization of quasi-Hamiltonian $G$-spaces will be addressed, and the choice of inner product on $\mathfrak{g}$ will be significant.  The definition below is used to highlight this choice.  Recall that when $G$ is simple, any invariant inner product on  $\mathfrak{g}$ is necessarily a multiple of the basic inner product $B$ (see Remarks \ref{remark:innerproducts} and \ref{remark:innerproducts} in Section \ref{sec:liethy}).

\begin{defn} 
\label{defn:qHlevels}
Let $G$ be a simple compact Lie group. Suppose that the invariant inner product in Definition \ref{defn:qhspace} is $l B$, where $B$ is the basic inner product and $l>0$. The level of the quasi-Hamiltonian $G$-space $(M,\omega,\phi)$ is the number $l$.
\end{defn}

Note that if the underlying level of a quasi-Hamiltonian $G$-space is $l$, then the canonical $3$-form $\eta=l\eta_1$, where $\eta_1=\frac{1}{12}B(\theta^L, [\theta^L,\theta^L])$ is the canonical $3$-form corresponding to the basic inner product.  When $G$ is simply connected, it is well known (see \cite{PS}, for example) that $[\eta_1]\in H^3(G;\R)$ is integral, and that its integral lift generates $H^3(G;\Z)\cong \Z$. Therefore, the level $l$ may also be interpreted as the cohomology class $[\eta] \in H^3(G;\R)\cong \R$ via the isomorphism that identifies $[\eta_1]$ with $1\in\R$ so that the coefficient homomorphism is just the standard inclusion $\Z\to\R$.

\section{Two fundamental theorems}

In order to describe some of the examples in the following section, two fundamental theorems about quasi-Hamiltonian $G$-spaces are reviewed.  The interested reader may wish to consult \cite{AMM} for proofs. 

The first theorem describes how  the product of two quasi-Hamiltonian $G$-spaces is naturally a quasi-Hamiltonian $G$-space.  It is used in the description of the moduli space of flat $G$-bundles over a compact oriented surface in Section \ref{eg:modulispace}.

\begin{thm}[Fusion] \cite{AMM}
\label{fact:fusion} Let $(M_1,\omega_1,\phi_1)$ and $(M_2,\omega_2,\phi_2)$ be quasi-Hamiltonian $G$-spaces. Then $M_1\times M_2$ is a quasi-Hamiltonian $G$-space, with diagonal $G$-action, invariant $2$-form $\omega=\pr_1^*\omega_1 + \pr_2^*\omega_2 +
\half(\pr_1^*\phi_1^*\theta^L,\pr_2^*\phi_2^*\theta^R)$, and group-valued moment map $\phi=\mu\circ(\phi_1\times\phi_2)$, where $\mu:G^2\to G$ is group multiplication. This quasi-Hamiltonian $G$-space is called the fusion product of $M_1$ and $M_2$, and is denoted $(M_1\circledast M_2,\omega,\phi)$.
\end{thm}

The following theorem is the quasi-Hamiltonian analogue of Meyer-Marsden-Weinstein reduction in symplectic geometry.

\begin{thm}[Reduction] \cite{AMM}
\label{fact:reduction} Let $(M,\omega,\phi)$ be a quasi-Hamiltonian $G$-space.
For any $h\in G$ that is a regular value of $\phi$, the centralizer $Z_h$ of $h$ acts locally freely on the level set $\phi^{-1}(h)$, and the restriction of $\omega$ to $\phi^{-1}(h)$ descends to a symplectic form on the orbifold $\phi^{-1}(h)/Z_h$. In particular, if the identity $e$ is a
regular value of $\phi$, and $j:\phi^{-1}(e)\to M$ denotes the inclusion, then $j^*\omega$ descends to a symplectic form
$\omega_{red}$ on $M/\!/G:=\phi^{-1}(e)/G$, called the symplectic quotient of $M$.
\end{thm}

\section{Examples}

This section reviews a short list of examples of quasi-Hamiltonian $G$-spaces.  The verification that these examples satisfy the conditions of Definition \ref{defn:qhspace} appears in \cite{AMM}.  

\begin{eg} 
\label{eg:conjclass} Conjugacy classes in the Lie group $G$
\end{eg}
In analogy with co-adjoint orbits $\mathcal{O}\subset \mathfrak{g}^*$, central examples in the theory of  Hamiltonian $G$-spaces, conjugacy classes $\mathcal{C}\subset G$ are  important examples of  quasi-Hamiltonian $G$-spaces.  (Compare with Example \ref{eg:coadjointorbits} in Chapter \ref{chapter:loopgroups}.)  The group-valued moment map in this case is the inclusion $\iota:\mathcal{C}\hookrightarrow G$.  Since the tangent space at any point $g\in \mathcal{C}$ is spanned by generating vector fields, the 2-form $\omega$ can be described by the formula:
$$ \omega_g(\zeta^\sharp(g), \xi^\sharp(g))=\half \left( (\xi, Ad_g\zeta) - (\zeta, Ad_g\xi) \right). 
$$ \hfill $\qed$

\begin{eg} 
\label{eg:exp} Exponential of a Hamiltonian $G$-space
\end{eg}

As mentioned previously, if $(M,\sigma,\Phi)$ is a Hamiltonian $G$-space, there is a natural 2-form $\omega$ on $M$ such that the triple $(M,\omega,\phi)$, where $\phi=\exp\circ \,\Phi$, satisfies conditions Q1 and Q2 of Definition \ref{defn:qhspace}.  Specifically, $\omega=\sigma + \Phi^*\varpi$, where $\varpi$ is a 2-form  on $\mathfrak{g}$ that is $G$-invariant and satisfies $d\varpi+\exp^*\eta=0$ (see \cite{AMM} for details).  If in addition, the differential $\mathrm{d}(\exp)$ is a bijection on the image $\Phi(M)$, then Q3 is satisfied as well, and $(M,\omega,\phi)$ is a quasi-Hamiltonian $G$-space.  

As remarked in \cite{AMM}, it is interesting to note that if $\exp$ is invertible on some neighborhood containing $\phi(M)$, then one may reverse the above recipe, and produce a Hamiltonian $G$-space from a quasi-Hamiltonian $G$-space. \hfill $\qed$
\\

\begin{eg} 
\label{eg:double} The double $\mathbf{D}(G)$
\end{eg}

The double $\mathbf{D}(G):=G \times G$ of a compact Lie group $G$ is an important example of a quasi-Hamiltonian $G$-space.  As the next section shows, it is the basic building block of the moduli space of flat $G$-bundles over a closed oriented surface, which is a motivating example in the theory of quasi-Hamiltonian $G$-spaces.  

The group $G$ acts on $\mathbf{D}(G)$ diagonally (i.e., by conjugation on each factor).
The $G$-invariant 2-form $\omega$ is a bit complicated to describe. For the curious reader, here it is:
$$ \omega=\half \left( (\pr_1^*\theta^L,\pr_2^*\theta^R) + (\pr_1^*\theta^R, \pr_2^*\theta^L) -  (\mu^*\theta^L, (\mu\circ T)^*\theta^R) \right),
$$
where $\mu:G\times G \to G$ is group multiplication, and $T:G\times G \to G\times G$ is the map that switches factors.  The moment map $\phi:G\times G \to G$ is the commutator map $\phi(a,b)=aba^{-1}b^{-1}$.

 Let $\tilde{G}$ denote the universal covering group of $G$, with covering homomorphism $\pi:\tilde{G} \to G$.  Since $\ker \pi$ is contained in the center $Z(\tilde{G})$ (see Section \ref{sec:liegroupsprelim}), the moment map $\phi$ lifts canonically to a moment map $\tilde{\phi}:G\times G \to \tilde{G}$. Therefore, $\mathbf{D}(G)$ can also be viewed as a quasi-Hamiltonian $\tilde{G}$-space.  (Compare with Sections \ref{sec:nonsimplyconnectedG}, and \ref{sec:compofmoduli}.) \hfill $\qed$
 \\
 
 \begin{eg} 
 The spinning $2n$-sphere 
 \end{eg}
 
 The sphere $S^{2n}$ is a quasi-Hamiltonian $U(n)$-space, where $U(n)$ acts via the inclusions $U(n) \hookrightarrow O(2n) \subset O(2n+1)$.  A description of the moment map $\phi:S^{2n}\to U(n)$ and $2$-form $\omega$ may be found in \cite{HJS}.
 
 \section{Moduli spaces of flat $G$-bundles} 
 \label{eg:modulispace}

An elegant outgrowth of the theory of quasi-Hamiltonian $G$-spaces is an explicit finite dimensional description of Atiyah-Bott's symplectic structure (see \cite{AB}, and \cite{AMM}) on the moduli space of flat $G$-bundles over a compact oriented surface $\Sigma$.  Such a description comes from realizing the moduli space as a symplectic quotient of a quasi-Hamiltonian $G$-space.  This is a motivating example for much of this work, and will be reviewed in this section.

\subsection{Flat $G$-bundles}

Before diving into the quasi-Hamiltonian descriptions, it may be useful to quickly review the objects of study, and establish some notation.

Let $G$ be a compact  Lie group, with Lie algebra $\mathfrak{g}$.  A \emph{flat $G$-bundle (over a connected base $B$)} is a pair $(P,\theta)$ consisting of a principal $G$-bundle $\pi:P\to B$ equipped with a flat connection $\theta \in \Omega^1(P;\mathfrak{g})$.  Two flat $G$-bundles $(P_1, \theta_1)$ and $(P_2, \theta_2)$ are equivalent if there is a bundle isomorphism $f: P_1 \to P_2$ (a $G$-equivariant map covering the identity map on $B$) such that $f^*\theta_2=\theta_1$.

Fix a base point $b\in B$.
A \emph{based flat $G$-bundle}  is a triple  $(P,\theta,p)$ consisting of a flat $G$-bundle $(P,\theta)$ and a chosen base point $p$ in $P|_b$, the fiber over $b$.  Two based flat $G$-bundles $(P_1, \theta_1, p_1)$, and $(P_2, \theta_2, p_2)$ are equivalent if there is a bundle isomorphism $f:P_1\to P_2$ such that $f^*\theta_2=\theta_1$, and $p_2=f(p_1)$. 

There is a well known correspondence (see \cite{AB}, or \cite{Mor2}),
\begin{equation} 
\label{eqn:holonomy}
\left\{ \begin{array}{c}
\text{equivalence classes of} \\
\text{ based flat $G$-bundles}
\end{array} \right\} \stackrel{\text{1:1}}{\longleftrightarrow}
\mathrm{Hom}(\pi_1(B,b), G)
\end{equation}
which identifies the equivalence class of a based flat $G$-bundle $(P,\theta, p)$ with the holonomy homomorphism $\Hol: \pi_1(B,b) \to G$ determined by the connection $\theta$ and the base point $p \in P$.  (Recall that the holonomy $\Hol([\alpha])\in G$  of $[\alpha]\in \pi_1(B,b)$ is defined by the relation $\tilde\alpha(1)=p\cdot (\Hol([\alpha]))^{-1}$, where  $\tilde\alpha$ is a horizontal lift of $\alpha$ satisfying $\tilde\alpha(0)=p$.)

Furthermore, there is a $G$-action on each of the above. Namely, for $g\in G$, the action on the  spaces appearing in (\ref{eqn:holonomy}) is given by
$$ g\cdot[(P,\theta,p)]=[(P,\theta, p\cdot g^{-1})] \quad\text{and} \quad g\cdot \rho =Ad_g\circ \rho.
$$
The above correspondence is $G$-equivariant, and thus descends to a correspondence:
$$
\left\{ \begin{array}{c}
\text{equivalence classes} \\
\text{ of flat $G$-bundles}
\end{array} \right\} \stackrel{\text{1:1}}{\longleftrightarrow}
\mathrm{Hom}(\pi_1(B,b), G) /G.
$$
One may endow topologies on the sets appearing in (\ref{eqn:holonomy}), so that the correspondence becomes a homeomorphism (see \cite{AB}, and \cite{AMM} for details).  For future reference, let
$$M_G(B,b)=\left\{ \begin{array}{c}
\text{equivalence classes of} \\
\text{ based flat $G$-bundles}
\end{array} \right\}, \quad \text{and} \quad 
M_G(B)=\left\{ \begin{array}{c}
\text{equivalence classes} \\
\text{ of flat $G$-bundles}
\end{array} \right\}.
$$

\subsection{Flat $G$-bundles over compact oriented surfaces}

Now suppose that  the base $B$ is a compact, oriented surface $\Sigma$ of genus $g$ with $r\geq 0$ boundary circles $v_1,\ldots, v_r$.   (To emphasize the genus $g$ and the number of boundary components $r$, write $\Sigma=\Sigma^g_r$.)  In order to make use of the correspondence (\ref{eqn:holonomy}), it will be useful to have a description of the fundamental group $\pi_1(\Sigma,b)$, and to recall some elementary facts regarding the topology of compact, oriented surfaces.

To that end, recall that  $\Sigma$ is  homeomorphic to a  polygon whose edges have been identified in a certain way.  This is well known, and will be reviewed briefly.  (See \cite{Ma} for details.)

When $r=0$, the edges of a $4g$-gon can be identified as in Figure \ref{fig:polygon1}, where the oriented edges of the same name are identified. Consequently, the edges represent loops in $\Sigma$, and can therefore be viewed as elements in $\pi_1(\Sigma,b)$, where $b \in \Sigma$ is the point corresponding to the vertices of the polygon.

 \begin{figure}[ht]
   \centering
   \includegraphics[totalheight=0.25\textheight]{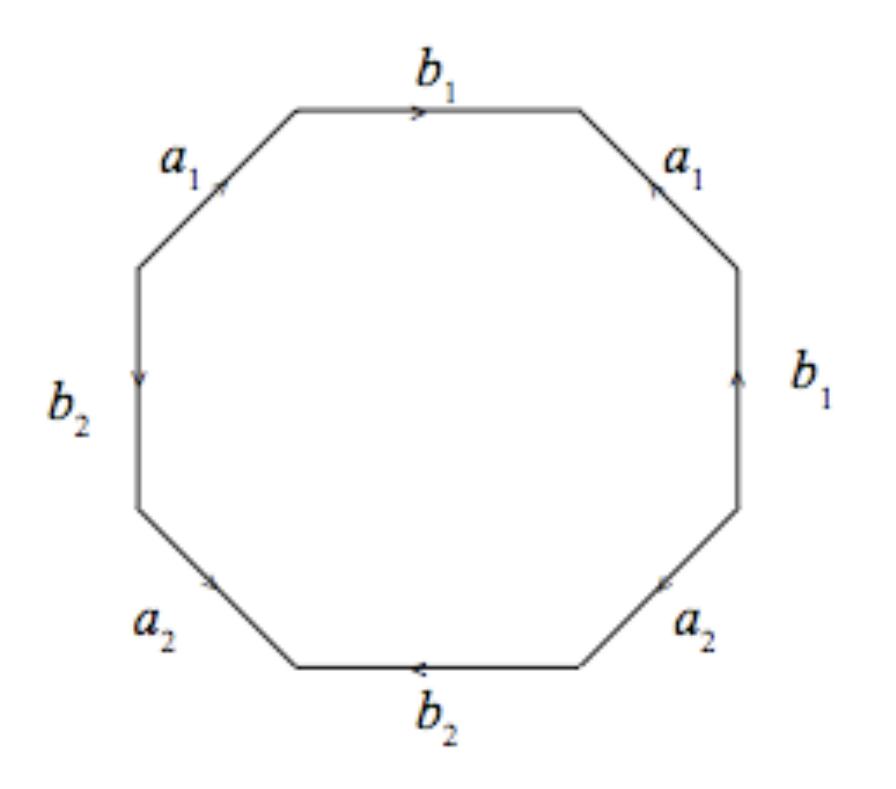} 
   \caption{$\Sigma^2_0$ as the quotient of a polygon}
   \label{fig:polygon1}
\end{figure}

The fundamental group $\pi_1(\Sigma,b)$ can then be easily calculated (using Van Kampen's Theorem, for example).  In particular, when $\partial \Sigma = \emptyset$, the fundamental group is freely generated with a single relation:
$$ \pi_1(\Sigma^g_0,b) \cong \langle a_1, b_1, \ldots, a_g, b_g \,|\, \prod [a_i, b_i]=e \rangle
$$
where $[x,y]$ denotes the commutator $xyx^{-1}y^{-1}$, and $e$ is the group identity.

Note that such a description may equivalently be viewed as a $CW$-structure for $\Sigma$, where $\Sigma$ is obtained by attaching a single 2-cell to a bouquet of circles.  Indeed, the 2-cell is the $4g$ sided polygon, and the above identifications define the attaching map from the boundary of this polygon to the wedge of $2g$ circles $a_1, b_1, \ldots, a_g, b_g$.  

When $r \geq 1$, the above picture may be modified by considering the same $4g$-gon, but with $r$ small discs removed, as in Figure \ref{fig:polygon2}.  The  boundary of $\Sigma$ is then the union of the boundaries $v_1, \ldots, v_r$ of the removed discs.   
 
 \begin{figure}[ht]
   \centering
   \subfigure[with boundary circles]{
   \includegraphics[totalheight=0.25\textheight]{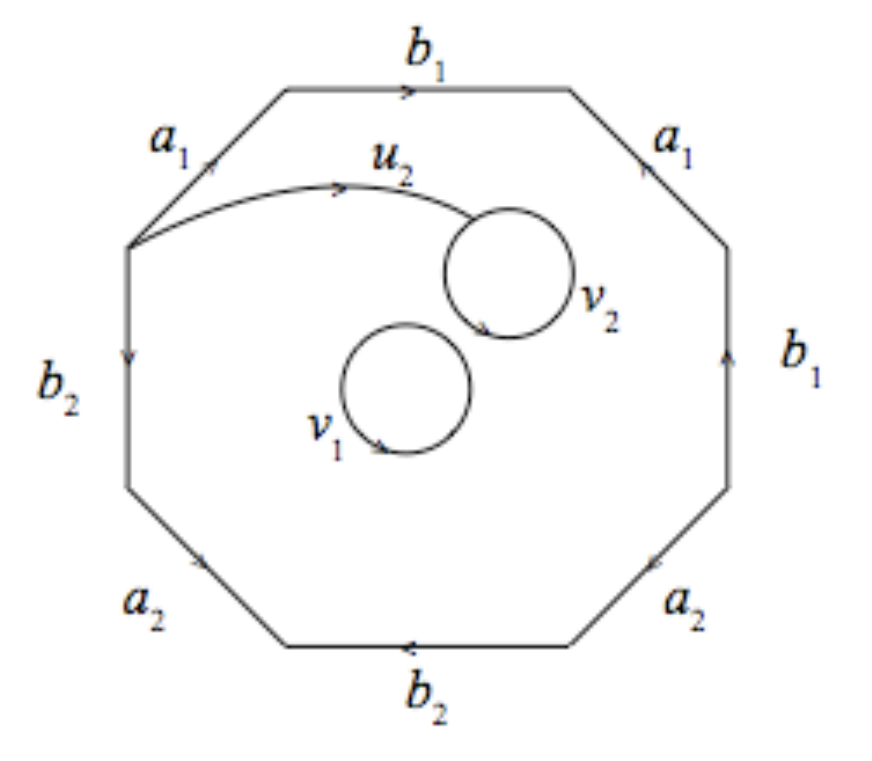} \label{fig:polygon2}}
   \subfigure[once-punctured polygon]{
   \includegraphics[totalheight=0.25\textheight]{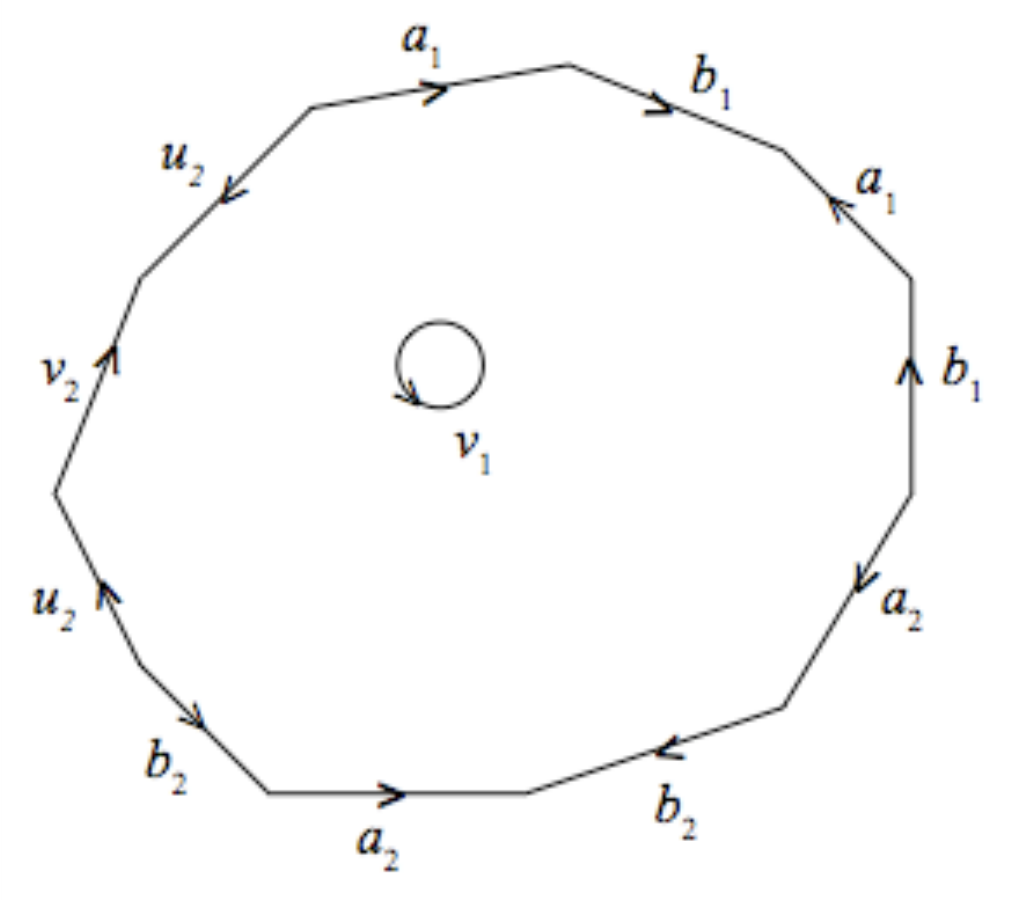} \label{fig:polygon3}}
   \caption{$\Sigma^2_2$ as the quotient of a polygon}
\end{figure}
 
 By introducing paths $u_2, \ldots, u_r$ from $b$ to chosen points on $v_2, \ldots, v_r$, one may equivalently realize $\Sigma$ as the quotient of the $4g+3(r-1)$ sided once-punctured polygon, as in Figure \ref{fig:polygon3}. To see this, note that either identification describes the same $CW$-structure on $\Sigma$: namely, attaching a single 2-cell (via the same attaching map) to the 1-skeleton $X^{(1)}$, which can be described as follows. $X^{(1)}=Y\coprod v_1$, where $Y$ consists of a bouquet of $2g$ circles $a_1, \ldots, b_g$ attached to $r-1$ disjoint circles $v_2, \ldots, v_r$ by $r-1$ contractible line segments $u_2, \ldots, u_r$. (See Figure \ref{fig:bouquet}.)

\begin{figure}[ht]
   \centering
   \includegraphics[totalheight=0.25\textheight]{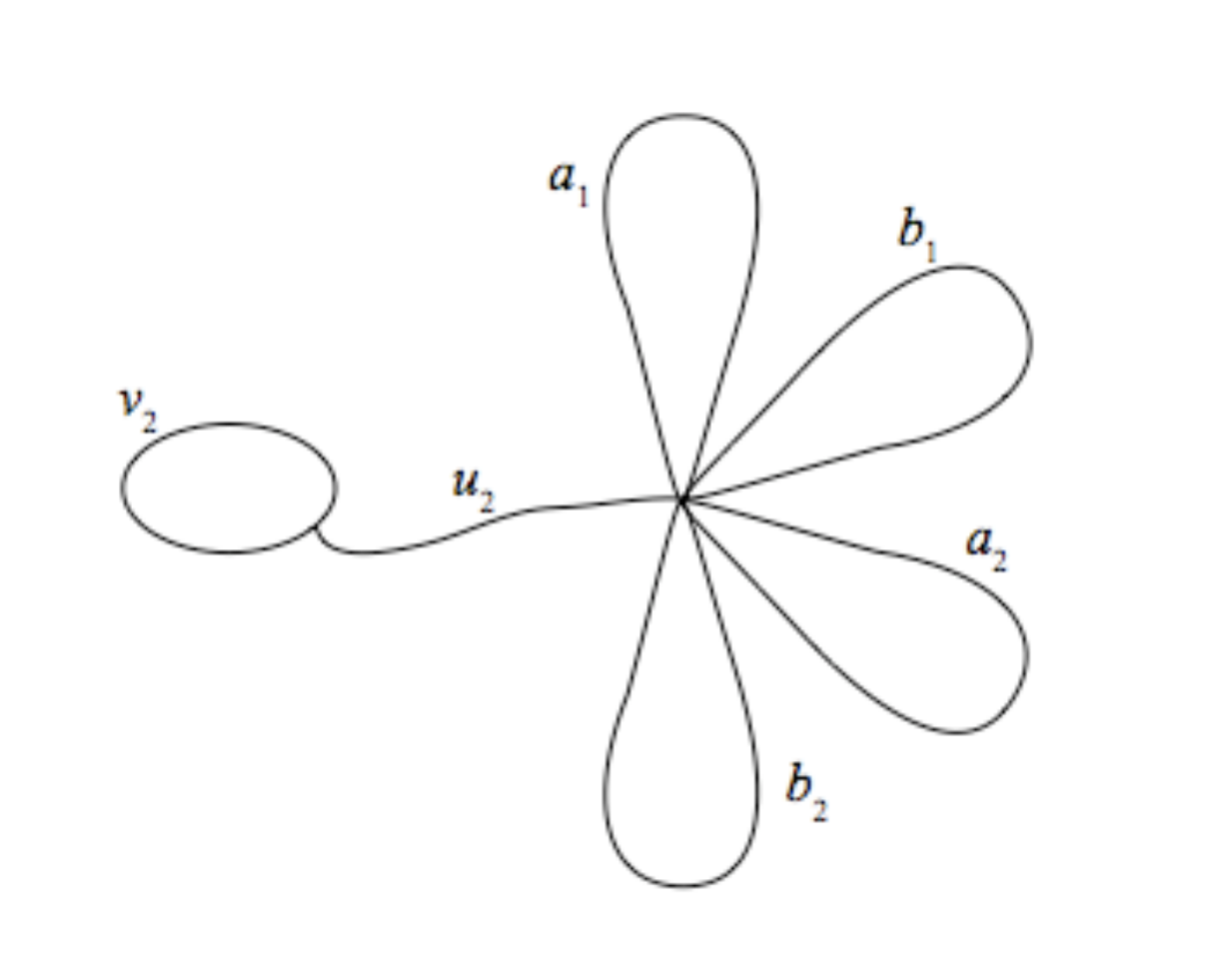} 
   \caption{The retract $Y\subset\Sigma^2_2$ }
   \label{fig:bouquet}
\end{figure}

In fact, from the description of $\Sigma$ as a quotient of a once-punctured polygon,  it is easy to see that $\Sigma$ retracts onto $Y$, which in turn is homotopy equivalent to a bouquet of $2g+r-1$ circles.  Indeed, the $4g+3(r-1)$ sided once-punctured polygon (as in Figure \ref{fig:polygon3}) clearly retracts onto its outer edges, which after identifications is $Y$. 

Therefore, when  $r\geq1$, 
$$
\pi_1(\Sigma^g_r,b)\cong \langle a_1, b_1, \ldots, a_g, b_g, v_2, \ldots, v_r \rangle
$$ 
where $v_j$ really means the boundary component of the same name composed with some path to and from the base point $b$. This choice of generators induces the identification $\mathrm{Hom}(\pi_1(\Sigma,b),G)\cong G^{2g+r-1}$, by sending the homomorphism $\rho:\pi_1(\Sigma,b)\to G$ to 
$$
\rho \mapsto (\rho(a_1), \rho(b_1), \ldots, \rho(a_g), \rho(b_g), \rho(v_2), \ldots, \rho(v_r)).
$$  
The correspondence (\ref{eqn:holonomy}) now reads $M_G(\Sigma,b)\cong G^{2g+r-1}$.

 Now suppose that each of the based flat $G$-bundles in $M_G(\Sigma,b)$ is to have prescribed holonomies $c_2, \ldots, c_r \in G$ around the generators $v_2, \ldots, v_r$ respectively.  Then,
$$
\left\{ \begin{array}{c}
\text{equivalence classes of based} \\
\text{ flat $G$-bundles over $\Sigma$ with} \\
\text{ holonomies $\Hol(v_j)=c_j$}
\end{array} \right\}  \cong G^{2g} \times \{c_2\} \times \cdots \times \{c_r\}. 
$$ 
Letting the $c_j$'s vary within their conjugacy class $\mathcal{C}_j$ gives
$$
M_G(\Sigma,b;\mathcal{C}_2,\ldots, \mathcal{C}_r)  :=\left\{ \begin{array}{c}
\text{equivalence classes of based} \\
\text{ flat $G$-bundles over $\Sigma$ with} \\
\text{ holonomies $\Hol(v_j)\in \mathcal{C}_j$}
\end{array} \right\} 
\cong G^{2g} \times \mathcal{C}_2 \times \cdots \times \mathcal{C}_r 
$$
which is a quasi-Hamiltonian $G$-space.  Specifically, it is shown in \cite{AMM} that the above identification  realizes $M_G(\Sigma, b; \C_2, \ldots, \C_r)$ as the fusion product of $g$ copies of the double $\mathbf{D}(G)$ with the conjugacy classes $\C_2, \ldots, \C_r$.  

To interpret the geometric meaning of  the moment map, $\varphi:M_G(\Sigma, b; \C_2, \ldots, \C_r) \to G$, recall that the boundary circle $v_1$ satisfies the homotopy equivalence 
$$
v_1^{-1} \approx [a_1,b_1]\cdots [a_g,b_g] v_2, \cdots v_r;
$$ 
therefore, the moment map $\varphi:G^{2g} \times \C_2 \times \cdots \times \C_r   \to G$,
$$
(a_1, b_1, \ldots, a_g, b_g, c_2, \ldots, c_r) \mapsto \prod[a_i,b_i]c_2\cdots c_r,
$$
 induced by the fusion product sends a based flat $G$-bundle with prescribed holonomies to its holonomy around the boundary component $v_1^{-1}$.

Therefore, given a conjugacy class $\C_1\subset G$, $\varphi^{-1}(\C_1)$ consists of the based flat $G$-bundles whose holonomy along each of the boundary circles $v_1, v_2 \ldots v_r$ lies in $\C_1^-,\C_2 \ldots, \C_r$, respectively, where $\C_1^-$ denotes the image of the conjugacy class $\C_1$ under the inversion map $g\mapsto g^{-1}$ in $G$.

Alternatively, (after re-labelling)
$$ M_G(\Sigma,b;\mathcal{C}_1,\ldots, \mathcal{C}_r) \cong \phi^{-1}(e),$$
where $\phi:G^{2g}\times \mathcal{C}_1 \times \cdots \times \mathcal{C}_r \to G$ is the map 
$$
(a_1, b_1, \ldots, a_g, b_g, c_1, \ldots, c_r) \mapsto \prod[a_i,b_i]c_1\cdots c_r. $$

Let $M_G(\Sigma; \mathcal{C}_1,\ldots, \mathcal{C}_r)$ be the set of equivalence classes of flat $G$-bundles with prescribed holonomies along the boundary components.  Then,
$$M_G(\Sigma; \mathcal{C}_1,\ldots, \mathcal{C}_r) \cong \phi^{-1}(e)/G,$$
a symplectic quotient of a quasi-Hamiltonian $G$-space.  In \cite{AMM}, Alekseev, Malkin, and Meinrenken show that the symplectic structure inherited by the above description coincides with that of Atiyah-Bott in \cite{AB}, where $M_G(\Sigma; \C_1, \ldots, \C_r)$ is described as an ordinary symplectic quotient of an infinite dimensional symplectic space equipped with a Hamiltonian action of an infinite dimensional Lie group. 

\subsection{The moduli space $M_G(\Sigma^g_1,b)\cong G^{2g}$}
\label{sec:compofmoduli}

Of particular importance in this work is the moduli space $M_G(\Sigma^g_1,b)\cong G^{2g}$, which will now be examined further. (In this section, $\Sigma=\Sigma^g_1$.) Let the compact Lie group $G$ be simple, with universal covering homomorphism $\pi:\tilde{G} \to G$.  The kernel $Z:=\ker \pi$ is a central subgroup of $\tilde{G}$, which is also finite (since $G$ is simple).

As in Example \ref{eg:double}, the moment map $\phi:G^{2g}\to G$ lifts canonically to the universal cover $\tilde\phi:G^{2g} \to \tilde{G}$. Since $\phi$ measures the holonomy along the boundary $\partial\Sigma$ of a given based flat $G$-bundle over $\Sigma$, the lift $\tilde\phi$ will be referred to as the \emph{lifted holonomy} along the boundary of $\Sigma$.  It follows that $M_G(\Sigma,b)$ can be viewed as a quasi-Hamiltonian $\tilde{G}$-space.  Furthermore, the symplectic quotient $\phi^{-1}(e)/G$ may be expressed as
\begin{equation}
\label{eq:decomp}
 \phi^{-1}(e)/G =\coprod_{z\in Z} \tilde\phi^{-1}(z)/\tilde{G}.
\end{equation}
The fibers $\tilde\phi^{-1}(z)$ are connected, by Theorem 7.2 in \cite{AMM}; therefore, the above equation expresses the connected components of the symplectic quotient $\phi^{-1}(e)/G\cong M_G(\Sigma; \{e\})$.  

To better understand the decomposition (\ref{eq:decomp}), observe that $M_G(\Sigma; \{e\}) \cong M_G(\hat\Sigma)$, where $\hat\Sigma$ denotes the closed oriented surface obtained from $\Sigma$ by attaching a disc $D$ to $\Sigma$ with a degree 1 attaching map $\partial D \toby{=}\partial \Sigma \subset \Sigma$.  Also, recall that there is a bijective correspondence between principal $G$-bundles over $\hat\Sigma$ and $\pi_1(G)\cong Z$: every principal $G$-bundle over $\hat\Sigma$ can be constructed by gluing together trivial bundles over both $\Sigma$ and $D$ with a transition function $S^1=\Sigma\cap D \to G$. (Note that since $\Sigma$ is homotopy equivalent to a bouquet of circles, every principal $G$-bundle over $\Sigma$ is trivializable.)  Proposition \ref{prop:bundletype} below shows that the decomposition (\ref{eq:decomp}) describes the connected components of $M_G(\hat\Sigma)$ in terms of the bundle types enumerated by $Z$. That is,  $\tilde\phi^{-1}(z)/\tilde{G}$ consists of isomorphism classes of  flat $G$-bundles $(P, \theta)$ over $\hat\Sigma$ where $P\to \hat\Sigma$ is in the isomorphism class of principal $G$-bundles corresponding to  $z\in Z$.

\begin{prop} \label{prop:bundletype}
Let $j:\Sigma \hookrightarrow \hat\Sigma$ denote inclusion, and suppose that $(P,\theta)$ is a flat $G$-bundle over $\hat\Sigma$, where $P\to \hat\Sigma$ is a principal $G$-bundle corresponding to $z\in Z\cong \pi_1(G)$.  Then under the identification (\ref{eqn:holonomy}), $\tilde\phi([(j^*P, \hat{j}^*\theta)])=z$, where $\hat{j}:j^*P \hookrightarrow P$ is the induced inclusion.
\end{prop}
\begin{proof}
Since $P\to \hat\Sigma$ corresponds to $z\in Z$, $P$ is isomorphic to a $G$-bundle constructed by gluing together trivial $G$-bundles over $\Sigma$ and over a disc $D$ by a transition function $f:S^1 = \Sigma \cap D \to G$, where $[f]=z\in Z\cong \pi_1(G)$.  Without loss of generality, assume $j^*P=\Sigma \times G$, and $i^*P=D\times G$, where $i:D\hookrightarrow \hat\Sigma$ denotes inclusion.  Note that since $D$ is simply connected, the correspondence (\ref{eqn:holonomy}) shows that there is only one flat $G$-bundle over $D$ up to isomorphism; therefore, if $\hat{i}:i^*P\hookrightarrow P$ is the induced inclusion, then it may also be assumed that  $\hat{i}^*\theta$ is the trivial connection.  Let $\Hol:\pi_1(\Sigma) \to G$ be the resulting holonomy homomorphism of the flat $G$-bundle $(\Sigma \times G, \hat{j}^*\theta)$.

In order to describe the lifted holonomy map $\tilde\phi$ geometrically, recall how it is defined:  $\tilde\phi(a_1, b_1, \ldots, a_g, b_g)=\prod[\tilde{a}_i, \tilde{b}_i]$, where $\tilde{a}_i$, $\tilde{b}_i$ are pre-images of $a_i$, $b_i$ in $\tilde{G}$.  Therefore, under the identification (\ref{eqn:holonomy}), this corresponds to finding a flat $\tilde{G}$-bundle over $\Sigma$ whose holonomy homomorphism $\widetilde\Hol: \pi_1(\Sigma) \to \tilde{G}$ satisfies $\pi \circ \widetilde\Hol = \Hol$. For that purpose, consider the flat $\tilde{G}$-bundle $\Sigma \times \tilde{G} \to \Sigma$ with flat connection $(1\times \pi)^*j^*\theta$.   To check that $\pi \circ \widetilde\Hol = \Hol$, choose a loop $\alpha$ in $\Sigma$ based at $b$, and let $\tilde\alpha=(\alpha, \tilde\alpha_2)$ be a horizontal lift in $\Sigma \times \tilde{G}$, with $\tilde\alpha_2(0)=\tilde{e}\in \tilde{G}$.  Then $\widetilde\Hol ([\alpha])=\tilde\alpha_2(1)^{-1}$. Since $(1\times \pi)$ sends horizontal vectors to horizontal vectors, then $(1\times \pi)\circ \tilde\alpha$ is a horizontal lift in $\Sigma \times G$.  Therefore, $\Hol([\alpha])=(\pi\tilde\alpha_2(1))^{-1}=\pi(\widetilde\Hol([\alpha]))$.

Having found a suitable flat $\tilde{G}$-bundle, it now suffices to verify that $\widetilde\Hol(\partial\Sigma)=z$.  To that end, consider a horizontal lift $\tilde\alpha$ of $\alpha=\partial \Sigma$ in $\Sigma \times \tilde{G}$.  Then $\tilde\alpha=(\alpha, \tilde\alpha_2)$ is a path in $\Sigma \times \tilde{G}$ that lifts the horizontal loop $(\alpha, \pi\tilde\alpha_2)$ in $\Sigma \times G$ over $\partial \Sigma$.  As a loop in $D\times G$, the horizontal loop over $\alpha$ must be $(\alpha, e)$ because $D\times G$ is equipped with the trivial connection.  Therefore, $\pi\tilde\alpha_2=f$, which identifies $\tilde\alpha_2$ as a lift of $f$ in $\tilde{G}$, and hence $\widetilde\Hol(\partial\Sigma)=[f]\cdot \tilde{e}$, where $[f]$ is the deck transformation corresponding to the loop $f$.  That is, $\widetilde\Hol(\partial\Sigma)=z$, as required.
\end{proof}

\section{Pullbacks of quasi-Hamiltonian $G$-spaces} 
\label{sec:nonsimplyconnectedG}

Handling quasi-Hamiltonian $G$-spaces when $G$ is not simply connected can sometimes require extra care.  A  useful technique is to introduce an intermediate quasi-Hamiltonian $\tilde{G}$-space, where $\tilde{G}$ is the universal covering group of $G$.  This section reviews this elementary construction.

Let $G$ be a semi-simple compact, connected, Lie group, and $\pi:\tilde{G}\to G$ its universal covering homomorphism.  For a quasi-Hamiltonian $G$-space $(M, \omega, \phi)$, let $P(M)$ be the pullback in the following diagram:

\begin{diagram}
P(M) & \rTo^\varphi & \tilde{G} \\
\dTo>p &  & \dTo>\pi \\
M & \rTo^\phi & G
\end{diagram}

\begin{prop} Let $G$ be a semi-simple, compact, connected Lie group.
If $(M,\omega,\phi)$ is a quasi-Hamiltonian $G$-space, and $P(M)$ is the pullback defined above, then $(P(M), p^*\omega, \varphi)$ is a quasi-Hamiltonian $\tilde{G}$-space.
\end{prop}

\begin{proof} First, note that the $\tilde{G}$-action on $P(M):=\{(m,g) \in M\times \tilde{G}\, |\, \phi(m)=\pi(g)\}$ is given by
$$ g\cdot (m,a) = (\pi(g)\cdot m, gag^{-1}),$$
so that the map $\varphi:P(M) \to \tilde{G}$ is obviously $\tilde{G}$-equivariant.  And since the isomorphism $\pi_*$ identifies the Lie algebra $\mathfrak{g}$ of both $\tilde{G}$ and $G$, the map $p:P(M) \to M$ is $\mathfrak{g}$-equivariant.  Checking that $(P(M), p^*\omega, \varphi)$ satisfies Definition \ref{defn:qhspace} is now easy. (Note that $\tilde{G}$ is compact because $G$ is semi-simple. \cite{MT})

Property Q1 of Definition \ref{defn:qhspace} follows trivially since the pullback diagram commutes.  And since $p$ is $\mathfrak{g}$-equivariant, 
$$\iota_{\xi^\sharp} p^*\omega = p^* \iota_{\xi^\sharp} \omega = -\half p^*\phi^*(\theta^L + \theta^R, \xi) = -\half \varphi^* \pi^* (\theta^L+\theta^R,\xi)= -\half \varphi^*  (\theta^L+\theta^R,\xi),$$
and Q2 holds.  Finally, property Q3 follows because $p$ is a covering projection, and hence a local diffeomorphism. Indeed, suppose $v\in T_xP(M)$ lies in $\ker p^*\omega_x \cap \mathrm{d} \varphi |_x$.  Then $\omega_{p(x)}(p_*v, p_*w)=0$ for all $w\in T_xP(M)$. And since $p$ is a covering projection, that means that $p_*v \in \ker \omega_{p(x)}$. Since the diagram commutes, $p_*v$ also lies in $\ker \mathrm{d}\phi|_{p(x)}$, which implies that $p_*v=0$, and hence $v=0$ since $p$ is a covering projection.
\end{proof}

An important reason for considering the pullback quasi-Hamiltonian $\tilde{G}$-space $P(M)$ is the relationship between the symplectic quotients $M/\!/G$ and $P(M)/\!/\tilde{G}$ when $G$ is simple.  Namely, since $P(M)$ is defined as a pullback, the fibers $\varphi^{-1}(\tilde{e})$ and $\phi^{-1}(e)$ can be identified. Moreover, since central elements in $\tilde{G}$ act trivially on $P(M)$, the orbit spaces $\varphi^{-1}(\tilde{e})/\tilde{G}$ and $\phi^{-1}(e)/G$ can be identified as well.  

Note that in the case where the moment map $\phi:M \to G$ lifts to $\tilde\phi:M \to \tilde{G}$, the pullback $P(M)\cong M \times Z$, and the moment map $\varphi:M\times Z \to \tilde{G}$ takes the form $\varphi(m,z)=\tilde\phi(m)z$.  Therefore $\varphi^{-1}(\tilde{e}) = \tilde\phi^{-1}(Z)$, and if $M$ is connected, then (by Theorem 7.2 in \cite{AMM}) the fibers $\tilde\phi^{-1}(z)$ are connected. Consequently,  the connected components of the symplectic quotient
$$
\phi^{-1}(e)/G = \coprod_{z\in Z} \tilde\phi^{-1}(z)/\tilde{G}
$$
are described as symplectic quotients indexed by $Z$. (Compare with (\ref{eq:decomp}) in Section \ref{sec:compofmoduli}.)

\subsection*{Pulling back over fusion products}

Let $(M_1, \omega_1, \phi_1)$ and $(M_2,\omega_2,\phi_2)$ be quasi-Hamiltonian $G$-spaces.  Recall from Theorem \ref{fact:fusion} that the product $M_1\times M_2$ with group-valued moment map $\phi=\mu\circ(\phi_1\times\phi_2)$ is also a quasi-Hamiltonian $G$-space called the fusion product, sometimes denoted $M_1\circledast M_2$. This section relates the pullback quasi-Hamiltonian $\tilde{G}$-spaces $P(M_1)$, $P(M_2)$, and $P(M_1\times M_2)$ defined above.

Let $Z$ be the kernel of the universal covering homomorphism $\pi:\tilde{G}\to G$, which is necessarily a subgroup of the center $Z(\tilde{G})$ (see Section \ref{sec:liegroupsprelim}).  Then $Z$ acts naturally on each $P(M_i)=\{(m,g)\, | \, \phi_i(m)=\pi(g)\}$ by $z\cdot(m,g)=(m,zg)$, and hence on the product $P(M_1) \times P(M_2)$ according to $z\cdot (m_1,g_1,m_2,g_2) = (m_1,zg_1, m_2, z^{-1}g_2)$. Then there is a natural map 
$$
f:P(M_1\times M_2) \to (P(M_1)\times P(M_2) )/Z
$$ defined as follows. Write $P(M_1\times M_2)=\{(m_1,m_2,g) \, | \, \phi_1(m_1)\phi_2(m_2)=\pi(g)  \}$.
  For $m_1\in M_1$, choose $a\in \tilde{G}$ so that $\pi(a)=\phi_1(m_1)$, and let $f(m_1,m_2,g)=[(m_1,a,m_2,a^{-1}g)]$.  Since $\phi_1(m_1)=\pi(a)$, and $\phi_2(m_2)=\phi_1(m_1)^{-1}\pi(g)=\pi(a)^{-1}\pi(g)=\pi(a^{-1}g)$, the map $f$ is well defined, provided it does not depend on the choice of $a$.  But this is clear, since for $a'=za$ (where $z\in Z$), $(m_1,za,m_2, a^{-1}z^{-1}g)=z\cdot(m_1,a,m_2,a^{-1}g)$.
  
It is easy to see that $f$ is actually a diffeomorphism, with inverse induced by the natural map coming from the universal property of a pullback, since $P(M_1)\times P(M_2)$ fits in the commutative square below.
\begin{diagram}
P(M_1)\times P(M_2) & \rTo^{\mu\circ(\varphi_1\times\varphi_2)} & \tilde{G} \\
\dTo &  & \dTo>\pi \\
M_1\times M_2 & \rTo^\phi & G
\end{diagram}

\begin{remark} \label{remark:whenpullbacksplits}
Recall that if the moment map $\phi_1:M_1\to G$ lifts to $\tilde\phi_1:M_1\to \tilde{G}$, then $P(M_1)\cong M_1\times Z$.  Therefore,
\begin{align*}
P(M_1\times M_2) & \cong (P(M_1)\times P(M_2))/Z \\
&\cong (M_1\times Z \times P(M_2))/Z \\
&\cong M_1\times P(M_2)
\end{align*}
since each $Z$-orbit of $(M_1\times Z \times P(M_2))$ contains exactly one point with coordinate of the form $(-,e,-)$. 
\end{remark}

\chapter{Pre-quantization} \label{chapter:prequantization}

Before analyzing the pre-quantization of the moduli space $M(\Sigma)$, this chapter begins with a discussion of pre-quantization in the realm of quasi-Hamiltonian $G$-spaces.  In this chapter, and in the rest of this work, $G$ is assumed to be a compact, simple Lie group.

\section{Pre-quantization of quasi-Hamiltonian $G$-spaces} \label{sec:quasi}

This section addresses pre-quantization in the context of quasi-Hamiltonian $G$-spaces.   To help motivate the definition of pre-quantization (Definition \ref{preq}) in this setting, the analogous definition for classical Hamiltonian $G$-spaces is briefly recalled \cite{GGK}.  

For a Hamiltonian $G$-space $(M, \sigma, \Phi)$, a pre-quantization is geometrically represented by a certain complex line bundle over $M$ called a pre-quantum line bundle.  More precisely, a pre-quantum line bundle (or $G$-equivariant pre-quantum line bundle) for $(M, \sigma, \Phi)$ is a $G$-equivariant complex line bundle $L\to M$ whose equivariant curvature class is $[\sigma_G]\in H^2_G(M;\R)$.  By Theorem 6.7 in \cite{GGK}, such a line bundle exists if and only if the equivariant cohomology class $[\sigma_G]$ has an integral lift (i.e. a pre-image under the coefficient homomorphism $\iota_\R: H^2_G(M;\Z) \to H^2_G(M;\R)$).  A similar discussion holds for symplectic manifolds $(M,\sigma)$, where a pre-quantization is a complex line bundle over $M$ whose curvature class is $[\sigma]$, which exists if and only if $[\sigma]$ has an integral lift.  

Since isomorphism classes of $G$-equivariant complex line bundles over $M$ are in one-to-one correspondence with $H^2_G(M;\Z)$ (Theorem C.48 in \cite{GGK}), a pre-quantization may be viewed as an integral lift of $[\sigma_G]$, provided the choice of line bundle within a given isomorphism class is unimportant.  (Recall that pre-quantization is an ingredient used to determine the quantization of $(M, \sigma, \Phi)$, which does not depend on the choice of line bundle---see Remark 6.56 in \cite{GGK}.)

\begin{defn} \label{preq}
A pre-quantization of a quasi-Hamiltonian $G$-space $(M,\omega,\phi)$ is an integral lift of $[(\omega,\eta_G)]\in H^3_G(\phi;\R)$.  That is, a pre-quantization is a cohomology class $\alpha\in H^3_G(\phi;\Z)$ satisfying $\iota_\R(\alpha)=[(\omega,\eta_G)]$, where $\iota_\R:H^*_G(\quad;\Z)\to H^*_G(\quad;\R)$ is the coefficient homomorphism.
\end{defn}

There is a geometric interpretation of Definition \ref{preq}, in terms of \emph{relative gerbes} and \emph{quasi-line bundles} that is analogous to the situation for Hamiltonian $G$-spaces reviewed above.  The interested reader may wish to consult \cite{Sh} for details.

Proposition \ref{prop:eq-pqispq} below shows that a pre-quantization may be viewed as an ordinary cohomology class when $G$ is  simply connected.  The analogous statement for Hamiltonian $G$-spaces is also true, and follows directly from Lemma \ref{lemma:simplyG1}.

\begin{lemma} \label{lemma:simplyG1} Let $X$ be some $G$-space, where $G$ is simply connected.  Then the canonical map $H^i_G(X;R) \to H^i(X;R)$ is an isomorphism when $i=1$, and $2$, and is injective when $i=3$ with any coefficient ring $R$. 
\end{lemma}
\begin{proof} Recall that the canonical map $H^i_G(X) \to H^i(X)$  is the  induced homomorphism of the map $p: X \to X_G$, the inclusion of the fiber in the Borel construction $\rho:X_G \to BG$ (see Section \ref{sec:classifyingspaces}). Consider the Serre spectral sequence associated to the fibration $\rho$, with $E_2\cong H^*(X)\otimes H^*(BG)$, converging to $H^*(X_G)$. As $G$ is simply connected, $H^i(BG)=0$ for $i=1,2,3$. Therefore, the first non-zero differential is $d:H^3(X)\to H^4(BG)$, and the result follows. \end{proof}

\begin{prop} \label{prop:eq-pqispq} Suppose $G$ is simply connected, and
$(M,\omega,\phi)$ is a quasi-Hamiltonian $G$-space. Then
$(M,\omega,\phi)$ admits a pre-quantization if and only if
the cohomology class $[(\omega,\eta)]\in H^3(\phi;\R)$ is integral.
\end{prop}
\begin{proof} Easy applications of Lemma
\ref{lemma:simplyG1} with $X=M$ and $X=G$, Proposition
\ref{lemma:simplyG2}, and the five-lemma show that the canonical map
$H^3_G(\phi) \to H^3(\phi)$ is an isomorphism with any coefficient ring. \end{proof}  

The next proposition relates the pre-quantization of a quasi-Hamiltonian $G$-space $(M,\omega,\phi)$ with its symplectic quotient $M/\!/G$.  If it is known that $G$ acts freely on the level set $\phi^{-1}(e)$, then the isomorphism $H^2_G(\phi^{-1}(e)) \cong H^2(M/\!/G)$ guarantees that $[\omega_{red}]$ is integral if and only if $[j^*\omega]$ is integral.  In this case, the next proposition shows that a pre-quantization of a quasi-Hamiltonian $G$-space descends to a pre-quantization of its symplectic quotient.

\begin{prop} \label{prop:qeqpreq=reducedpreq} Let $(M,\omega,\phi)$ be a quasi-Hamiltonian
$G$-space, and suppose the identity $e\in G$ is a regular value for the moment map $\phi$.  If $(M,\omega,\phi)$ admits a pre-quantization, then the cohomology class $[j^*\omega] \in H^2_G(\phi^{-1}(e);\R)$ is integral.
\end{prop}
\begin{proof} There's a canonical map $H^3_G(\phi) \to H^2_G(\phi^{-1}(e))$ (with any coefficients) given by the composition of the map induced by restriction $H_G^3(\phi) \to H_G^3(\phi|_{\phi^{-1}(e)})$ and the projection onto the first summand 
$$H_G^3(\phi|_{\phi^{-1}(e)}) \cong H^2_G(\phi^{-1}(e))\oplus H_G^3(*) \to H^2_G(\phi^{-1}(e)).
$$ In other words, there is a diagram:
\begin{diagram}
H^2_G(M) 	& \rTo & H_G^3(\phi) 	& \rTo & H_G^3(G) \\
	\dTo>{ j^*}		&	   &\dTo	 	&	&\dTo \\
H^2_G(\phi^{-1}(e))	& \rTo & H_G^3(\phi|_{\phi^{-1}(e)})& \rTo & H^3_G(*)\\
\end{diagram}
\noindent in which the bottom row is (canonically) split exact with any coefficients. Since $H^3_G(*;\R)=0$ it suffices to check that (with real coefficients) the middle map sends the relative cohomology class $[(\omega,\eta_G)]$ to $[(j^*\omega,0)]$, which is clear.
\end{proof}

\begin{remark} \label{remark:reduction}  In light of the classification of $G$-equivariant line bundles by $H^2_G(-;\Z)$, the previous proposition has the following geometric interpretation.  If $(M,\omega,\phi)$ is a pre-quantized quasi-Hamiltonian $G$-space  and $e\in G$ is a regular value, then  there is  a $G$-equivariant pre-quantum line bundle over the level set $\phi^{-1}(e)$. And under the additional hypothesis that $G$ acts freely on $\phi^{-1}(e)$,  this $G$-equivariant line bundle descends to a pre-quantum line bundle over the symplectic quotient.  Of course, if $G$ only acts locally-freely on $\phi^{-1}(e)$, one obtains instead an \emph{orbi-bundle} over the symplectic quotient. (The reader may wish to consult \cite{MS} for a more thorough account of pre-quantization of singular spaces.)
 \end{remark}

Proposition \ref{prop:fusion}  shows that the fusion product $M_1\circledast M_2$ of two pre-quantizable quasi-Hamiltonian $G$-spaces is pre-quantizable.  In fact, the proof of the proposition shows that the pre-quantization of $M_1\circledast M_2$ is canonically obtained from the pre-quantizations of $M_1$ and $M_2$.

\begin{prop} \label{prop:fusion} Let $G$ be simply connected. If $(M_1,\omega_1,\phi_1)$ and $(M_2,\omega_2,\phi_2)$
are pre-quantized quasi-Hamiltonian $G$-spaces, then their fusion product $(M_1 \circledast M_2,\omega,\phi)$ inherits a pre-quantization.
\end{prop}

Two proofs of this proposition will be given. The first is carried out at the chain level, and hence has the advantage of explicitly showing how the fusion product inherits the pre-quantization from its factors. The main disadvantage, however, is that it may appear like a lucky guess. The second proof, carried out at the cohomology level, is a little more methodical despite being less transparent.

\begin{proof}[Proof $1$] First, recall that for a manifold $X$, the sub-complex of smooth singular cochains $S_{sm}^*(X;\Z)$ is chain homotopy equivalent to the singular cochain complex $S^*(X;\Z)$. Both $S^*_{sm}(X;\Z)$ and the de Rham complex of differential forms on $X$ may be viewed  as a sub-complexes of $S^*_{sm}(X;\R)$. \cite{Ma}

 By Proposition \ref{prop:eq-pqispq} it suffices to construct a cohomology class $[\alpha]\in S^3_{sm}(\phi;\Z)$ such that $\iota_\R([\alpha])=[(\omega,\eta)]$, where  $\omega=\pr_1^*\omega_1 + \pr_2^*\omega_2 + \half(\pr_1^*\phi_1^*\theta^L,\pr_2^*\phi_2^*\theta^R)$, and $\iota_\R$ is the coefficient homomorphism.

 To begin, choose a cochain representative $\bar{\eta} \in S^3_{sm}(G;\Z)$ (unique up to coboundary) that satisfies $\iota_\R([\bar{\eta}])=[\eta]$. For concreteness, let $\varepsilon \in S^2_{sm}(G;\R)$ be a smooth cochain that satisfies $\bar{\eta} -\eta = d\varepsilon$.
Since the cohomology class $[\bar{\eta}]\in H^3(G;\Z)$ is primitive (as $G$ is simply connected), there exists a smooth cochain $\tau \in S^2_{sm}(G\times G;\Z)$ (unique up to coboundary, given the choice of $\bar\eta$) that satisfies $d\tau = \mu^*\bar{\eta} - \pr_1^*\bar{\eta} - \pr_2^*\bar{\eta}$.

Since (for $i=1$ and $2$) $(M_i,\omega_i,\phi_i)$  is pre-quantized, there exists a
cochain representative $\bar\omega_i \in S^2_{sm}(M_i;\Z)$ (unique up to coboundary, given the choice of $\bar\eta$) such that $\bar\omega_i -\omega_i + \phi_i^*\varepsilon$ is exact in $S^2_{sm}(M_i;\R)$, so that $[(\bar\omega_i, \bar\eta)]$ is the given pre-quantization of $(M_i,\omega_i, \phi_i)$. 

Define the smooth relative cochain
 $\alpha=(\pr_1^*\bar\omega_1 + \pr_2^*\bar\omega_2 - (\phi_1\times \phi_2)^*\tau,\bar{\eta})$ in $S^3_{sm}(\phi;\Z)$.   It must be verified that (a) $\alpha$ is  a relative cocycle that is cohomologous to $(\omega,\eta)$ with real coefficients, and (b) $[\alpha]$ is independent of the choice of $\bar\eta$, and the subsequent choices of $\tau$ and $\bar\omega_i$. It will follow that $[\alpha]$  is the desired pre-quantization. 

It is clear that $d\alpha = 0$ since
\begin{align*}
d(\pr_1^*\bar\omega_1 + \pr_2\bar\omega_2 -(\phi_1\times\phi_2)^*\tau)
	& =\pr_1^*(-\phi_1^*\bar{\eta}) + \pr_2^*(-\phi_2^*\bar{\eta}) 
	  - (\phi_1\times\phi_2)^*(\mu^*\bar{\eta} - \pr_1^*\bar{\eta} - \pr_2^*\bar{\eta}) \\
	&=\pr_1^*(-\phi_1^*\bar{\eta}) + \pr_2^*(-\phi_2^*\bar{\eta}) 
	 -\phi^*\bar{\eta} + \pr_1^*\phi_1^*\bar{\eta} + \pr_2^*\phi_2^*\bar{\eta} \\
	& = -\phi^*\bar{\eta}.
\end{align*}
\noindent And $\alpha$ is cohomologous to $(\omega,\eta)$ with real coefficients, since 
\begin{align*}
\alpha - (\omega,\eta) & = (\pr_1^*\bar\omega_1 + \pr_2^*\bar\omega_2 - (\phi_1\times \phi_2)^*\tau 
  - (\pr_1^*\omega_1 + \pr_2^*\omega_2 + \half(\pr_1^*\phi_1^*\theta^L,\pr_2^*\phi_2^*\theta^R)) ,\bar{\eta} -\eta) \\
& = (  \pr_1^*(\bar\omega_1-\omega_1) + \pr_2^*(\bar\omega_2 -\omega_2) - (\phi_1\times \phi_2)^*\tau 
  -\half (\pr_1^* \phi_1^* \theta^L, \pr_2^*\phi_2^*\theta^R)    , d\varepsilon) \\
& = ( \pr_1^*(du_1 -\phi_1^*\varepsilon) + \pr_2^*(du_2 - \phi_2^*\varepsilon) 
 - (\phi_1\times \phi_2)^*(\tau + \half(\pr_1^*\theta^L, \pr_2^*\theta^R) ) , d\varepsilon) 
\end{align*}

\noindent for some primitives $u_i$ of $\bar\omega_i -\omega_i +\phi_i^*\varepsilon$. And since $\mu^*\eta=\pr_1^*\eta + \pr_2^* \eta -\half d(\pr_1^*\theta^L, \pr_2^*\theta^R)$, it follows that $d(\tau + \half (\pr_1^*\theta^L, \pr_2^*\theta^R)) = d(\mu^*\varepsilon - \pr_1^*\varepsilon - \pr_2^* \varepsilon)$. Therefore $\tau + \half  (\pr_1^*\theta^L, \pr_2^*\theta^R) -\mu^*\varepsilon + \pr_1^*\varepsilon + \pr_2^* \varepsilon$ is exact since $H^2(G\times G;\R)=0$. Continuing, 
\begin{align*}
\alpha - (\omega,\eta) & = ( \pr_1^*(du_1 -\phi_1^*\varepsilon) + \pr_2^*(du_2 - \phi_2^*\varepsilon) 
 - (\phi_1\times \phi_2)^*(\mu^*\varepsilon - \pr_1^*\varepsilon - \pr_2^* \varepsilon + dv) , d\varepsilon) \\
& = d(\pr_1^*u_1 + \pr_2^*u_2 - (\phi_1\times \phi_2)^*v, -\varepsilon)
\end{align*} 
\noindent where $v$ is some primitive of  $\tau + \half (\pr_1^*\theta^L,\pr_2^*\theta^R)- \mu^*\varepsilon + \pr_1^*\varepsilon + \pr_2^*\varepsilon$.

Lastly, it is straightforward to check that the cohomology class $[\alpha]$  does not depend on the choices made. Given a choice of $\bar\eta$, changing  $\bar\omega_i$ to $\bar\omega_i + d\beta_i$ changes $\alpha$ by $d(\beta_i,0)$. Changing $\tau$ to $\tau +d\gamma$ changes $\alpha$ by $d(-(\phi_1\times \phi_2)^*\gamma,0)$. Finally, if $\bar\eta$ is changed to $\bar\eta'=\bar\eta + d\rho$, then one may choose cochains $\bar\omega_i'=\bar\omega_i-\phi_i^*\rho$ and $\tau'=\tau +\mu^*\rho - \pr_1^*\rho - \pr_2^*\rho$ that satisfy the appropriate properties, producing the relative cocycle $\alpha'$. It is easy to see  that $\alpha'-\alpha=d(0,-\rho)$.  \end{proof}

\begin{proof}[Proof $2$] Recall that $\phi=\mu \circ (\phi_1\times\phi_2)$ where $\mu:G\times G\to G$ is group multiplication. Therefore there is a diagram  with exact rows:
\begin{diagram}
 H^3(\mu;\Z)&\rTo & H^3(\phi;\Z)&\rTo^\rho & H^3(\phi_1\times\phi_2;\Z)& \rTo^\delta& H^4(\mu;\Z)\\
 \dTo & & \dTo & & \dTo & & \dTo  \\
H^3(\mu;\R)&\rTo & H^3(\phi;\R)&\rTo^\rho & H^3(\phi_1\times\phi_2;\R) &\rTo& H^4(\mu;\R)
\end{diagram}

Since $\mu^*:H^3(G)\to H^3(G\times G)$ is injective, $H^3(\mu)=0$ and the maps labelled $\rho$ are injective. Therefore to show that the cohomology class $k[(\omega,\eta)]\in H^3(\phi;\R)$ is integral, it suffices to show that
\begin{enumerate}
\item[(a)]  $\rho(k[(\omega,\eta)])$ is integral, and 
\item[(b)] any integral pre-image of $\rho(k[(\omega,\eta)])$ maps to zero under $\delta$.
\end{enumerate}
Computing at the chain level, 
\begin{align*}
\rho(\omega,\eta)  & =(\pr_1^*\omega_1 + \pr_2^*\omega_2 +
 \half(\pr_1^*\phi_1^*\theta^L,\pr_2^*\phi_2^*\theta^R), 
  \pr_1^*\eta + \pr_2^*\eta -\half
d(\pr_1^*\theta^L,\pr_2^*\theta^R))\\
 & =(\pr_1,\pr_1)^*(\omega_1,\eta) +
 (\pr_2,\pr_2)^*(\omega_2,\eta) + \half d(0,(\pr_1^*\theta^L,\pr_2^*\theta^R))
\end{align*}
That is, $\rho(k[(\omega,\eta)])=k\alpha_1 + k\alpha_2$ for cohomology classes $\alpha_i=(\pr_i,\pr_i)^*[(\omega_i,\eta)]$ in $H^3(\phi_1\times\phi_2;\R)$. Since $k[(\omega_i,\eta)]$ are integral, then so are the classes $k\alpha_i$ and hence so is  $\rho(k[(\omega,\eta)])$, which verifies (a).

To verify (b), it suffices to check that $H^4(\mu;\Z)\to H^4(\mu;\R)$ is injective, or equivalently that $H^4(\mu;\Z)$ is torsion free, which is clear from the exact sequence
$$ H^3(G;\Z) \to H^3(G\times G;\Z) \to H^4(\mu;\Z) \to H^4(G;\Z),
$$ 
where the last group is trivial.
\end{proof}  

\section{Pre-quantization of the moduli space $M_G(\Sigma_1^g,b)$} \label{subsection:restate}

In Section \ref{eg:modulispace}, the incarnation of the moduli space of based flat $G$-bundles over $\Sigma_1^g$ as a quasi-Hamiltonian $\tilde{G}$-space was reviewed.  In particular, the proof of Proposition \ref{prop:bundletype} showed that under the identification $M_G(\Sigma_1^g,b)\cong G^{2g}$, the moment map $M_G(\Sigma^1_g, b) \to \tilde{G}$,
$$[(P,\theta, p)] \mapsto \widetilde\Hol(\partial\Sigma)$$
is simply $\tilde\phi:G^{2g}\to \tilde{G}$, the canonical lift of $\phi:G^{2g} \to G$,
$$ \phi(a_1, b_1, \cdots, a_g, b_g) = \prod [a_i,b_i].$$

This section formulates the obstruction to the existence of a pre-quantization of $M_G(\Sigma^1_g,b)$ as a quasi-Hamiltonian $\tilde{G}$-space.  In particular, the obstruction is a certain cohomology class in $H^3(G^{2g};\Z)$, and Theorem \ref{thm:coho-calc-equiv} shows that this obstruction vanishes at certain levels (recall Definition \ref{defn:qHlevels}).

\begin{lemma}Let $(M,\omega,\phi)$ be a level $l$ quasi-Hamiltonian
$G$-space. If $G$ is simply connected, and $(M,\omega,\phi)$ admits a pre-quantization, then $l \in \N$.
\end{lemma}
\begin{proof} Indeed, in the long exact sequence:
\[\cdots \to H^2(M;\R)\to H^3(\phi\, ;\R) \to H^3(G;\R)\to\cdots \]
the class $[(\omega,\eta)]$ maps to $[\eta]$. Therefore an
integral pre-image of $[(\omega,\eta)]$ gives an integral
pre-image of $[\eta]$.
It is well known that a generator of $H^3(G;\Z)\cong \Z$ maps to
$[\eta_1]$ under the coefficient homomorphism $\iota_\R$. (See
\cite{PS}, for example.) Since $\eta=l \cdot\eta_1$, $l \in \N$.
\end{proof}

\begin{prop} \label{prop:restateprob} Let $(M,\omega,\phi)$ be a
level $l$ quasi-Hamiltonian $G$-space. Assume $H^2(M;\R)=0$, and $G$
is simply connected.  Then $(M,\omega,\phi)$ admits a pre-quantization if and only if $l\in \mathrm{ker}\,\phi^*\subset
\Z=H^3(G;\Z)$.
\end{prop}
\begin{proof} The previous lemma shows that when $(M,\omega,\phi)$ admits a pre-quantization, $l$ is in the image of
$H^3(\phi\,;\Z)\to H^3(G;\Z)$ and therefore in the kernel of
$\phi^*$.

Conversely, suppose $l$ is in the kernel of $\phi^*$, or equivalently
in the image of $H^3(\phi\,;\Z)\toby{q}H^3(G;\Z)$. As in the
previous lemma, $l$ yields an integral lift of $[\eta]$. Choose
$\alpha\in H^3(\phi\,;\Z)$ with $q(\alpha)=l$, or with real coefficients,
$q(\iota_\R\alpha)=[\eta]$. Since $H^2(M;\R)=0$, the map $q$ is an injection
with real coefficients. Therefore $\iota_\R\alpha = [(\omega,\eta)]$.
 \end{proof}  

 Therefore, to decide which levels $l$ admit a pre-quantization of $M_G(\Sigma_1^g,b)\cong G^{2g}$, it suffices to determine which $l\in \N$ are in the kernel of $\tilde\phi^*$.  The kernel is a subgroup of $\Z$, and is therefore generated by some least positive integer $l_0(G)$, which will be computed in Section \ref{section:calculation}.
 
 \begin{remark} \label{remark:rationalcoeffs}
 Note that when $G$ is simple, a quasi-Hamiltonian $G$-space necessarily satisfies $\phi^*=0$ in degree 3 cohomology with real coefficients (by condition (1) of Definition \ref{defn:qhspace}).  Therefore with integer coefficients, the image of $\phi^*$ is necessarily torsion.
 \end{remark}

\begin{cor}Let $(M,\omega,\phi)$ be a level $l$ quasi-Hamiltonian $G$-space, where $G$ is simply connected. If $H_2(M;\Z)$ is $r$-torsion, and $r$ divides $l$, then $(M,\omega,\phi)$ admits a pre-quantization.
\end{cor}
\begin{proof} By the Universal Coefficient Theorem, $H^2(M;\R)=0$. Therefore it suffices to show $l\in \ker \phi^*$ by the previous proposition. But this is clear, since the torsion subgroup of $H^3(M;\Z)$ is isomorphic to (the torsion subgroup of) $H_2(M;\Z)$, which is $r$-torsion. 
\end{proof}  

The problem of determining $l_0(G)$ is independent of $g$. To see this,
let $\phi: G^2 \to G$ denote the commutator map, and
$\tilde\phi:G^2\to \tilde{G}$ its natural lift. Then  the moment
map $\phi_g:G^{2g}\to G$ lifts naturally as

\[ \tilde\phi_g=\mu_g\circ(
\underbrace{ \tilde\phi\times\cdots \times \tilde\phi}_g) \] where
$\mu_g:\tilde{G}^g\to\tilde{G}$ denotes successive multiplication
$\mu_g(x_1,\ldots,x_g)=x_1\cdots x_g$.

This gives $\mathrm{ker}\,\tilde\phi\subset
\mathrm{ker}\,\tilde\phi_g$. Indeed, the generator $z_3\in
H^3(\tilde{G};\Z)$ satisfies $\mu^*(z_3)=z_3\otimes 1 + 1 \otimes
z_3$. (Note that by the K\"unneth formula,
$H^3(\tilde{G}^2;\Z)=H^3(\tilde{G};\Z)\otimes H^0(\tilde{G}\;\Z)
\oplus H^0(\tilde{G};\Z)\otimes H^3(\tilde{G}\;\Z)$.) And
successive multiplication $\mu_g$ will therefore send $z_3$ to a
sum of tensors, where each tensor contains exactly one $z_3$ and
$(g-1)$ $1$'s. Therefore applying  $(\tilde\phi\times\cdots \times
\tilde\phi)^*$ to the resulting sum of tensors gives zero if
$\tilde\phi^*(z_3)=0.$

Furthermore, if $i:G^2\to G^{2g}$ denotes the inclusion of $G^2$ as the first pair
of factors, then$\tilde\phi=\tilde\phi_g\circ i$. Therefore $\mathrm{ker}\,\tilde\phi_g \subset \mathrm{ker}\,\tilde\phi$.  By Proposition
\ref{prop:restateprob} and the calculations in the following chapter, this yields the main result of the thesis:

\begin{thm}\label{thm:coho-calc-equiv} Let $G$ be a non-simply connected compact simple
Lie group with universal covering group $\tilde{G}$.  The quasi-Hamiltonian $\tilde{G}$-space $M_G(\Sigma_1^g,b)$ admits a pre-quantization if and only if the underlying level $l=ml_0(G)$ for some $m\in \N$, where $l_0(G)$ is given in Table \ref{table:main}.
\end{thm}

\begin{table}[h]
\centering

\begin{tabular}{|c||c|c|c|c|c|c|c|c|}
\hline
\multirow{2}{*}{$G$} & $PU(n)$ & $SU(n)/\mathbb{Z}_k$ & $PSp(n)$ & $SO(n)$ & $PO(2n)$ & $Ss(4n)$ & $PE_6$ & $PE_7$ \\
  & $n\geq 2$  &$n\geq 2$ &$n\geq 1$ &$n\geq 7$ &$n\geq 4$ &$n\geq 2$ & &\\
\hline
\multirow{2}{*}{$l_0(G)$} & \multirow{2}{*}{$n$} & \multirow{2}{*}{$\mathrm{ord}_k(\frac{n}{k})$} & 1, $n$ even & \multirow{2}{*}{1} & 2, $n$ even & 1, $n$ even & \multirow{2}{*}{3} & \multirow{2}{*}{2} \\
 & & & 2, $n$ odd & & 4, $n$ odd & 2, $n$ odd &  &    \\
\hline
\end{tabular}

\caption{\small{The integer $l_0(G)$.  \underbar{Notation}: $\mathrm{ord}_k(x)$ denotes the order of $x$ mod $k$ in $\Z_k$.}}
\label{table:main}
\end{table}

\begin{remark} \label{remark:simplyconnectedG} If $G$ is simply connected, $l_0(G)=1$ since $H^3(G\times G;\Z)$ is torsion free.
\end{remark} 
 
\section{Pre-quantization of $M_{PU(n)}(\Sigma^g_r;\mathcal{C}_1, \ldots, \mathcal{C}_r)$} \label{sec:punconj}

Let $G=PU(n)$.
This section addresses the pre-quantization of the moduli space $M_G(\Sigma^g_r,b;\mathcal{C}_2, \ldots, \mathcal{C}_r)$ of based flat $G$-bundles over $\Sigma$ with holonomies around the boundary components $v_2, \ldots, v_r $ in $ \Sigma$ prescribed in the conjugacy classes $\mathcal{C}_2,\ldots,\mathcal{C}_r\subset G$, respectively.  Recall that Section \ref{eg:modulispace} recognizes the moduli space as the fusion product 
$$
M_G(\Sigma_r^g,b;\mathcal{C}_2,\ldots,\mathcal{C}_r) \cong G^{2g} \times \mathcal{C}_2 \times \cdots \times \mathcal{C}_r,
$$ 
with group-valued moment map $\phi:G^{2g}\times \mathcal{C}_2, \times \cdots \times \mathcal{C}_r \to G$ given by
$$
(a_1, b_1, \ldots, a_g, b_g, c_2, \ldots, c_r) \mapsto \prod [a_i,b_i] c_2 \cdots c_r.
$$

To be more precise, this section considers the pre-quantization of the pullback quasi-Hamiltonian $\tilde{G}$-space $P(X)$ (see Section \ref{sec:nonsimplyconnectedG}) over $X=G^{2g}\times \mathcal{C}_1\times \cdots \times \mathcal{C}_r$, where $\tilde{G}=SU(n)$ is the universal covering group of $G$. Since the moment map $\phi:G^{2g} \to G$ lifts to $\tilde{G}$, by Remark \ref{remark:whenpullbacksplits}, $P(X)=G^{2g} \times P(M)$, where $P(M)$ is the pullback quasi-Hamiltonian $\tilde{G}$-space over $M=\mathcal{C}_1\times \cdots \times \mathcal{C}_r$. 
Appealing to Proposition \ref{prop:fusion}, and Theorem \ref{thm:coho-calc-equiv}, it therefore suffices to consider the pre-quantization of $P(M)$.

Note that the  pre-quantization of conjugacy classes $\C$ in simply connected Lie groups $H$ is already well understood  by the following theorem from \cite{M}.  (The theorem makes use of the correspondence mentioned in Remark \ref{remark:alcovecorrespondence} between a chosen alcove $\Delta\subset \mathfrak{t}$ and the set of conjugacy classes in $H$.  Also, $I$ denotes the integral lattice, and $B$ denotes the basic inner product on $\mathfrak{t}$---see Remark \ref{remark:innerproducts}---with induced isomorphism  $B^\flat:\mathfrak{t} \to \mathfrak{t}^*$.)   The pre-quantization of $P(M)$ will make use of the theorem for $H=\tilde{G}=SU(n)$.

\begin{thm} \label{thm:conj} \cite{M} Let $\mathcal{C}$ be the conjgacy class of $\exp \zeta$ in the simply connected compact simple Lie group $H$, where $\zeta \in \Delta$.  Then $\mathcal{C}$ admits a pre-quantization at level $k$ if and only if $B^\flat(k\zeta) \in I^*=\mathrm{Hom}(I,\Z)$, the lattice of integral forms.
\end{thm}

\subsection*{Pre-quantization of $P(M)$}

To begin, recall from Section \ref{sec:liethy} that the set of conjugacy classes $Con(SU(n))$ in $SU(n)$ is in bijective correspondence with $\Delta$, a chosen alcove in $\mathfrak{t}$, the Lie algebra of a chosen maximal torus $T$. The center $Z(SU(n))\cong \Z_n$ acts on the set of conjugacy classes by translation, and it was found in Section \ref{sec:liethy} that, under the correspondence $\Delta \longleftrightarrow Con(SU(n))$, the center acts on the $(n-1)$-simplex $\Delta$ by shifting the vertices
$$
v_0\mapsto v_1, \quad v_1\mapsto v_2, \quad \cdots, \quad v_n \mapsto v_0.
$$
Consequently, as observed in Remark \ref{remark:specialconjclass}, the only fixed point of the action is the barycenter $\zeta_0=\tfrac{1}{n} \Sigma v_i$.  Equivalently, the center $Z(SU(n))$ acts on the conjugacy class of $\exp \zeta_0$, and the $Z(SU(n))$-orbit of any other conjugacy class is a disjoint union of $n$ conjugacy classes in $SU(n)$. This proves the followng Lemma.

\begin{lemma} \label{lemma:specialconj}
Let  $\C$ be a conjugacy class in $PU(n)$, and consider $\tilde\C=\pi^{-1}(\C)$.  If $\C$ is the conjugacy class of $c_0=\exp \zeta_0$, where $\zeta_0$ is the barycenter of the alcove $\Delta$, then $\tilde\C$ is a single conjugacy class in $SU(n)$.  Otherwise, $\tilde\C$  is  a disjoint union of $n$ conjugacy classes in $SU(n)$.
\end{lemma}

 If $\C$ is the conjugacy class of $c_0$, then $\tilde\C$ is the conjugacy class of $\exp\zeta_0$, and $\C=\tilde\C/\Z_n$.  Moreover, the covering projection $\pi:SU(n) \to PU(n)$ restricts to a covering projection $\pi:\tilde\C \to \C$, which is the universal covering projection.  To see this, recall that since $SU(n)$ is simply connected, the centralizer $Z(g)$ of any element $g$ is connected (by Corollaire 5.3.1 in \cite{Bourbaki}).  From the identification $\tilde\C=SU(n)/Z(\exp \zeta_0)$, and the exact sequence of homotopy groups associated to the fibration sequence $SU(n) \to SU(n)/Z(\exp\zeta_0) \to BZ(\exp\zeta_0)$, it follows that $\tilde\C$ is simply connected. 

Let $M=\mathcal{C}_1\times \cdots\times  \mathcal{C}_m$, where each $\C_j$ is a conjugacy class in $PU(n)$.  The next Lemma shows that it will suffice to consider the case where each $\C_j=\C$, the conjugacy class of $c_0$.  

\begin{lemma} Let $M=\C_1 \times \cdots \times \C_m$ be a product of conjugacy classes in $PU(n)$, and let $\C$ be the conjugacy class of $c_0=\exp \zeta_0$, where $\zeta_0$ is the barycenter of the alcove $\Delta$.  If $\C_1\neq \C$, then the pullback quasi-Hamiltonian $SU(n)$-space $P(M)$ is the fusion product $P(\C_1) \times P(M')$, where $M'=\C_2 \times \cdots \times \C_m$.
\end{lemma}
\begin{proof}
Suppose that $\C_1\neq \C$, the conjugacy class of $c_0$.  Then, by Lemma \ref{lemma:specialconj}, $\tilde\C_1=\pi^{-1}(\C_1)$ is a disjoint union of conjugacy classes $\tilde\C_1^1 \cup \cdots \cup \tilde\C_1^n$.  Equivalently, $\tilde\C_1\cong \tilde\C_1^1 \times \Z_n$.  Write $M'=\C_2 \times \cdots \times \C_m$, so that $M=\C_1\times M'$.  Then, by Remark \ref{remark:whenpullbacksplits}, $P(M)=\tilde\C_1^1 \times P(M')$.  
\end{proof}

Therefore, the pre-quantization of $P(M)$ may be studied using Proposition \ref{prop:fusion} and Theorem \ref{thm:conj}, and it will suffice to consider the case where each $\C_j=\C$, the conjugacy class of $c_0=\exp \zeta_0$, where $\zeta_0$ is the barycenter of the alcove $\Delta$. From now on, let $M=\C \times \cdots \times \C$ (with $m$ factors), where $\C$ is the conjugacy class of $c_0$.  

By Lemma \ref{lemma:htwoistorsion} below, $H^2(P(M);\R)=0$.  Therefore, by Proposition \ref{prop:restateprob}, $P(M)$ admits a pre-quantization at level $l$ if and only if $\varphi^*(l)=0$ in $H^3(P(M);\Z)$, where $\varphi:P(M)\to SU(n)$ is the group-valued moment map of the pullback quasi-Hamiltonian $SU(n)$-space over $M$.  This is the statement of Proposition \ref{prop:preqofPM} below.

In order to arrive at Lemma \ref{lemma:htwoistorsion}, it will be necessary  to understand the induced action of $\Z_n$ on $H^*(\tilde\C;\Z)$ (\emph{cf}. Proposition \ref{prop:cohomconj}).  A nice description of $H^*(\tilde\C;\Z)$ arises from the identification of $\tilde\C$ as a homogeneous space.  Recall that the centralizer of $\exp \zeta_0$ is $T$, the maximal torus; therefore, that the map $SU(n)/T\to  \tilde\C$ given by $gT\mapsto g\exp \zeta_0 g^{-1}$ is a homeomorphism.    The $\Z_n$-action on $\tilde\C$ corresponds to a $\Z_n$-action on $SU(n)/T$, which will be described in the next paragraph,  yielding the  $\Z_n$-action on $H^*(SU(n)/T;\Z)\cong H^*(\tilde\C;\Z)$ (see Corollary \ref{cor:Wcohoaction} below) that is used to prove Proposition \ref{prop:cohomconj}. 

Suppose $z\in \Z_n$.  Since $z\exp \zeta_0 \in T$ is conjugate to $\exp \zeta_0$, Theorem 3.18 in Chapter V of \cite{MT} guarantees that there exists an element $h\in N(T)$, the normalizer of $T$, such that $z\exp \zeta_0=h^{-1}\exp\zeta_0 h$.  Such an $h$ is only well defined up to translation by an element in $T$, and hence gives a well defined element $hT$ in $N(T)/T\cong W$, the Weyl group.  Recall that $N(T)/T$ acts on $SU(n)/T$ by $hT\cdot gT=gh^{-1}T$.  
Therefore, for $z$ and the corresponding $hT$ as above,  the equality
$$
zg\exp\zeta_0g^{-1}=gz\exp\zeta_0g^{-1}=gh^{-1}\exp\zeta_0 h g^{-1}
$$
shows that acting by $z$ in $\tilde\C$ corresponds to acting by $hT$ in $SU(n)/T$.  It is an easy exercise to check that the  map $\Z_n \to N(T)/T$ defined above is an injective homomorphism, and that the $\Z_n$-action on $\tilde\C$ corresponds to the restriction of the $N(T)/T$-action on $SU(n)/T$.

\begin{prop}
Define $\lambda:\Z_n \to N(T)/T$ as follows. For $z\in \Z_n$, write $z\exp \zeta_0=h^{-1} \exp \zeta_0 h$ for some $h\in N(T)$, and let $\lambda(z)=hT$.  Then $\lambda$ is an injective homomorphism.
\end{prop}

En route to describing the $\Z_n$-action on $H^*(SU(n)/T;\Z)$, recall that $BT$ carries an action of the Weyl group $W\cong N(T)/T$, which may be described using the universal bundle $ESU(n)\to ESU(n)/T=BT$. Indeed,
$$
hT\cdot [e] =[e\cdot h^{-1}],
$$
where $e\cdot h^{-1}$ denotes the action of $h^{-1}\in N(T)\subset SU(n)$ on $e\in ESU(n)$, and $[x]$ denotes the $T$-orbit of $x\in ESU(n)$.  The action is well defined, since if $hT=h'T$, then $h'=ht$ for some $t\in T$, and 
$$
h'T\cdot [e]=[e\cdot (t^{-1}h^{-1})] =[(e\cdot (h^{-1}ht^{-1}h^{-1})) 
]=[(e\cdot h^{-1})\cdot (ht^{-1}h^{-1})]=[e\cdot h^{-1}].
$$

It is easy to see that the classifying map $\rho:SU(n)/T \to BT$ is $W$-equivariant. Indeed, consider the $T$-equivariant map $SU(n) \to ESU(n)$ that includes $SU(n)$ as the $SU(n)$-orbit of the base point, $g\mapsto e_0\cdot g$.  The classifying map $\rho:SU(n)/T \to BT$ is therefore the induced map of $T$-orbits. 
Hence, $\rho(hT\cdot gT)=[e_0\cdot (gh^{-1})] = [(e_0\cdot g)\cdot h^{-1}]=hT \cdot \rho(gT)$.

\begin{thm} \cite{MT} \label{thm:cohomGT} The induced map $\rho^*:H^*(BT;\Z) \to H^*(SU(n)/T;\Z)$ is onto, inducing the isomorphism 
$$
H^*(SU(n)/T;\Z)\cong \Z[t_1,\ldots, t_n]/(\sigma_1, \ldots \sigma_n),
$$
where $\sigma_i=\sigma_i(t_1,\ldots t_n)$ is the $i^{th}$ elementary symmetric polynomial, and by abuse of notation, $\rho^*(t_i)=t_i=t_i\, \mathrm{mod} \sigma_1$, where $t_i\in H^2(BT;\Z)$ are generators. Furthermore, $W\cong \Sigma_n$ acts on $H^*(BT;\Z)$ by permuting the $t_i$'s, and hence the same is true for $H^*(SU(n)/T;\Z)$.
\end{thm}

\begin{cor}  \label{cor:Wcohoaction} There is a generator $z\in \Z_n$ such that the $\Z_n$-action on $H^*(SU(n)/T;\Z)$ is determined by $z\cdot t_1=t_2$, $z\cdot t_2=t_3$, \ldots, $z\cdot t_{n-1}=t_n$, $z\cdot t_n=t_1$.
\end{cor}
\begin{proof} 
The maximal torus $T=(S^1)^{n-1}$ in $SU(n)$ may be viewed as lying inside the maximal torus $T'=(S^1)^n$ of $U(n)$ as the kernel of the multiplication map $T' \to S^1$.  By Theorem I-3.13 in \cite{MT}, the Weyl group $W=N(T)/T$ acts on $T$ by switching factors in $T'$.  Using the obvious identification of $H^1(T;\Z)\cong H^2(BT;\Z)$, it is clear that $W$ acts by permuting the $t_i$'s, as stated in the previous Theorem.  It is therefore necessary to check that $\lambda(z)$ is the permutation in $W$ that sends $t_i\mapsto t_{i+1}$ for $i=1$ to $n-1$, and $t_n\mapsto t_1$, where $z\in \Z_n$ is a generator.

Recall from Example \ref{eg:conj} that $z=\exp v_1$ generates $\Z_n$, where 
$$
v_1=\frac{1}{n}((n-1)e_1 - e_2 - \cdots -e_n)
$$
is a vertex of the alcove $\Delta\subset \mathfrak{t}\cong \{x=\sum x_ie_i \in \R^n \, | \, \sum x_i=0\}$. An easy calculation shows that the barycenter $\zeta_0=\tfrac{1}{n}\sum v_i$ of the alcove is 
$$
\zeta_0 = \frac{1}{2n}((n-1)e_1 + (n-3)e_2 + (n-5)e_3 + \cdots + (-n+1) e_n),
$$
and hence,
$$
\zeta_0+v_1 = \frac{1}{2n}(3(n-1) e_1 + (n-5)e_2  + (n-7)e_3 + \cdots + (-n-1)e_n).
$$
Observe that the effect on $\zeta_0$ of  translation by $v_1$ moves the $j^{th}$ coordinate  to the $(j-1)^{st}$ coordinate for $j=3$ to $n$.  Furthermore, since $(n-3)-3(n-1)=-n$, and $(n-1)-(-n-1)=n$, it follows that translation by $\exp v_1$ sends the $j^{th}$ factor of $T\subset T'$ to the $(j-1)^{st}$, for $j=2$ to $n$, and that it sends the first factor to the $n^{th}$. In other words, $\lambda(\exp v_1)$ sends $t_i\mapsto t_{i-1}$. Therefore $z=\exp (-v_1)$ satisfies the requirements of the Corollary.
\end{proof}

\begin{prop} \label{prop:cohomconj} The subgroup of invariant elements $H^2(\tilde\C)^{\Z_n}$ with respect to the $\Z_n$-action on $H^2(\tilde{\C};\Z)$ is trivial.
\end{prop}
\begin{proof}
If $x=\sum a_i t_i$ in $H^2(\tilde\C;\Z)$ is invariant under the action of the generator in  Corollary \ref{cor:Wcohoaction}, then
$$
a_1t_1 + a_2t_2 + \cdots a_nt_n = a_nt_1 + a_1t_2 + \cdots + a_{n-1}t_n \quad\mathrm{mod}\, \sigma_1.
$$
Therefore, $a_1-a_n=a_2-a_1=\cdots a_n-a_{n-1}=c$ for some $c$. Adding these equations together gives $0=nc$, and hence $a_1=a_2=\cdots =a_n$.  That is, $\sigma_1=t_1+t_2+\cdots t_n$ divides $\sum a_it_i$, and $x=0$.
\end{proof}

\begin{lemma}  \label{lemma:htwoistorsion}
Let $M=\C^m$, where $\C$ is the conjugacy class of $c_0$ in $PU(n)$.  Then $H^2(P(M);\Z)\cong (\Z_n)^{\oplus(m-1)}$.
\end{lemma}
\begin{proof}
Let  $\tilde{M}=\tilde{\C} \times \cdots \times \tilde{\C}$ ($m$ factors), where $\tilde\C$ is the conjugacy class of $\exp \zeta_0$, and $\zeta_0$ is the barycenter of the alcove $\Delta$. Then the map $\tilde{M}\to M$ is the universal covering projection, with fiber $\Z_n \times \cdots \times \Z_n$ ($m$ factors).  Therefore, the induced covering $\tilde{M} \to P(M)$ is also universal, and has fiber $K\cong (\Z_n)^{m-1}$, the kernel of the multiplication map $\Z_n \times \cdots \times \Z_n \toby{\mathrm{mult}} \Z_n$.

By the K\"unneth Theorem, $H^2(\tilde{M};\Z)\cong\bigoplus H^2(\tilde\C;\Z)$. 
Since the $K$-action on $H^2(\tilde{M};\Z)\cong\bigoplus H^2(\tilde\C;\Z)$ factors through $(\Z_n)^m$, 
$$
 H^2(\tilde{M};\Z)^K=\bigoplus H^2(\tilde\C;\Z)^{\Z_n},
$$
which is zero by Proposition \ref{prop:cohomconj}. Therefore, $H^2(P;\Z)\cong H^2(BK;\Z)$, by Proposition \ref{prop:serreexactsequence}. 
\end{proof}

As mentioned earlier, this yields the following Proposition.
\begin{prop} \label{prop:preqofPM}
Let $M=\C^m$, where $\C$ is the conjugacy class of $c_0$ in $PU(n)$, and let $P(M)$ be the pullback quasi-Hamiltonian $SU(n)$-space with group-valued moment map $\varphi:P(M)\to SU(n)$.  Then $P(M)$ admits a pre-quantization at level $k$ if and only if $\varphi^*(k)=0$.
\end{prop}

Notice that if $P(M)$ admits a pre-quantization, then $\tilde{M}=\tilde\C^m$ does as well.  Indeed,  the commutative diagram
\begin{diagram}
\tilde{M} & \rTo^{\tilde\phi} & \tilde{G} \\
\dTo & & \dEquals \\
P & \rTo^\varphi & \tilde{G} \\
\end{diagram}
provides a natural map $H^3(\varphi;R)\to H^3(\tilde\phi;R)$ with any coefficient ring $R$.  Therefore, integral lifts in $H^3(\varphi;\Z)$ map to integral lifts in $H^3(\tilde\phi;\Z)$.  It follows that a necessary condition for the pre-quantization of $P(M)$ is that the conjugacy class $\tilde\C$ admit a pre-quantization.  

\begin{prop} Let $\tilde\C$ be the conjugacy class of $\exp \zeta_0\in SU(n)$, where $\zeta_0$ is the barycenter in the alcove $\Delta$.  If $\tilde\C$ admits a pre-quantization at level $k$, then $n$ divides $k$.
\end{prop}
\begin{proof}
By Theorem \ref{thm:conj}, $\tilde\C$ admits a level $k$ pre-quantization if and only if $B^\flat(k\zeta_0)$ is in  $I^*=\mathrm{Hom}(I,\Z)$, the lattice of integral forms. As in Example \ref{eg:conj} of Section \ref{sec:liethy},  the integral lattice $I=\ker \exp_T\subset \mathfrak{t}$ is the restriction of the standard integer lattice $\Z^n\subset \R^n$ to $\mathfrak{t}=\{x=\sum x_ie_i \in \R^n| \sum x_i=0\}$. It follows that for any $v\in \mathfrak{t}$, $B^\flat (v)$ belongs to $I^*$ if and only if $B(v,e_i-e_{i+1})$ is an integer.  Recall from the proof of Corollary \ref{cor:Wcohoaction} that 
$$
\zeta_0=\frac{1}{2n}((n-1)e_1 + (n-3)e_2 + \cdots + (-n +1)e_n).
$$
Therefore, 
$B(\zeta_0,e_i-e_{i+1})=\frac{1}{n}$, and $B^\flat(k\zeta_0) \in I^*$ if and only if $n$ divides $k$.
\end{proof}

\begin{cor} \label{cor:necessaryPM} If $P(M)$ admits a level $k$ pre-quantization, then $n$ divides $k$.
\end{cor}

Theorem \ref{thm:preqofPM}  is the main result on the pre-quantization of $P(M)$.

\begin{lemma} \label{lemma:atmost} $H^1(\C;\Z)=0$, $H^2(\C;\Z)\cong \Z_n$, and $H^3(\C;\Z)$ is a  cyclic group of order at most $n$.
\end{lemma}
\begin{proof} Consider the  Cartan-Leray spectral sequence of the covering $\tilde\C \to \C$, with $E_2^{p,q}=H^p(\Z_n; H^q(\tilde\C;\Z))$.  It will suffice to compute the first three rows of the $E_2$-term, $E_2^{p,q}$ for $0\leq q \leq 2$.

Since $E_2^{p,0}=H^p(\Z_n;\Z)$, $E_2^{p,0}=0$ for $p$ odd, and $\Z_n$ for $p>0$ even.  Also, since $H^q(\tilde\C;\Z)=0$ for $q$ odd (see Theorem \ref{thm:cohomGT}), then $E_2^{p,1}=0$ for all $p$.  To compute $E_2^{p,2}$, begin with the standard (periodic) projective resolution of the trivial $\Z\Z_n$ module $\Z$,
$$
\cdots \to \Z\Z_n \toby{N} \Z\Z_n \toby{\sigma-1} \Z\Z_n \toby{N} \Z\Z_n \toby{\sigma-1}  \Z \Z_n \toby{\varepsilon}\Z
$$
where $\Z_n$ is generated by $\sigma$, and $N=1+\sigma + \cdots + \sigma^{n-1}$.  Applying $\mathrm{Hom}(-,H^2(\tilde\C;\Z))$, and taking homology gives
$$
E_2^{\mathrm{odd},2} = \frac{\{x\in H^2(\tilde\C;\Z) \, | \, Nx=0\}}{(\sigma-1)H^2(\tilde\C;\Z)}, \quad \text{and} \quad E_2^{\mathrm{even},2}=\frac{H^2(\tilde\C;\Z)^{\Z_n}}{NH^2(\tilde\C;\Z)}.
$$
Recall from Proposition \ref{prop:cohomconj} that $H^2(\tilde\C;\Z)^{\Z_n}=0$,   and hence $E_2^{\mathrm{even},2}=0$.  In other words, the induced map $(\sigma-1):H^2(\tilde\C;\Z) \to H^2(\tilde\C;\Z)$ is injective; therefore, the induced map $N:H^2(\tilde\C;\Z) \to H^2(\tilde\C;\Z)$ must be trivial, and $E_2^{\mathrm{odd},2} = \mathrm{coker} \, (\sigma-1)$.  

To compute the cokernel of $(\sigma-1):H^2(\tilde\C;\Z) \to H^2(\tilde\C;\Z)$, choose the basis 
$$\{t_2-t_1, t_3-t_2, \ldots, t_{n-1}-t_{n-2}, t_{n-1} \},
$$ 
and observe that 
$$
(\sigma -1)(t_1)=t_2-t_1, \quad (\sigma -1) (t_2)=t_3-t_2, \quad \ldots \quad, (\sigma-1)(t_{n-1})=t_{n-2}-t_{n-1},
$$ 
and $(\sigma-1)(t_1+2t_2+ \cdots + (n-2)t_{n-2} - t_{n-1})=nt_{n-1}$.  Therefore, $\mathrm{coker} \, (\sigma-1) \cong \Z_n$.

Having computed $E_2^{p,q}$ for $p+q\leq 3$, it is easy to see that $H^1(\C;\Z)=0$, $H^2(\C;\Z)\cong \Z_n$, and that $H^3(\C;\Z)$ is the kernel of the differential $d:E_2^{1,2} \to E_2^{0,4}$.  That is, $H^3(\C;\Z)$ is a cyclic subgroup of order at most $n$.
\end{proof}

\begin{lemma} The covering projection $\pi:SU(n) \to PU(n)$ induces an isomorphism $\pi^*:H^3(PU(n);\Z) \to H^3(SU(n);\Z)$ when $n$ is odd, and is multiplication by $2$ when $n$ is even.
\end{lemma}
\begin{proof}
Recall that the inclusion $SU(n)\to U(n)$ induces an isomorphism on $H^3(-;\Z)$, and consider the Serre spectral sequence for the fibration $U(n) \toby{\pi} PU(n) \to \mathbb{C} P^\infty$, with $E_2$-term $\Lambda(x_1, x_3, \ldots, x_{2n-1}) \otimes \Z[\alpha]$.  Since $H^2(PU(n)) \cong \pi_1(PU(n)) = \Z_n$, the differential $d:E_2^{0,1} \to E_2^{2,0}$ must be multiplication by $n$. Therefore,
$$
E_3\cong \Lambda(x_3, \ldots, x_{2n-1}) \otimes \Z_n[\alpha].
$$
To determine the differential $d:E_3^{0,3} \to E_3^{4,0}$, note the following diagram of fibrations, where the map $f:\mathbb{C} P^\infty \to BU(n)$ classifies the $U(n)$-bundle $PU(n) \to \mathbb{C} P^\infty$.  
\begin{diagram}
U(n) & \rTo^\pi & PU(n) & \rTo  & \mathbb{C} P^\infty \\
\dEquals & & \dTo &  & \dTo>f \\
U(n) & \rTo & EU(n) & \rTo & BU(n)
\end{diagram}
Recall that $H^*(BU(n)) = \Z[c_1, \ldots, c_n]$, and that the Chern classes $c_i$ are transgressive.  Therefore, $d(x_3)=f^*(c_2)$. The classifying map $f$ factors as $f:\mathbb{C} P^\infty =BS^1 \toby{B\Delta} BT^n \toby{j} BU(n)$, where $\Delta:S^1 \to T^n$ is the diagonal inclusion of the center into the maximal torus of $U(n)$. 
Since  $H^*(BT^n)=\Z [t_1, \ldots, t_n]$, and  
$$j^*(c_2)= \sum_{r<s} t_r t_s$$
it follows that $f^*(c_2)=\frac{n(n-1)}{2} \alpha^2$.  Hence, $d(x_3)=0$ when $n$ is odd, and $d(x_3) =\frac{n}{2}\alpha^2$ when $n$ is even.  
Therefore, the induced map $\pi^*:H^3(PU(n);\Z) \to H^3(U(n);\Z)$ is an isomorphism when $n$ is odd, and is multiplication by $2$ when $n$ is even.
\end{proof}

\begin{thm} \label{thm:preqofPM}
Let $M=\C^m$, where $\C$ is the conjugacy class of $c_0=\pi(\exp \zeta_0)$, and $\zeta_0$ is the barycenter of the alcove $\Delta$.  If $n$ is odd, then the pullback quasi-Hamiltonian $SU(n)$-space $P(M)$ admits a level $k$ pre-quantization if and only if $n$ divides $k$.  If $n$ is even, then $P(M)$ admits a level $k$ pre-quantization if $2n$ divides $k$.
\end{thm}
\begin{proof} By Proposition \ref{prop:preqofPM}, it will suffice to show that $\varphi^*(n)=0$ when $n$ is odd, and $\varphi^*(2n)=0$ when $n$ is even.  
By the pullback diagram that defines $P(M)$,
\begin{diagram}
P(M) &\rTo^\varphi &    SU(n) \\
\dTo>p    &  & \dTo>\pi \\
M & \rTo^\phi & PU(n) \\
\end{diagram}
and the previous Lemma, $\varphi^*(n)=\varphi^*\pi^*(n)=p^*\phi^*(n)$ when $n$ is odd, and $\varphi^*(n)=\varphi^*\pi^*(\frac{n}{2})=p^*\phi^*(\frac{n}{2})$ when $n$ is even.  By Lemma \ref{lemma:atmost} and the K\"unneth Theorem, $H^3(M;\Z)$ is a torsion group with elements of order at most $n$.  Therefore, $\phi^*(n)=0$, which proves the Theorem.
\end{proof}

Observe that if $n$ is odd, Theorem \ref{thm:preqofPM} says that the necessary condition of Corollary \ref{cor:necessaryPM} (for a level $k$ pre-quantization) is actually sufficient.  If $n$ is even, however, Theorem \ref{thm:preqofPM} only states that a sufficient condition for a level $k$ pre-quantization is that $2n$ divides $k$.  In particular, Theorem \ref{thm:preqofPM} does not state whether or not the necessary condition of Corollary \ref{cor:necessaryPM} is sufficient. In fact, the following example shows that Theorem \ref{thm:preqofPM} is not best possible.

\begin{eg} \label{eg:Stwo} $\tilde\C \subset SU(2)$, $M=\C\times \C$
\end{eg}

Let $\tilde\C\subset SU(2)$ be the conjugacy class of 
$$
\exp \zeta_0=\exp (\tfrac{1}{4}e_1-\tfrac{1}{4}e_2) = 
\begin{pmatrix} 
i & 0 \\
0 & -i \\
\end{pmatrix}.
$$
It is easy to check that the elements in $\tilde\C$ that are conjugate to $\exp\zeta_0$ are matrices of the form 
$\left(\begin{smallmatrix}
yi & u+iv \\
-u+iv & -yi \\
\end{smallmatrix} \right)$
in $SU(2)$.  That is, $\tilde\C$ is a sphere $S^2$. Furthermore, the $\Z_2$-action of the center, generated by translation by 
$\bigl( \begin{smallmatrix} 
-1 & 0 \\
0 & -1 \\
\end{smallmatrix} \bigr)$,
is clearly the antipodal action. Hence $\C=\R P^2$.

In this case, the cohomology of $P(M)$, where $M=\R P^2 \times \R P^2$ is easy to calculate.  In particular, $H_1(P(M);\Z)\cong\Z_2$; therefore, by Poincar\'e duality,
 $H^3(P(M);\Z)\cong \Z_2$.  By Proposition \ref{prop:preqofPM}, it follows that $P(M)$ admits a level $k$ pre-quantization if and only if $k$ is even, as $\varphi^*(2)=0$. \hfill $\qed$

\chapter{Computations} \label{chapter:calc}

This Chapter contains the calculations that support the main result of this thesis.   The main result, which is the determination of $l_0(G)$, is summarized in Table \ref{table:main} of Theorem \ref{thm:coho-calc-equiv}, and is spread over various  Theorems  in this Chapter.

\section{General remarks about computing $l_0(G)$} \label{section:calculation}

The computation of $l_0(G)$ (see Theorem \ref{thm:coho-calc-equiv}) uses the classification of compact simple Lie groups, and known results about their cohomology (see Section \ref{sec:coholiegroup} and Appendix \ref{chapter:app}). The strategy for determining the image of $\tilde\phi^*$ with integer coefficients is to work with coefficients in
$\Zp$ for every prime $p$. This information is assembled using the Bockstein spectral sequence (see Section \ref{sec:spectralsequences}) to deduce the image of  $\tilde\phi^*$  with integer coefficients.  

Throughout this section, $\pi: \tilde{G}\to G$ denotes the universal covering homomorphism, and $\tilde\phi: G\times G \to \tilde{G}$ denote the canonical lift of the commutator map $\phi:G\times G \to G$ (as in Example \ref{eg:double}). To remove ambiguities that may arise, the commutator map $\phi$ and its lift $\tilde\phi$ may also be denoted $\phi_G$ and $\tilde\phi_G$, respectively.  Similarly, group multiplication will be denoted $\mu_G=\mu:G\times G\to G$.

Some of the calculations that follow make use of the fact that  the commutator map $\phi$  factors as the composition:
\begin{diagram}
\phi:G^2& \rTo^{\Delta\times \Delta}& G^4&\rTo^{\mathrm{id}\times T\times \mathrm{id}}&  G^4& \rTo^{\mathrm{id} \times \mathrm{id} \times c\times c}& G^4 & \rTo^{\mu\times \mu}& G^2 &\rTo^{\mu} &G
\end{diagram}
where $\Delta:G\to G\times G$ denotes the diagonal map, and $T:G\times G\to G\times G$ switches factors $T(g,h)=(h,g)$.  Therefore, in order to compute the induced map $\phi^*$ it is necessary to know the induced maps $\Delta^*$, $\mu^*$, and $c^*$, which were discussed in  Section \ref{sec:coholiegroup}, and that $T^*(u\otimes v)=(-1)^{|u||v|} v\otimes u$.

Finally, notice that the only relevant primes $p$ that require investigation are primes $p$ dividing the order of the central subgroup $Z=\ker \pi$.  Indeed, as discussed in Section \ref{sec:coholiegroup}, for primes $p$ not dividing the order of $Z$, the covering $\pi:\tilde{G}\to G$ induces an isomorphism on $H^*(-;\Z_p)$.  It is easy to see that $\tilde\phi^*$ is the zero map in this case.

Indeed, consider the following commutative diagram. 
\begin{diagram}
\tilde{G} \times \tilde{G} & \rTo^{\phi_{\tilde{G}}} & \tilde{G} \\
\dTo<{\pi \times \pi} & \ruTo^{\tilde\phi} & \\
G\times G & & \\
\end{diagram}
Let $z_3 \in H^3(\tilde{G};\Z_p)$ be the reduction mod $p$ of an integral generator of $H^3(\tilde{G};\Z)$, which is necessarily primitive as it is of minimal degree.  Then the calculation, 
\begin{align*}
\phi_{\tilde{G}}^*(z_3) & = (\Delta\times \Delta)^*(\id \times T \times \id)^*(\id \times \id \times c\times c)^*(\mu\times \mu)^* \mu^*(z_3) \\
&=   (\Delta\times \Delta)^*(\id \times T \times \id)^*(\id \times \id \times c\times c)^*(\mu\times \mu)^* (z_3\otimes 1 + 1 \otimes z_3) \\
&=   (\Delta\times \Delta)^*(\id \times T \times \id)^*(\id \times \id \times c\times c)^*(z_3\otimes 1 \otimes 1 \otimes 1 + 1\otimes z_3 \otimes 1\otimes 1 \\
 & \phantom{=} + 1\otimes 1 \otimes z_3\otimes 1 + 1\otimes 1 \otimes 1\otimes z_3) \\
&= (\Delta\times \Delta)^*(\id \times T \times \id)^*( z_3\otimes 1 \otimes 1 \otimes 1 + 1\otimes z_3 \otimes 1\otimes 1 
  - 1\otimes 1 \otimes z_3\otimes 1  \\
  &\phantom{=} -1\otimes 1 \otimes 1\otimes z_3)  \\
&=(\Delta\times \Delta)^*( z_3\otimes 1 \otimes 1 \otimes 1 + 1\otimes 1 \otimes z_3\otimes 1 - 1\otimes z_3 \otimes 1\otimes 1 - 1\otimes 1 \otimes 1\otimes z_3) \\
&=z_3\otimes 1 + 1\otimes z_3 - z_3\otimes 1 - 1 \otimes z_3 \\
&=0  
\end{align*}
shows that  $\tilde\phi^*(z_3)$ lies in the kernel of the isomorphism $(\pi \times \pi)^*$.  Notice that the above calculation shows that $\phi^*(x)=0$ for any primitive cohomology class $x$.  

The difficulty in general---for primes $p$ that divide the order of $Z$---is that the covering projection $\pi\times \pi$ does not induce an injection on $H^3(-;\Z_p)$; therefore, to compute $\tilde\phi^*(z_3)$ it is necessary to compute in $G$ rather than $\tilde{G}$.   

The rest of the chapter will adopt the above notation for $z_3\in H^3(\tilde{G};\Z_p)$, the reduction mod $p$ of the integral generator of $H^3(\tilde{G};\Z)$, which by abuse of notation will also be denoted $z_3$.

Finally, the proofs of the Theorems in this Chapter share similarities.  
The proof of Theorem \ref{thm:An} is discussed in the most  detail, and serves as a model for the calculations that appear in the remaining Theorems.  To avoid repetition, details for parts of the proofs of these Theorems may be omitted.

\section{$G=PU(n)$, and $G=SU(n)/\Z_k$} \label{sec:interestingcase}

Consider $\tilde{G}=SU(n)$, with center $\Zn$. Then the central
subgroups of $\tilde{G}$ are cyclic groups $\Z_k$ where $k$
divides $n$. 

\begin{thm} \label{thm:An}  $l_0(PU(n))=n$, and more generally, $l_0(SU(n)/\Z_k)=\mathrm{ord}_k(\frac{n}{k})$, where $\mathrm{ord}_k(x)$ denotes the order of $x \mod k$ in $\Z_k$.
\end{thm}

\begin{proof} To begin, assume $G=PU(n)=\tilde{G}/\Z_n$. It will be shown that $\tilde\phi^*(z_3)$ generates a $\Z_n$-summand in $H^3(G\times G;\Z)$, and therefore $l_0(G)=n$.
\\

\begin{case} $n\neq 2\,\mathrm{mod}\,4$ 
\end{case}

In this case,  the computation for primes $p$ dividing $n$ may be carried out
simultaneously.  The reader may wish to refer to Theorem \ref{thm:BB} in the Appendix, regarding the cohomology of  $PU(n)$ as a Hopf algebra.

Before proceeding with the computation, it will be verified that there exist cohomology classes $x_1,y_2,x_3\in H^*(PU(n);\Z_p)$ such  that $\pi^*(x_3)=z_3$, the reduction mod $p$ of an integral generator in $H^3(SU(n);\Z)$, and $\mu^*(x_3)=x_3\otimes 1 + x_1\otimes y_2 + 1 \otimes x_3$. (Theorem \ref{thm:BB} does not explicitly guarantee the first of these equalities.) 

 In the Serre spectral sequence for the fibration sequence $SU(n)\to PU(n) \to B\Z_n$, the coefficient system is trivial (see Section \ref{sec:coholiegroup}), and
\begin{align*}
E_2 & \cong H^*(SU(n);\Z_p)\otimes H^*(B\Z_n;\Z_p) \\
&\cong \Lambda(z_3, \ldots, z_{2n-1})\otimes \Lambda(x) \otimes \Z_p[y]
\end{align*}
The first non-trivial differential is the transgression $d:E_4^{0,3} \to E_4^{4,0}$.
Working backwards from Theorem \ref{thm:BB}, $y^2\in E_4^{4,0}$ survives to $E_\infty$,  since $n\neq 2$ mod $4$; therefore, $z_3$ is in the image of $\pi^*$. That is, $\pi^*(\lambda x_3)=z_3$ for some $\lambda$, 
where $x_3$ is as in Theorem \ref{thm:BB}. Since $\bar\mu^*(\lambda x_3)=\lambda x_1\otimes y_2$, replacing $x_3$ with $\lambda x_3$, and $x_1$ with $\lambda x_1$ will suffice.

  Since $\pi\circ \tilde\phi=\phi$, it will  suffice to calculate $\phi^*(x_3)$, which is done next. Observe that $c^*(x_3)=-x_3+x_1y$.  Indeed,
\begin{align*}
0 & =  \Delta^*(1\times c)^*(x_3\otimes 1 + x_1\otimes y + 1\otimes x_3) \\
& =  \Delta^*(x_3\otimes 1 - x_1\otimes y + 1\otimes c^*(x_3)) \\
& =  x_3 - x_1 y +  c^*(x_3) \\
\end{align*}

To compute $\phi^*(x_3)$:
\begin{align*}
\phi^*(x_3) &= (\Delta\times\Delta)^*(1\times T\times
1)^*(1\times 1\times c\times c)^*(\mu\times \mu)^*\mu^*(x_3) \\
& = (\Delta\times\Delta)^*(1\times T\times 1)^*(1\times 1\times
c\times c)^*(x_3\otimes 1\otimes 1\otimes 1  \\
& \quad + x_1\otimes y \otimes 1\otimes 1 + 1\otimes x_3\otimes 1\otimes 1 + x_1\otimes 1\otimes
y\otimes 1 \\
& \quad + x_1\otimes 1\otimes 1\otimes y+1\otimes x_1 \otimes
y\otimes 1 + 1\otimes x_1 \otimes 1\otimes y \\
& \quad + 1\otimes 1\otimes x_3\otimes 1 + 1\otimes 1\otimes x_1\otimes y + 1\otimes 1\otimes1\otimes x_3)\\
\end{align*}
\begin{align*}
&= (\Delta\times\Delta)^*(x_3\otimes 1\otimes 1\otimes 1  +
x_1\otimes 1 \otimes y\otimes 1 + 1\otimes 1\otimes x_3\otimes 1 + \\
&\quad - x_1\otimes y\otimes 1\otimes 1 - x_1\otimes 1\otimes
1\otimes y - 1\otimes y \otimes x_1\otimes 1 \\
& \quad - 1\otimes 1 \otimes x_1\otimes y - 1\otimes x_3\otimes 1\otimes1 + 1\otimes x_1y\otimes 1
\otimes 1 \\
& \quad + 1\otimes x_1\otimes 1\otimes y
- 1\otimes 1\otimes1\otimes x_3 +  1\otimes 1\otimes 1\otimes x_1y)\\
&=x_1\otimes y - y\otimes x_1
\end{align*}

By using the Bockstein spectral sequence, it will be shown next that
$x_1\otimes y - y\otimes x_1$ in $H^3(G\times G;\Zp)$ is the reduction
$\mathrm{mod}\,p$ of an integral class that generates a $\Z_n$-summand. This will show that $\phi^*(z_3)$ has order $n$ in $H^3(G\times G;\Z)$, as required.

Write $n=p^rn'$ where $p$ does not divide $n'$.  As
$H_1(G;\Z)=\pi_1(G)=\Zn$, by the Universal Coefficient Theorem,
$H^2(G;\Z)$ contains a $\Zn\cong \Z_{p^r} \times \Z_{n'}$ summand.
This implies that $\beta^{(r)}(x_1)=y$ where $x_1$, and $y$ are
the cohomology classes in Theorem \ref{thm:BB}, and $\beta^{(r)}$
is the $r$-th Bockstein operator (see Section \ref{sec:spectralsequences}).
Therefore $\beta^{(r)}(x_1\otimes x_1)=y\otimes x_1 - x_1\otimes y $
in $H^*(G;\Zp)\otimes H^*(G;\Zp)\cong H^*(G\times G;\Zp)$.
This implies that $y\otimes x_1 - x_1\otimes y$ is the reduction
$\mathrm{mod}\,p$ of a generator of the $\Z_{p^r}$ summand in
$H^3(G\times G;\Z_{(p)})$ (i.e. ignoring torsion prime to $p$). Since this is true for all primes $p$ dividing $n$, $x_1\otimes y - y\otimes x_1 \in H^3(G\times G;\Zp)$ is the reduction $\mathrm{mod}\,p$ of an integral class that generates a $\Z_n$-summand. (See Remark \ref{remark:bockstein} in Section \ref{sec:spectralsequences}.)
\\

\begin{case} $n=2\,\mathrm{mod}\,4$
\end{case}

First observe that the calculations and conclusions from the previous case remain valid for all primes $p\neq 2$ dividing $n$. That is, for such primes, $\tilde\phi^*(z_3)=\phi^*(x_3)$ is the reduction mod $p$ of a generator of the $\Zn$ summand in $H^3(G\times G;\Z)$. The same is true for $p=2$, as shown next.

In other words, it will be shown that $\tilde\phi^*(z_3)=x_1\otimes x_1^2 + x_1^2\otimes x_1$, as above, although by a different method.  (The method in the previous case fails for $p=2$, since $z_3$ is not in the image of $\pi^*$.) In fact, as will be seen shortly, it will suffice to consider the case $n=2$.

To see this, note that there is a homomorphism $j:PU(2)\to PU(n)$ that induces isomorphisms on $H^q(\quad;\Ztwo)$ when $q\leq 4$. The homomorphism $j$ is induced by the diagonal inclusion $\iota: SU(2)\to SU(n)$,
$$
\iota(A)= \left(
\begin{array}{ccc}
A & &0  \\
 &  \rotatebox[origin=c]{-15}{$\ddots$}  & \\
 0 & & A\\
\end{array}\right),
$$
which includes the center $Z(SU(2)) \hookrightarrow Z(SU(n))$. Therefore,  $j_\sharp:\pi_1(PU(2)) \to \pi_1(PU(n))$ is the non-trivial homomorphism $\Z_2\hookrightarrow \Z_n$. Since $H_1(G;\Z_2)\cong \pi_1(G)\otimes \Z_2$, it follows that the induced map $j_*:H_1(PU(2);\Z_2)  \to H_1(PU(n);\Z_2)$ is an isomorphism, and hence the same is true for $j^*:H^1(PU(n);\Z_2)\to H^1(PU(2);\Z_2)$.  Consulting Theorem \ref{thm:BB}, it follows that the ring homomorphism $j^*$ is an isomorphism in degrees 2, 3, and 4 as well.

Also, it will be verified shortly that $\iota^*:H^3(SU(n);\Z_2) \to H^3(SU(2);\Z_2)$ is an isomorphism.  Therefore, from the commutative diagram of homomorphisms
\begin{diagram}
SU(2) & \rTo^\iota & SU(n) \\
\dTo & & \dTo \\
PU(2) & \rTo^j & PU(n) \\
\end{diagram}
it follows that $\iota \circ \tilde\phi_{PU(2)} = \tilde\phi_{PU(n)} \circ (j\times j)$, and hence $\tilde\phi^*_{PU(n)}(z_3)=x_1\otimes x_1^2 + x_1^2\otimes x_1$ if and only if the same is true for $\tilde\phi^*_{PU(2)}$. Therefore, it suffices to consider the case $G=PU(2)$. 

To check that  $\iota^*:H^3(SU(n);\Z_2) \to H^3(SU(2);\Z_2)$ is an isomorphism, it will suffice to check  that with integer coefficients, $\iota^*:H^3(SU(n);\Z) \to H^3(SU(2);\Z)$ is multiplication by $m$, where $n=2m$ and $m$ is odd.  To that end, recall that the various inclusions $\iota_k:SU(2)\to SU(n)$ for $k=1, \ldots, m$
$$
\iota_1(A)= \left(
\begin{array}{cccc}
A & & & 0  \\
    & I & & \\
  & &  \rotatebox[origin=c]{-15}{$\ddots$}  & \\
 0 & & & I\\
\end{array}\right), \quad 
\iota_2(A)= \left(
\begin{array}{cccc}
I & & & 0  \\
    & A & & \\
  & &  \rotatebox[origin=c]{-15}{$\ddots$}  & \\
 0 & & & I\\
\end{array}\right), \quad \text{etc.}
$$
each induce an isomorphism on $H^3(-;\Z)$.  Indeed, each of the $\iota_k$ may be written as a composition of the standard inclusions $SU(d)\to SU(d+1)$ that appear in the fibration sequences
$$ 
SU(d) \to SU(d+1) \to SU(d+1)/SU(d)=S^{2d+1}.
$$
The maps $SU(d)\to SU(d+1)$ induce isomorphisms on $H^3(-;\Z)$ for connectivity reasons (using the Serre spectral sequence, for example), hence the $\iota_k$ do as well.  Finally, $\iota$ may be written as the composition
$$
\iota = \mu^m \circ (\iota_1 \times \cdots \times \iota_m) \circ \Delta^m,
$$
where $\mu^m:SU(n)\times \cdots \times SU(n)\to SU(n)$ denotes repeated multiplication $\mu^m(X_1, \ldots, X_m)=X_1\cdots X_m$, and $\Delta^m:SU(2) \to SU(2) \times \cdots \times SU(2)$ denotes the diagonal $\Delta^m(A)=(A, \cdots, A)$. Since $z_3\in H^3(SU(n);\Z)$ is primitive, it follows easily that $\iota$ induces multiplication by $m$.

Recall that  $PU(2)\cong SO(3)$, which is homeomorphic to $\R P^3$, and that $SU(2)$ is homeomorphic to the sphere $S^3$, so that the covering $\pi:SU(2) \to PU(2)$ is simply the covering $S^3 \to \R P^3$. Since much of what follows depends very heavily on the topological identifications $SU(2)=S^3$, and $PU(2)=\R P^3$, the notation $\pi:S^3 \to \R P^3$ shall be adopted to denote the universal covering homomorphism $SU(2)\to PU(2)$.

Since its restriction to $G\vee G$ is null homotopic,  the commutator map $\phi:G\times G\to G$ for a connected topological group $G$ induces a map $\varphi:G\wedge G \to G$. (As in Section \ref{sec:wheadprod},  $\varphi=\langle \mathrm{id},\mathrm{id} \rangle$ where $\id:G\to G$ is the identity map.)  Since $G\wedge G$ is simply connected, there is a unique lift $\tilde\varphi: G\wedge G \to \tilde{G}$ to the universal covering group $\tilde{G}$. In other words, there is a commutative diagram
\begin{diagram}
G \times G & \rTo^{\tilde\phi} & \tilde{G} \\
\dTo & \ruTo^{\tilde\varphi} & \dTo \\
G\wedge G & \rTo^\varphi & G \\
\end{diagram}
Since  the induced map $H^*(G\wedge G;R) \to H^*(G\times G;R)$ is injective for any coefficient ring $R$, $\tilde\phi^*(z_3)=0$ if and only if $\tilde\varphi^*(z_3)=0$.

 Consider now the case at hand: namely, $G=\RP^3$.  To begin, recall the cohomology of $\RP^3\wedge \RP^3$ with $\Z_2$ coefficients.  In particular, consider the following description of the vector spaces $H^q(\RP^3\wedge \RP^3;\Z_2)$ for $q\leq 4$, with basis listed in brackets $\langle \quad \rangle$:
$$
 H^q(\RP^3\wedge\RP^3;\Z_2)=\left\{ \begin{array}{cl}
0                                           & \text{if } q=1 \\
 \langle x_1\otimes x_1 \rangle                        & \text{if } q=2 \\
 \langle x_1\otimes x_1^2, x_1^2\otimes x_1 \rangle   &\text{if } q=3 \\
 \langle x_1\otimes x_1^3, x_1^2\otimes x_1^2, x_1^3\otimes x_1 \rangle & \text{if } q=4\\
\end{array} \right.  
$$
Here, the injection $H^*(\RP^3\wedge \RP^3;\Z_2) \to H^*(\RP^3\times \RP^3;\Z_2)$ permits the borrowing of notation from Theorem \ref{thm:BB}.  Write $\tilde\varphi^*(z_3)=a(x_1\otimes x_1^2 + x_1^2\otimes x_1) + b(x_1\otimes x_1^2)$ for some $a$ and $b$ in $\Z_2$.  The next paragraph shows that $b=0$.

Since $\beta x_1=x_1^2$ in $H^*(\RP^3;\Z_2)$, where $\beta$ is the mod 2 Bockstein operator (see Section \ref{sec:spectralsequences}), it follows that $\beta(x_1\otimes x_1)=x_1^2\otimes x_1 + x_1\otimes x_1^2$, and that $\beta(x_1\otimes x_1^2)=\beta(x_1^2\otimes x_1)=x_1^2\otimes x_1^2$ in $H^*(\RP^3\wedge \RP^3;\Z_2)$.  Since $\beta(z_3)=0$ in $H^3(S^3;\Z_2)$, it follows that $\beta \tilde\varphi^*(z_3)=\tilde\varphi^*\beta(z_3)=0$ by naturality. 
But since
\begin{align*}
\beta(a(x_1\otimes x_1^2 + x_1^2\otimes x_1) + b(x_1\otimes x_1^2))&=a\beta(\beta(x_1\otimes x_1)) + b \beta(x_1\otimes x_1^2) \\
&=b(x_1^2\otimes x_1^2),
\end{align*}
$b$ must be zero.
 Therefore, either $\tilde\varphi^*(z_3)=x_1\otimes x_1^2 + x_1^2\otimes x_1$ or $\tilde\varphi^*(z_3)=0$. 

The strategy for showing that $\tilde\varphi^*(z_3)=x_1\otimes x_1^2 + x_1^2\otimes x_1$ will be to first show that there is a map $f:\Sigma \RP^2 \to \RP^3\wedge \RP^3$  that has the property that  $f^*(x_1\otimes x_1^2 + x_1^2\otimes x_1)$ in $H^3(\Sigma \RP^2; \Z_2)$ is non zero.  In order to arrive at the desired conclusion about $\tilde\varphi^*(z_3)$, it will then suffice to check that the composition $\tilde\varphi\circ f:\Sigma \RP^2\to S^3$ induces a non-zero homomorphism on $H^3(-;\Z_2)$.  As will be verified later, there are only two homotopy classes of maps $\Sigma \RP^2 \to S^3$ (see Fact \ref{fact1}), and the essential maps induce isomorphisms on $H^3(-;\Z_2)$ (see Fact \ref{fact2}).  Therefore, it will be necessary to check that $\tilde\varphi\circ f$ is essential (see Proposition \ref{prop:essential}), which will provide the desired conclusion.

In order to realize the map $f$, consider the inclusion $\RP^2\wedge \RP^2 \subset \RP^3 \wedge \RP^3$, which induces an isomorphism in cohomology $H^q(-;\Z_2)$ for $q\leq 3$.  Since $H_2(\RP^2 \wedge \RP^2;\Z)\cong \Z_2$ (using the Kunn\"eth Theorem, for example), the first non-vanishing homotopy group of $\RP^2\wedge \RP^2$ is $\pi_2(\RP^2 \wedge \RP^2)\cong \Z_2$.  If $a:S^2 \to \RP^2\wedge \RP^2$ represents a generator of $\pi_1(\RP^2 \wedge \RP^2)$, then the composition $S^2 \toby{2} S^2 \toby{a} \RP^2 \wedge \RP^2$ is null homotopic. Therefore, by Proposition \ref{prop:extend}, there is an extension  $\Sigma \RP^2 \to \RP^2 \wedge \RP^2$, which may be composed with the inclusion $\RP^2\wedge \RP^2 \subset \RP^3\wedge \RP^3$ to define $f:\Sigma \RP^2 \to \RP^3 \wedge \RP^3$.  

It remains to check that $f^*$ is non-zero on $H^3(-;\Z_2)$.  Since $a_*:H_2(S^2;\Z) \to H_2(\RP^2 \wedge \RP^2;\Z)$ is onto, then by the Universal Coefficient Theorem, $a_*:H_2(S^2;\Z_2) \to H_2(\RP^ \wedge \RP^2;\Z_2)$ is an isomorphism, and hence $a^*:H^2(\RP^2 \wedge \RP^2;\Z_2)\to H^2(S^2;\Z_2)$ is an isomorphism as well.  Using the suspension isomorphism $\sigma:H^*(X;R) \to H^{*+1}(\Sigma X;R)$, let $\sigma x\in H^2(\Sigma \RP^2;\Z_2)$ be a generator, so that $\beta \sigma x=\sigma \beta x \in H^3(\Sigma \RP^2;\Z_2)$ is also a generator.  Then $f^*(x_1\otimes x_1) = \sigma x$, and hence 
\begin{align*}
f^*(x_1\otimes x_1^2 + x_1^2 \otimes x_1) &= f^*(\beta(x_1\otimes x_1) \\
&=\beta(f^*(x_1\otimes x_1)) \\
&=\beta \sigma x
\end{align*}
is non-zero, as required.

\begin{fact} \label{fact1}
The set of homotopy classes of maps $[\Sigma\RP^2,S^3]=\Ztwo$.
\end{fact}
\begin{proof} From the cofibration sequence
\[ S^2\to\Sigma\RP^2 \to S^3 \to S^3 \]
 there is a long exact sequence
\[ [S^3,S^3] \toby{-\times 2} [S^3,S^3] \to [\Sigma\RP^2,S^3] \to [S^2,S^3] \]
which gives the result. \end{proof}

\begin{fact} \label{fact2}
The essential maps $\Sigma\RP^2\to S^3$ induce isomorphisms on
$H^3(-;\Z_2)$.
\end{fact}
\begin{proof} Again, consider the cofibration sequence
$S^2\to\Sigma\RP^2 \to S^3$ and the corresponding long exact
sequence in cohomology. \end{proof}

\begin{prop} \label{prop:essential}
The composition $\Sigma \RP^2 \toby{f} \RP^3 \wedge \RP^3 \toby{\tilde\varphi} S^3$ is essential.
\end{prop}
\begin{proof}
It is equivalent to show that the composition $\Sigma\RP^2
\toby{\tilde\varphi f} S^3\toby{\pi} \RP^3$ is essential, since  $\Sigma\RP^2$ is
simply connected, and hence $[\Sigma\RP^2,\RP^3]=[\Sigma\RP^2,S^3]$.
It will be shown next that the map $\pi\tilde\varphi f$ is (homotopic to) either composition in the following  diagram:
\begin{diagram}
\Sigma\RP^2 & \rTo^i & \RP^2\wedge\RP^2 & \rTo & \RP^3\wedge\RP^3
\\
  &   & \dTo^{\langle\alpha, \alpha\rangle} &   & \dTo_{\varphi} \\
  &   & \Omega\Sigma\RP^2 & \rTo^{\Omega(\rho)} & \RP^3 \\
\end{diagram}

The map labeled  $\alpha:\RP^2\to
\Omega\Sigma\RP^2$ is the adjoint of the identity map on
$\Sigma\RP^2$, and $\langle\alpha, \alpha\rangle$ is the
Samelson Product of $\alpha$ with itself (see Section \ref{sec:wheadprod}).
The map $\Omega\Sigma\RP^2\to \RP^3$ along the bottom of the diagram is described next.  The adjoint of the inclusion $\RP^2\hookrightarrow \RP^3 \approx \Omega B\RP^3$ gives a map $\rho:\Sigma\RP^2\to B\RP^3$, and hence a map $\Omega(\rho):\Omega\Sigma\RP^2\to \Omega B\RP^3\approx \RP^3$.  An important point is that this map is an H-map, so that (up to homotopy) commutators map to commutators.  This shows that the diagram homotopy commutes, as claimed.

Now suppose for contradiction that $\pi\circ \tilde\varphi\circ f$, or equivalently, the composition 
\begin{diagram}
\Sigma\RP^2 & \rTo & \RP^2\wedge \RP^2 & \rTo^{\langle \alpha,\alpha \rangle}& \Omega\Sigma \RP^2 & \rTo^{\Omega(\rho)} &\RP^3
\end{diagram}
is null homotopic. Then there exists a lift $g:\Sigma\RP^2 \to \Omega Z$ where $Z$ is the homotopy fiber of the map $\rho:\Sigma \RP^2 \to B\RP^3$ (and hence $\Omega Z$ is the homotopy fiber of the map $\Omega (\rho)$).
Consider the restriction of $g$ to the $2$-skeleton of $\Sigma\RP^2$, namely $S^2 \to\Sigma \RP^2 \to \Omega Z$.  Since $\pi_2(\Sigma RP^2)=\Ztwo$ is  torsion, then according to the following fact from Wu \cite{W}, the map $S^2 \to\Sigma \RP^2 \to \Omega Z$ is null homotopic.

\begin{fact} (Proposition 6.6 in \cite{W}) $\pi_2(\Omega Z) \cong \Z$
\end{fact}

But then the composition $h:S^2\to \Sigma\RP^2 \to \Omega Z \to \Omega \Sigma
\RP^2$ is null homotopic, which will turn out to be a contradiction by the following fact, also from \cite{W}.

\begin{fact} (Proposition 6.5 in \cite{W}) $\pi_3(\Sigma\RP^2) = \Z_4$ is generated by the
composition $S^3\toby{\eta}S^2\to \Sigma\RP^2$, where $\eta$ is
the Hopf map.
\end{fact}

It is a contradiction because the map $h$ is actually the adjoint of twice the generator of
$\pi_3(\Sigma\RP^2)$. To see this, recall that $h$ is the composition of $S^2 \to \Sigma\RP^2 \to \RP^2\wedge \RP^2$, which is the inclusion of the 2-skeleton, and $\RP^2 \wedge \RP^2  \to \Omega\Sigma \RP^2$, which is the Samelson product $\langle \alpha, \alpha \rangle$. Therefore $h$ is either composition in the diagram below:
\begin{diagram}
S^2 & \rEquals & S^1\wedge S^1 & \rTo^{j\wedge j} & \RP^2\wedge\RP^2 \\
 & & \dTo<{\langle\iota, \iota\rangle} &   & \dTo>{\langle\alpha, \alpha\rangle}\\
 & & \Omega S^2  &\rTo^{\Omega\Sigma j}   & \Omega\Sigma\RP^2 \\
\end{diagram}
where $j:S^1\to\RP^2$ is the inclusion of the 1-skeleton, and
$\iota$ is the adjoint of the identity on $S^2$. (The square commutes, since  $\Omega\Sigma j\circ \iota= \alpha\circ j$, and Samelson product is functorial.)

Recall that the adjoint of $h$ is the composition $\mathrm{ad}^{-1}h: S^3 \toby{\Sigma h} \Sigma \Omega \Sigma \RP^2 \toby\beta \Sigma \RP^2$, where $\beta$ denotes the adjoint of the identity map $\Omega \Sigma \RP^2 \to \Omega \Sigma \RP^2$ (see Section \ref{sec:loopsetc}).  By naturality, it is easy to see that  $\mathrm{ad}^{-1}h$ is either composition in the diagram below.
\begin{diagram}
S^3 & \rTo^{\Sigma\langle \iota,\iota \rangle} & \Sigma \Omega S^2 & \rTo^{\Sigma \Omega \Sigma j} & \Sigma \Omega \Sigma \RP^2 \\
   & \rdTo^{[\iota,\iota]} & \dTo>\beta & & \dTo>\beta \\
    & & S^2 & \rTo^{\Sigma j} & \Sigma \RP^2 \\
\end{diagram}
where $[\iota,\iota]$ denotes the Whitehead product, which is homotopic to $2\eta$ where $\eta$ is the Hopf map (see Example \ref{eg:whiteheadistwicehopf}). Therefore the composition $h$ is indeed the adjoint of twice the generator of $\pi_3(\Sigma\RP^2)$. 
\end{proof}

This completes the proof of the case $G=PU(n)$.
\\

The case $G=SU(n)/\Z_k$, where $k|n$,  will now be considered.  It will be shown that $l_0(G)$ is the order of $q=\frac{n}{k} \mod k$ in $\Z_k$.

Consider the covering homomorphism $f:G\to PU(n)$ with kernel $\Z_n/\Z_k \cong \Z_{n/k}$, and the commutative diagram of homomorphisms:
\begin{diagram}
SU(n)  &  \rEquals& SU(n) \\
\dTo &        & \dTo \\
G & \rTo^{f} & PU(n) \\
\end{diagram}
This shows $\tilde{\phi}_G=\tilde{\phi}_{PU(n)}\circ(f\times f)$.  Therefore, $\tilde\phi_G^*(z_3) = (f\times f)^*\tilde\phi_{PU(n)}^*(z_3)$ in $H^3(G\times G;\Z)$.   Since $\tilde\phi_{PU(n)}^*(z_3)$ generates the $\Z_n$-summand in $H^3(PU(n)\times PU(n);\Z)$, it will suffice to check that the induced map
 $$
 H^3(PU(n)\times PU(n);\Z) \to H^3(G\times G;\Z)
 $$ 
 sends a generator of the $\Z_n$-summand to $\frac{n}{k}$ times a generator of the $\Z_k$-summand in $H^3(G\times G;\Z)$.
 
 To that end, recall that the K\"unneth Theorem says that the cross product defines an injective homomorphism 
 $$
 \bigoplus_{r+s=3} H^r(PU(n);\Z) \otimes H^s(PU(n);\Z) \toby{\times} H^3(PU(n)\times PU(n);\Z)
 $$
with cokernel naturally isomorphic to $\displaystyle \bigoplus_{r+s=4}\mathrm{Tor}(H^r(PU(n);\Z),H^s(PU(n);\Z))$, and similarly for $G$. In other words, the torsion summands in  $H^3(PU(n)\times PU(n);\Z)$, and $H^3(G\times G;\Z)$ are naturally isomorphic to the corresponding $\mathrm{Tor}$-terms, the cokernels of the cross product homomorphisms.  Therefore, it will be necessary to compute the induced map $\mathrm{Tor}(f^*,f^*)$ in the following diagram (with integer coefficients):
 \begin{diagram}
 H^3(PU(n) \times PU(n)) & \rTo & \mathrm{Tor} (H^2(PU(n)),H^2(PU(n))) & \rTo 0 \\
 \dTo<{(f\times f)^*} & & \dTo>{\mathrm{Tor}(f^*,f^*)} & &\\
 H^3(G \times G) & \rTo & \mathrm{Tor} (H^2(G),H^2(G) )& \rTo 0 \\
 \end{diagram}
 
 To compute $f^*:H^2(PU(n);\Z) \to H^2(G;\Z)$,  use the Universal Coefficient Theorem.  Specifically,  $H^2(PU(n);\Z)$ is naturally isomorphic to $\mathrm{Ext}(H_1(PU(n);\Z),\Z)$, and similarly for $G$, so that the following diagram commutes.
 \begin{diagram}
 \mathrm{Ext}(H_1(PU(n);\Z),\Z) &\rTo^\cong & H^2(PU(n);\Z) \\
 \dTo<{\mathrm{Ext}(f_*,1)} & &\dTo>{f^*} \\
 \mathrm{Ext}(H_1(G);\Z),\Z) & \rTo^\cong & H^2(G;\Z) \\
 \end{diagram}
 Recall that the inclusion $f_*:H_1(G;\Z) \cong \Z_k \to H_1(PU(n);\Z) \cong \Z_n$ sends the generator in $H_1(G;\Z)$ to $\frac{n}{k}$ times the generator of $H_1(PU(n);\Z)$.  Therefore, from the following commutative diagram of free resolutions,
 \begin{diagram}
 0 &\rTo & \Z & \rTo^k & \Z &\rTo & \Z_k & \rTo & 0 \\
    &        &  \dEquals &   & \dTo>{\frac{n}{k}} & & \dTo>{f_*} & & \\
 0 &\rTo & \Z & \rTo^n & \Z &\rTo & \Z_n  &\rTo & 0\\
 \end{diagram}
 it follows that the induced map $\mathrm{Ext}(f_*,1)$ sends a generator of $\mathrm{Ext}(H_1(PU(n);\Z),\Z) \cong \Z_n$ to a generator of $\mathrm{Ext}(H_1(G;\Z),\Z) \cong \Z_k$.  Therefore, under the identifications $H^2(PU(n);\Z)\cong \Z_n$, and $H^2(G;\Z)\cong \Z_k$, $f^*:\Z_n\to \Z_k$ is simply $f^*(1)=1$. 
 
 The induced map $\mathrm{Tor}(f^*,f^*):\mathrm{Tor}(\Z_n,\Z_n) \to \mathrm{Tor}(\Z_k,\Z_k)$ may be computed as the composition
\begin{diagram}
 \mathrm{Tor}(\Z_n,\Z_n)& \rTo^{\mathrm{Tor}(f^*,1)}& \mathrm{Tor}(\Z_k,\Z_n) &\rTo^{\mathrm{Tor}(1,f^*)}& \mathrm{Tor}(\Z_k,\Z_k).
\end{diagram}

Applying $-\otimes \Z_n$ to the diagram of free resolutions,
\begin{diagram}
 0 &\rTo & \Z & \rTo^n & \Z &\rTo & \Z_n & \rTo & 0 \\
    &        &  \dTo>{\frac{n}{k}}&   & \dEquals & & \dTo>{f^*} & & \\
 0 &\rTo & \Z & \rTo^k & \Z &\rTo & \Z_k  &\rTo & 0\\
 \end{diagram}
and taking homology shows that $\mathrm{Tor}(f^*,1)$ sends a generator of $\mathrm{Tor}(\Z_n,\Z_n) \cong \Z_n$ to a generator of $\mathrm{Tor}(\Z_k,\Z_n)=\ker\{\Z_n\toby{k} \Z_n\}$.  Similarly, applying $\Z_k\otimes -$ to the same diagram and taking homology shows that $\mathrm{Tor}(1,f^*)$ sends a generator of $\mathrm{Tor}(\Z_k,\Z_n) \cong \Z_k$ to $\frac{n}{k}$ times a generator of $\mathrm{Tor}(\Z_k,\Z_k)\cong \Z_k $.  Therefore, $\mathrm{Tor}(f^*,f^*)$ sends a generator of $\mathrm{Tor}(\Z_n,\Z_n)$ to $\frac{n}{k}$ times a generator of $\mathrm{Tor}(\Z_k,\Z_k)$, as required.

This completes the proof of Theorem \ref{thm:An}. \end{proof}


\section{$G=PSp(n)$}

\begin{thm}  If $n$ is even, then $l_0(PSp(n))=1$. If $n$ is odd, then $l_0(PSp(n))=2$.
\end{thm}
\begin{proof} First suppose $n$ is even. From the map
$g:Sp(n)\to SU(2n)$, which is known to induce an epimorphism on
cohomology (see \cite{BB}), there is a diagram:
\begin{diagram}
     Sp(n)     &\rTo^{g} & SU(2n)  \\
\dTo  &     & \dTo  \\
 PSp(n)     &\rTo^{\bar{g}} & SU(2n)/\Ztwo
\end{diagram}
\noindent which shows $\tilde\phi_{Sp(n)}^*(z'_3)=\tilde\phi_{Sp(n)}^*(g^*(z_3))=
(\bar{g}\times\bar{g})^*(\tilde\phi^*_{SU(2n)}(z_3))=0$, by Theorem \ref{thm:An}, where $z'_3=g^*(z_3)$ denotes  the generator of $H^3(Sp(n);\Z)$.  

Suppose $n$ is odd. The diagonal inclusion $\iota:Sp(1)\to Sp(n)$ induces a map $j:PSp(1)\to PSp(n)$.  And as in the proof of Theorem \ref{thm:An}, $\iota$ and $j$ induce isomorphisms on $H^q(\quad;\Ztwo)$ in dimensions $q\leq 3$. But $Sp(1)=SU(2)$, and this was covered in the previous section.  Hence $\tilde\phi^*(z_3)$ is the reduction mod $2$ of the generator of
the torsion summand. 
\end{proof}


\section{$G=SO(n)$, $n\geq 7$}



Recall that $G=SO(n)$ has  universal covering group $\tilde{G}=Spin(n)$. The covering projection $\pi:Spin(n)\to SO(n)$ induces an epimorphism on $H^3(\quad;\Ztwo)$, and according
to ~\cite{MT}, the generator of $H^3(Spin(n);\Ztwo)$ is $\pi^*(x_3)$, where $x_3\in H^3(SO(n);\Ztwo)$ is primitive.
It therefore suffices to calculate $\phi^*(x_3)$.  But since $x_3$ is
primitive,  $\phi^*(x_3)=0$. This proves:

\begin{thm} $l_0(SO(n))=1$.
\end{thm}




\section{$G=PO(2n)$, $n\geq 4$}

\begin{thm} If $n$ is odd, then $l_0(PO(2n))=4$. If $n$ is even, then $l_0(PO(2n))=2$. \label{lateruse}
\end{thm}

\begin{proof} The proof is very similar to the proof of Theorem \ref{thm:An}, and will be outlined next.  
 
Suppose first that $n\neq 2 $ mod 4.  By Theorem \ref{thm:PO}, it is easy to see that the universal covering projection $\pi:Spin(2n) \to PO(2n)$ induces an epimorphism on $H^3(\quad;\Ztwo)$.  (This can be seen by using the Serre spectral sequence for the fibrations $Spin (2n) \to PO(2n) \to B\Z_4$, for $n$ odd, and $Spin (2n) \to PO(2n) \to B\Z_2 \times B\Z_2$, for $n$ even, and working backwards, as in the proof of Theorem \ref{thm:An}.)  

In particular,  $\pi^*(u_3)=z_3\in H^3(Spin(2n);\Z_2)$, and $\tilde\phi^*(z_3)=\phi^*(u_3)$.  When $n$ is odd, using Theorem \ref{thm:PO}, it is easy to verify that  $\phi^*(u_3)=v_1\otimes u_2 + u_2\otimes v_1$.  Since $H^2(PO(2n;\Z)\cong \Z_4$, it follows that $\beta^{(2)}(v_1)=u_2$; therefore,  $= \beta^{(2)}(v_1\otimes v_1)=\phi^*(u_3)$. That is, $\phi^*(u_3)$ is the reduction mod 2 of a generator of the $\Z_4$ summand in $H^3(PO(2n)\times PO(2n);\Z)$, and hence $l_0(PO(2n))=4$.

When $n$ is even, it is easy to check that $\phi^*(u_3)=u_1\otimes v_1^2 + u_2\otimes v_1 + v_1^2\otimes u_1 + v_1\otimes u_2$.  This time, $\beta(v_1)=v_1^2$, and $\beta(u_1)=u_2$, which yields $\phi^*(u_3) = \beta(u_1\otimes u_1) + \beta(v_1\otimes u_1)$, showing that $l_0(PO(2n))=2$.

Suppose next that  $n= 2 $ mod 4, and write $2n=4k$, where $k$ is odd. Since $H^3(PO(4k)\times PO(4k);\Z)$ contains only 2-torsion, it will suffice to show that $\tilde\phi^*\neq 0$ on $H^3(\quad;\Z)$. To that end, recall that the center of the spinor group $Spin(4k)$ has two central subgroups of order 2. If $a$ denotes the generator of the kernel of the double cover $Spin(4k)\to SO(4k)$, and $b$ denotes another generator of the center $Z(Spin(4k))\cong\Ztwo \oplus \Ztwo$, then the  \emph{semi-spinor group} $Ss(4n)$ is the quotient $Spin(4k)/\langle b \rangle$ \cite{IKT}.  In particular, there is a commutative diagram
\begin{diagram}
Spin(4k) & \rEquals & Spin(4k) \\
\dTo & & \dTo \\
Ss(4k) & \rTo^{f} & PO(4k)
\end{diagram}
Therefore, $(f\times f)^*\tilde\phi_{PO(4k)}^*=\tilde\phi^*_{Ss(4k)}$. Since $\tilde\phi^*_{Ss(4k)}\neq 0$ (by Theorem \ref{mapcon}), the map $\tilde\phi^*_{PO(4k)}$ is non-zero. \end{proof}


\section{$G=Ss(4n)$}

\begin{thm}  If $n$ is odd, then $l_0(Ss(4n))=2$. If $n$ is
even, then $l_0(Ss(4n))=1$. \label{mapcon}
\end{thm}

\begin{proof}  For $n$ even,  the proof is similar to Theorems  \ref{thm:An} and\ref{lateruse}.  In this case, the covering projection $Spin(4n)\to Ss(4n)$ induces an epimorphism on $H^3(\quad;\Ztwo)$ (by looking at Theorem \ref{thm:Ss}, and working backwards from the Serre spectral sequence), and since the generator in $H^3(Ss(4n);\Ztwo)$ is primitive, $\phi^*=0$.

If $n$ is odd,  there is a map $g:\RP^3 \to Ss(4n)$ (constructed in the next paragraph) that induces a weak homotopy equivalence in degrees  $\leq 3$.  Therefore $g$ is covered by a map $\tilde{g}:S^3 \to Spin(4n)$ that induces an isomorphism on $H^3(\quad;\Z)$. And since $\mu_{\RP^3}^*g^* = (g\times g)^*\mu_{Ss(4n)}^*$ on $H^3(\quad;\Z)$, it follows that $\tilde\phi_{\RP^3}^*\tilde{g}^*=(g\times g)^*\tilde\phi_{Ss(4n)}^*$. Therefore this case reduces to Theorem \ref{thm:An}, and the image of $\tilde\phi^*$ is $2$-torsion.

The map $g$ may be constructed as follows. Recall that $\RP^3$ is obtained (as a CW complex) from $\RP^2$ by attaching a 3-cell with attaching map $S^2\to \RP^2$ the quotient map. Similarly, $\RP^2$ is obtained from $S^1$ by attaching a 2-cell with attaching map $2:S^1\to S^1$, the degree $2$ map.  Let $f:S^1 \to Ss(4n)$ represent a generator of $\pi_1(Ss(4n))=\Ztwo$. The composition $f\circ 2$  is null homotopic, and therefore extends to a map $F: \RP^2 \to Ss(4n)$. And since the composition $S^2 \to \RP^2 \to Ss(4n)$ is null homotopic  (as $\pi_2(Ss(4n))=0$), $F$ extends to a map $g:\RP^3 \to Ss(4n)$. By the structure of the cohomology rings  of $\RP^3$ and $Ss(4n)$ (see \cite{IKT}),  $g$ induces an isomorphism in $H^q(\quad;\Z)$ for $q\leq3$, since $g$ induces an isomorphism on $\pi_1(\quad)$. \end{proof}


\section{Exceptional Lie groups}

\begin{thm}$l_0(PE_6)=3$, and $l_0(PE_7)=2$.
\end{thm}

\begin{proof} Similar to the proof of Theorems \ref{thm:An}, \ref{lateruse}, and \ref{mapcon}, the covering projection $E_6\to PE_6$ induces an epimorphism on $H^3(\quad;\Z_3)$; therefore, $\tilde\phi^*(z_3)=\phi^*(x_3)$. Using Theorem \ref{thm:PEsix},  $\phi^*(x_3)=y_2\otimes x_1 - x_1\otimes y_2$, which is the reduction mod 3 of the generator of the $\Z_3$ summand, and hence $l_0(PE_6)=3$.

For $PE_7$, there is a map $g:\RP^3\to PE_7$ that is a weak homotopy 
equivalence in dimensions $\leq 3$, as in the proof of Theorem \ref{mapcon}. 
Therefore, as in Theorem \ref{mapcon}, $l_0(PE_7)=l_0(PU(2))=2$.
 \end{proof}


\chapter{Hamiltonian Loop Group Actions} \label{chapter:loopgroups}

The theory of quasi-Hamiltonian $G$-spaces is known to be equivalent to the theory of Hamiltonian loop group actions with proper moment map \cite{AMM}.  The latter is an infinite dimensional analogue of  Hamiltonian actions of compact Lie groups, and hence naturally affords a study of quantization.  Having studied pre-quantization in the framework of quasi-Hamiltonian $G$-spaces in Chapter \ref{chapter:prequantization}, this Chapter will address  pre-quantization in the framework of Hamiltonian loop group actions, based on the unpublished work of A. Alekseev and E. Meinrenken \cite{AM}.  Moreover, it will be shown that via the aforementioned equivalence of theories, the pre-quantization of Hamiltonian loop group actions corresponds to the pre-quantization of the associated quasi-Hamiltonian $G$-space.

The  technical issues that arise when dealing with infinite dimensional manifolds (regarding smoothness conditions, etc.)  will be ignored in this work.  The reader may consult \cite{MW}, where these details are considered.  Throughout this Chapter, $G$ will denote a compact, connected, simply connected Lie group.

\section{Loop groups}

This section surveys the necessary background concerning loop groups, and their central extensions.  The reader interested in a  more detailed treatment should consult  \cite{PS}, or \cite{KW}.  

Let $G$ be a compact, connected, simply connected Lie group, with Lie algebra $\mathfrak{g}$, and choose an  $\mathrm{Ad}$-invariant inner product $(-,-)$  on $\mathfrak{g}$.

The \emph{loop group} of $G$ is  the free loop space $LG=\mathrm{Map}(S^1,G)$, viewed as an infinite dimensional manifold, equipped with the pointwise group multiplication.  A vector $\xi$ in the tangent space at the identity element, the constant map $z\mapsto 1$, is of the form 
$$
\xi=\frac{\partial}{\partial t}\Big|_{t=0} \gamma(t,z),
$$
where $\gamma:(-\epsilon,\epsilon)\times S^1 \to G$ with $\gamma(0,z)=1$.  Therefore, $\xi=\xi(z)\in \mathfrak{g}$, defines a loop  in  $\Lg=\mathrm{Map}(S^1,\mathfrak{g})$ as $z\in S^1$ varies. The Lie algebra of $LG$, sometimes called the loop Lie algebra, is therefore $\Lg$, with pointwise Lie bracket.  The exponential map $\exp:\Lg \to LG$ is simply the pointwise exponential map

The  fibration $LG\to G$ given by evaluation at the base point $1\in S^1$, with fibre $\Omega G$, has a section $s:G\to LG$ that sends $g$ to the constant loop $s(g)(z)=g$.  Therefore, $LG$ is topologically the product $\Omega G\times G$.  However, the isomorphism $LG\to \Omega G\rtimes G$ given by $\omega(z)\mapsto (\omega(z)\omega(1)^{-1}, \omega(1))$ shows that $LG$ is isomorphic to the semidirect product $\Omega G \rtimes G$, with multiplication $(\omega_1(z), g_1)\cdot(\omega_2(z),g_2)=(\omega_1(z)g_1\omega_2(z) g_1^{-1},g_1g_2)$.

\subsection*{Central extensions}

A particular central extension $1\to S^1 \to \LGhat \to LG \to 1$ of the loop group $LG$ by $S^1\cong U(1)$  plays an important role in the theory of Hamiltonian loop group actions.   Studying this central extension, and its co-adjoint representation helps strengthen the analogy between Hamlitonian loop group actions and classical Hamiltonian $G$-actions of a compact Lie group $G$.  

To begin, consider the central extension of the loop Lie algebra $\Lghat:=\Lg\oplus \R$ with Lie bracket
$$
 [(\xi_1,t_1),(\xi_2,t_2)]:= ([\xi_1,\xi_2], \int_{S^1}(\xi_1,d\xi_2)).
$$
(Recall that a function $\xi:S^1\to \mathfrak{g}$ is a $0$-form in $\Omega^0(S^1;\mathfrak{g})$; therefore, $d\xi \in \Omega^1(S^1;\mathfrak{g})$, and $(\xi,d\xi)\in \Omega^1(S^1)$.)  Note that the central extension $\Lghat$ depends on the choice of invariant inner product on $\mathfrak{g}$.  And recall that if $G$ is simple (so that $\mathfrak{g}$ is simple), then all such inner products are multiples of the basic inner product (see Remark \ref{remark:innerproducts}). 

The central extension $\LGhat$ mentioned above is to have $\Lghat$ as Lie algebra, so that the following diagram commutes.
\begin{diagram}
0 & \rTo & \R & \rTo & \Lghat & \rTo & \Lg & \rTo& 0 \\
    &        &  \dTo>\exp   &   & \dTo>\exp   &        &   \dTo>\exp   & &\\
1 & \rTo & S^1 & \rTo & \LGhat & \rTo & LG & \rTo &1 \\
\end{diagram}
Such an extension defines a principal $S^1$-bundle over $LG$, which is classified by its Chern class $c$ in $H^2(LG;\Z)$.  Therefore, there may be topological obstructions to the existence of such a central extension. 

Since $LG=\Omega G \times G$ (topologically), then $H^2(LG;\Z)\cong H^2(\Omega G;\Z)$, as $G$ is simply connected.  Furthermore, if $G$ is simple, then $H^2(\Omega G;\Z)\cong H^3(G;\Z)\cong \Z$, where the isomorphism $H^2(\Omega G;\Z)\cong H^3(G;\Z)$ is the \emph{transgression}---the first non-trivial differential in the Serre spectral sequence for the path space fibration $PG \to G$ with fibre $\Omega G$.  

\begin{prop} \label{prop:simplyconnectedloopgroup} Let $G$ be a simple, compact, simply connected Lie group.  If the Chern class $c(\LGhat) \in H^2(LG;\Z)\cong \Z$ of the principal $S^1$-bundle $\LGhat \to LG$ is  $k$ times a generator, then  $\pi_1(\LGhat)\cong \Z_k$.  In particular, if $k=1$, then $\LGhat$ is simply connected.
\end{prop}
\begin{proof}
Let $\rho:LG\to \mathbb{C} P^\infty$ denote the classifying map of the $S^1$-bundle, so that $c(\LGhat)=\rho^*(y)$, where $y\in H^2(\mathbb{C}P^\infty;\Z)\cong \Z$ is a generator, and consider the Serre spectral sequence for the fibration $LG \to \mathbb{C}P^\infty$, with $E_2=H^*(\mathbb{C}P^\infty;H^*(\LGhat;\Z))$ converging to $H^*(LG;\Z)$

Since $\rho^*(y)$ is torsion-free, there is no non-trivial differential into $E_2^{0,2}\cong H^2(\mathbb{C}P^\infty;\Z)$.  Therefore, if $E_2^{0,1}\cong H^1(\LGhat;\Z)$ is non-zero, then it must survive to $E_\infty$, which contradicts the fact that $LG$ is simply connected.  Also, since $E_2^{0,3}\cong H^3(\mathbb{C}P^\infty;\Z)=0$, the group $E_2^{0,2}$ survives to $E_\infty$, and by the convergence of the spectral sequence, $H^2(LG;\Z)$ fits in the exact sequence
$$
0\to H^2(\mathbb{C}P^\infty;\Z) \toby{\rho^*} H^2(LG;\Z) \to E_2^{0,2} \to 0.
$$
Since $\rho^*$ is multiplication by $k$, it follows that $E_2^{0,2}\cong \Z_k$, which shows that $H^2(\LGhat;\Z)\cong \Z_k$. By the Universal coefficient theorem, $H_1(\LGhat;\Z)\cong \Z_k$, which by the Hurewicz theorem, and the fact that the fundamental group of a topological group is abelian, proves the result.
\end{proof}

The following Theorem classifies central extensions $\LGhat$ with Lie algebra $\Lghat$. (See Theorem 4.4.1 and Proposition 4.5.5 in \cite{PS}.)

\begin{thm} \cite{PS} Let $G$ be a simple, compact, simply connected Lie group.  There is a unique central extension $\LGhat$ with Lie algebra $\Lghat$ if and only if the chosen invariant inner product on $\mathfrak{g}$ is  $kB$, where $k \in \N$, and $B$ denotes the basic inner product. In this case, the Chern class of the corresponding $S^1$-bundle $\LGhat \to LG$ is $c(\LGhat)=k\cdot 1$, where $1\in H^2(LG;\Z)\cong \Z$ is a generator.
\end{thm}

Notice that the above Theorem and the preceding Proposition show that the central extension $\LGhat$ corresponding to the basic inner product (i.e. $k=1$) is the only central extension that is simply connected.   From now on, let $\LGhat$ denote the central extension corresponding to  the basic inner product.

\section*{Coadjoint representation of $\LGhat$}

Since the central subgroup $S^1\subset \LGhat$ acts trivially on $\Lghat$, both the adjoint and coadjoint representation of $\LGhat$ factor through $LG$.  As indicated in \cite{PS}, or \cite{KW}, the adjoint $LG$-action on $\Lghat$ is given by
$$
\mathrm{Ad}_g(\xi,t)=(\mathrm{Ad}_g\xi, t-\int_{S^1} (g^*\theta^L,\xi)).
$$ 
To describe the corresponding coadjoint action, it will be necessary to first describe the dual of the Lie algebra $\Lghat$.  

Let $\Lg^*=\Omega^1(S^1;\mathfrak{g})$, sometimes called the \emph{smooth dual} of $\Lg$. The pairing $\Lg\times \Lg^*\to \R$ given by $(\xi,A)\mapsto \int_{S^1}(\xi,A)$, induces an inclusion $\Lg^*\subset (\Lg)^*$.  Additionally, define
$\Lghat^*:=\Lg^*\oplus\R$ and consider the pairing $\Lghat\times \Lghat^*\to\R$ given by 
$$
((\xi,a), (A,t)) = \int_{S^1} (\xi,A)+ at.
$$ 

The corresponding  coadjoint action of $LG$ on $\Lghat^*$ is therefore (see \cite{PS} or \cite{KW}):
$$
g\cdot(A,t)=(\mathrm{Ad}_g(A)-t g^*\theta^R,t).
$$   
Notice that for each real number $\lambda$, the hyperplanes $t=\lambda$ are fixed.  Identifying $\Lg^*$ with $\Lg^* \times \{\lambda \} \subset \Lghat^*$ yields an action of $LG$ on $\Lg^*$, called the (affine) level $\lambda$ action.   

When $\lambda=1$, the $LG$-action on $\Lg^*$ reads,
$$
g\cdot A = \mathrm{Ad}_g(A)- g^*\theta^R,
$$
which may be interpreted as the action of a gauge transformation $g\in \mathrm{Map}(S^1,G)$ on a connection $A\in \Omega^1(S^1;\mathfrak{g})$ on the trivial principal bundle  $S^1\times G \to S^1$.  In this setting, $\Omega G$ is then the group of based gauge transformations, which acts freely on $L\mathfrak{g}^*$.  And since two connections have the same holonomy if and only if they are gauge equivalent by a based gauge transformation, then holonomy $\Hol:\Lg^*\to G$ descends to the identification $G=\Lg^*/\Omega G$.  Moreover, since $\Lg^*$ is contractible, the fibration $\Hol:\Lg^*\to  G$ is a universal $\Omega G$-bundle.

\section{Hamiltonian loop group actions} \label{sec:loopgroup}

As mentioned previously, the theory of quasi-Hamiltonian $G$-spaces is equivalent to the theory of Hamiltonian loop group actions.  Since Hamiltonian loop group actions are infinite dimensional counterparts of Hamiltonian actions by compact Lie groups, this avenue of study leads to another comparison of quasi-Hamiltonian $G$-spaces with their classical Hamiltonian counterparts. 

This section reviews Hamiltonian loop group actions, and recalls the equivalence with quasi-Hamiltonian $G$-spaces (see Theorem \ref{thm:equivalence}).  
\\

\begin{defn} A Hamiltonian loop group manifold at level $\lambda \in \R$ is a triple
$(\mathcal{M},\sigma,\varphi)$ consisting of a (Banach) manifold $\mathcal{M}$ equipped with an $LG$-action, an invariant $2$-form $\sigma$, and an equivariant map $\varphi:\mathcal{M}\to \Lg^*$, where $\Lg^*$ comes equipped with the level $\lambda$ action of $LG$, satisfying the following three properties:
\begin{enumerate}
\item $d\sigma =0$
\item $\iota_{\xi^\sharp}\sigma = -d(\varphi,\xi)$
\item $\sigma$ is weakly non-degenerate (i.e. induces an injection $T_x\mathcal{M} \to T^*_x\mathcal{M}$)
\end{enumerate}
\end{defn}

\begin{eg}
Coadjoint orbits in $\Lg^*$ \label{eg:coadjointorbits}
\end{eg}

Analogous to the classical setting, important examples of  Hamiltonian $LG$-spaces are the coadjoint orbits $\mathcal{O}=LG\cdot (A,\lambda)\subset \Lghat^*$ (see \cite{MW}), which carry the natural Kirillov-Kostant-Souriau symplectic structure.  Restricting to the level $\lambda=1$, a given coadjoint orbit $\mathcal{O}=LG\cdot A\subset \Lg^*$ may be viewed as the space of connections in $\Omega^1(S^1;\mathfrak{g})$ that are gauge equivalent to $A$.  The connections in the $\Omega G$-orbits of $\mathcal{O}$ are characterized by the fact that they have the same holonomy.  Therefore, elements in the orbit space $\mathcal{O}/\Omega G$ are determined by their holonomy.  And since gauge equivalent connections have conjugate holonmy, $\mathcal{O}/\Omega G$ is a conjugacy class in $G$.  In other words, there is a conjugacy class $\C\subset G$ such that following is a pullback diagram.
\begin{diagram}
\mathcal{O} & \rTo & \Lg^* \\
\dTo<{\Hol}      &              &  \dTo>{\Hol} \\
\C           & \rTo  & G \\
\end{diagram}

Recall from Example \ref{eg:conjclass} that conjugacy classes $\C\subset G$ are examples of quasi-Hamiltonian $G$-spaces.  The relationship between the coadjoint orbit $\mathcal{O}$ and the conjugacy class $\mathcal{C}$ is an instance of Theorem \ref{thm:equivalence} below.  \hfill $\qed$
\\

Since the based loop group $\Omega G\subset LG$ acts freely on
$\Lg^*$, by equivariance of $\varphi$, it acts freely on
$\mathcal{M}$; therefore, as in the preceding example, there is a pull-back diagram
\begin{diagram}
\mathcal{M} & \rTo{\varphi} & \Lg^* \\
\dTo<q      &              &  \dTo>{\Hol} \\
M           & \rTo^{\phi}  & G \\
\end{diagram}
inducing a $G=LG/\Omega G$-equivariant map $\phi:M\to G$.
When $\varphi$ is proper, the quotient $M=\mathcal{M}/\Omega G$ is
a finite dimensional manifold, and the following result of \cite{AMM} describes the correspondence between quasi-Hamiltonian $G$-spaces and Hamiltonian loop group actions.

\begin{thm} \cite{AMM}  \label{thm:equivalence} Let
$(\mathcal{M},\sigma,\varphi)$ be a Hamiltonian $LG$-space with
proper moment map $\varphi$, and let $M$ and $\phi$ be as above.
There is a unique two-form $\omega$ on $M$ with $q^*\omega =
\sigma + \varphi^*\varpi$ with the property that $(M,\omega,\phi)$
is a quasi-Hamiltonian $G$-space. Conversely, given a quasi-Hamiltonian $G$-space
$(M,\omega,\phi)$  there is a
unique Hamiltonian $LG$-space $(\mathcal{M},\sigma,\varphi)$ such
that $M=\mathcal{M}/\Omega G$.
\end{thm}

\begin{remark} 
Since $\Lg^*$ is contractible, the pullback $\Hol^*\eta$ of the canonical (closed) $3$-form on $G$ is exact.  The $2$-form $\varpi\in \Omega^2(\Lg^*)$ appearing in the above Theorem is a particular choice of primitive for $-\Hol^*\eta$ that is $LG$-invariant. See \cite{AMM} for details.
\end{remark}

\section{Pre-quantization of Hamiltonian $LG$-actions}

Let $\mathcal{M}$ be a Hamiltonian $LG$-space at level $\lambda$.
It is convenient to view the  $LG$-space $\mathcal{M}$ as an $\LGhat$ space
where the central circle acts trivially.  With this viewpoint, the moment map for the $\LGhat$-action is  $(\varphi,\lambda):\mathcal{M} \to \Lghat^*$,  defined by $m\mapsto (\varphi(m),\lambda)$.  Obviously, $\Lg^*$ therefore comes equipped with the level $\lambda$-action, so that $(\varphi,\lambda)$ is $\LGhat$-equivariant.

Pre-quantization in the context of Hamiltonian loop group actions appears in Definition \ref{def:preqloopgroupaction}.  As discussed at the  beginning of Section \ref{sec:quasi}, a pre-quantum line bundle for an ordinary Hamiltonian $G$-space $(M,\sigma,\Phi)$ is a $G$-equivariant complex line bundle over $M$ whose equivariant curvature class is $[\sigma_G]\in H^2_G(M;\R)$.  Considering instead the associated principal $U(1)$-bundle, a pre-quantization may also be defined as a $G$-equivariant principal $U(1)$-bundle over $M$, equipped with a $G$-invariant connection $\theta \in \Omega^1(P)^G$ whose equivariant curvature class is $[\sigma_G]$.  Recall that to say that $\sigma_G(\zeta)=\sigma+(\Phi,\zeta)$ is the equivariant curvature of  $\theta$ is to require that $d_G\theta=-\pi^*\sigma_G$.  Since 
$$
(d_G\theta)(\zeta)=d(\theta(\zeta))-\iota_\zeta \theta(\zeta) = d\theta -\theta(\zeta^\sharp),
$$
and 
$$
(\pi^*(\sigma_G))(\zeta)=\pi^*\sigma+ \pi^*(\Phi,\zeta),
$$
 the invariant connection $\theta$ must satisfy both $d\theta=-\pi^*\sigma$, and $\pi^*(\Phi,\zeta)=\theta(\zeta^\sharp)$.

\begin{defn} \label{def:preqloopgroupaction} Let  $(\mathcal{M},\sigma,\varphi)$ be a Hamiltonian $LG$-space at level $\lambda$.  A pre-quantization of $(\mathcal{M},\sigma,\varphi)$ is an $\LGhat$-equivariant
principal $U(1)$-bundle $\pi:P\to \mathcal{M}$ equipped with an $\LGhat$-invariant connection $\theta$  satisfying  $\pi^*\sigma= -d\theta $, and $\pi^*((\varphi,\lambda),(\xi,a))= \theta((\xi,a)^\sharp)$, for all  $(\xi,a)$ in $\Lghat$.
\end{defn}

Since the central circle $S^1\subset \LGhat$ acts trivially on $\mathcal{M}$, it therefore acts on the fibres of $\pi:P\to \mathcal{M}$. That is, there is homomorphism $S^1 \to U(1)$ from the central circle of $\LGhat$ to the structure group of $\pi:P\to \mathcal{M}$, $z\mapsto z^k$. At the level of Lie algebras, this gives the homomorphism $a\mapsto ka$, where $a\in \R=Lie(S^1)=u(1)$.

 If $(0,a)\in \Lghat$, then a pre-quantization of $(\mathcal{M}, \sigma, \varphi)$ requires that $\pi^*((\varphi,\lambda),(0,a))= \theta((0,a)^\sharp)$, or equivalently, $\lambda a = \theta((0,a)^\sharp)$.  Let $\frac{\partial}{\partial t}=1^\sharp$ denote the generating vector field $1\in u(1)$ of the action of the structure group, so that $\theta(\frac{\partial}{\partial t}) =1$.  Since the central circle acts via the homomorphism $z\mapsto z^k$, it follows that $(0,a)^\sharp=(ka)^\sharp$, where the left side of the equality is the generating vector field for the central circle action, and the right side is the generating vector field for the action of the structure group of $\pi:P\to \mathcal{M}$.  That is, $(0,a)^\sharp =ka\frac{\partial}{\partial t}$, and hence $\theta((0,a)^\sharp) =ka$. This shows that a pre-quantization requires that the level $\lambda$ be  the weight of the central circle action on the fibres of $P$. In particular, in order for a pre-quantization to exist, the underlying level $\lambda$ must be an integer.

A pre-quantization of a Hamiltonian $LG$-space also places a restriction on the symplectic form $\sigma$. Indeed, the condition $d\theta=-\pi^*\sigma$ says that $\sigma$ is the curvature of $\theta$; therefore, analogous to the pre-quantization of symplectic manifolds, the cohomology class $[\sigma]\in H^2(\mathcal{M};\R)$ must have an integral lift, namely the Chern class $c(P) \in H^2(\mathcal{M};\Z)$ of the bundle $ \pi:P\to \mathcal{M}$.

The following Proposition relates the distinguished cohomology classes $[\sigma] \in H^2(\mathcal{M};\R)$, and $[(\omega,\eta)]\in H^3(\phi;\R)$, which play a major role in the pre-quantization of Hamiltonian loop group actions, and their quasi-Hamiltonian counterparts.

\begin{prop} \label{prop:preqloopgroup1}
Let $(M,\omega,\phi)$ be a quasi-Hamiltonian $G$-space, and let $(\mathcal{M},\sigma,\varphi)$ be the corresponding Hamiltonian $LG$-space.  For any coefficient ring $R$, there is a natural isomorphism 
$\alpha:H^3(\phi;R)\to H^2(\mathcal{M};R)$.  When $R=\R$,  $\alpha[(\omega,\eta)]= [\sigma]$.
\end{prop}
\begin{proof} Recall the following pullback diagram that relates Hamiltonian $LG$-space $\mathcal{M}$ to the quasi-Hamiltonian $G$-space $M=\mathcal{M}/\Omega G$.
\begin{diagram}
\mathcal{M} & \rTo^{\varphi} & \Lg^* \\
\dTo<q      &              &  \dTo>{\Hol} \\
M           & \rTo^{\phi}  & G \\
\end{diagram}
Note that since $\Hol:\Lg^*\to G$ is the universal $\Omega G$-bundle over $G$, $\mathcal{M}$ is the homotopy fibre of $\phi$.  

Since the long exact sequence in cohomology associated to a map of spaces is natural, the above commutative square induces maps $(q,\Hol)^*:H^*(\phi)\to H^*(\varphi)$.  Since $\Lg^*$ is contractible, $H^q(\mathcal{M})\toby\cong H^{q+1}(\varphi)$ in one of those long exact sequences.  Another way of saying this is to observe that the above diagram induces a map of homotopy cofibres $C_\phi \to C_\varphi$, and hence a map $H^*(C_\varphi) \to H^*(C_\phi)$.  And since $\Lg^*$ is contractible, $C_\varphi \approx \Sigma \mathcal{M}$, so that $H^{q+1}(C_\varphi) \cong H^{q+1}(\Sigma \mathcal{M}) \cong H^q(\mathcal{M})$.

Let $\alpha:H^3(\phi) \to H^2(\mathcal{M})$ be the composition $H^3(\phi) \to H^3(\varphi) \cong H^2(\mathcal{M})$.  That  $\alpha$ is an isomorphism will follow from  the following commutative diagram with exact rows (by the five-lemma). 
\begin{equation} \label{diagram:alpha}
\begin{diagram}
0 & \rTo & H^2(M) & \rTo & H^3(\phi) & \rTo & H^3(G) & \rTo & H^3(M) & \rTo &\cdots \\
& & \dEquals & & \dTo>\alpha & & \dEquals & &\dEquals & &\\
0 & \rTo & H^2(M) & \rTo & H^2(\mathcal{M}) & \rTo^\tau & H^3(G) & \rTo & H^3(M) & \rTo &\cdots \\
\end{diagram}
\end{equation}
It will be verified next that the above diagram commutes.

A quick description of the diagram is in order.  The top row of the diagram is the long exact sequence in cohomology associated to $\phi$, and the bottom row is the Serre exact sequence for the fibration sequence $\mathcal{M} \to M\toby\phi G$ (which terminates after degree $3$).  The map labeled $\tau$ is the transgression in that spectral sequence.

Since the composition $\mathcal{M} \toby{q} M \to G$ is null homotopic, by Proposition \ref{prop:extend} there is an induced map from the mapping cone $C_q \to G$ that fits in the following diagram whose rows are cofibration sequences. (The map $\Sigma \mathcal{M} \to C_\phi$ is an induced map of cofibres.)
\begin{diagram}
M & \rTo & C_q & \rTo & \Sigma \mathcal{M} \\
\dEquals & & \dTo & & \dTo \\
M & \rTo^\phi & G & \rTo & C_\phi
\end{diagram}
Rewriting $H^3(\Sigma\mathcal{M})\cong H^2(\mathcal{M})$, this yields the two exact rows of the following commutative diagram.
\begin{diagram}
\cdots & \rTo & H^2(M) & \rTo & H^3(\phi) & \rTo & H^3(G) & \rTo & H^3(M)& \rTo & \cdots \\
& & \dEquals & & \dTo &\ruTo<\tau & \dTo & &\dEquals & &\\
\cdots & \rTo & H^2(M) & \rTo & H^2(\mathcal{M}) & \rTo & H^3(C_q) & \rTo & H^3(M) & \rTo & \cdots\\
\end{diagram}
Theorem 11.9.6 in \cite{S}, says that the transgression $\tau:H^2(\mathcal{M})\to H^3(G)$ makes the lower triangle in the above diagram commute.  
As will be verified shortly, the induced map $H^3(G)\to H^3(C_q)$ is an isomorphism.  Therefore, the upper triangle commutes as well, which shows that the diagram \ref{diagram:alpha} commutes.

The map $C_q \to G$  is actually a homotopy equivalence in dimensions $\leq$ 3.  One way to see this is to use Ganea's Theorem (see \cite{S}), which describes the homotopy type of the homotopy fibre of $C_q\to G$ as $\Sigma(\mathcal{M}\wedge \Omega G)$.  Since $\Omega G$ is simply connected, and $\mathcal{M}$ is connected, $\Sigma (\mathcal{M} \wedge \Omega G)$ is at least $3$-connected.  Therefore, $C_q \to G$ is a homotopy equivalence in dimensions $\leq$ 3, as stated.

Finally, to check that $\alpha[(\omega,\eta)]=[\sigma]$ when $R=\R$, it suffices to show that under the map $H^3(\phi;\R) \to H^3(\varphi;\R)$, $[(\omega,\eta)]\mapsto [(\sigma,0)]$. Indeed,
\begin{align*}
 (q,\Hol)^*[(\omega,\eta)] & = [(q^*\omega,\Hol^*\eta)]\\
		& = [(\sigma +\varphi^*\varpi,d\varpi)] \\
		& = [(\sigma,0)+d(0,\varpi)] \\
		& = [(\sigma,0)] 
\end{align*}
\end{proof}

\begin{remark}
The proof of Proposition \ref{prop:preqloopgroup1} actually shows that $[\sigma]$ transgresses to $[\eta] \in H^3(G;\R)$.  This reaffirms that a necessary condition for pre-quantization is that the underlying level $\lambda$ be an integer. 
\end{remark}

Note that the preceding proposition shows that the cohomology class $[\sigma] \in H^2(\mathcal{M};\R)$ has an integral lift if and only if $[(\omega,\eta)] \in H^3(\phi;\R)$ does as well.  By Proposition \ref{prop:eq-pqispq}, it is therefore expected that $(\mathcal{M},\sigma,\varphi)$ admits a pre-quantization if and only if $[\sigma]$ is integral.  This is indeed the case, as will be shown next.  

\begin{prop} \label{prop:preqloopgroup2}
Let $(\mathcal{M},\sigma,\varphi)$ be a Hamiltonian $LG$-space at level $k\in \N$. If $[\sigma]$ is integral, then $(\mathcal{M},\sigma,\varphi)$ admits a pre-quantization.
\end{prop}
\begin{proof} Since $[\sigma]$ is integral, there is a principal $U(1)$-bundle $P\to \mathcal{M}$ with a connection $\theta$ whose curvature is $\sigma$.  It remains to show that this bundle admits a lifting of the $\LGhat$-action to $P$ that leaves $\theta$ invariant and satisfies
\begin{equation} \label{second}
\pi^*((\varphi,\lambda),(\xi,a))= \theta((\xi,a)^\sharp),
\end{equation}
 for all  $(\xi,a)$ in $\Lghat$.  The formula
$$
(\xi,a)^\sharp = \mathrm{Lift}((\xi,a)^\sharp) + 
				((\pi^*\varphi,k),(\xi,a))\frac{\partial}{\partial t}	 
$$
due to Kostant \cite{Kos} provides a lift of the $\Lghat$-action on $\mathcal{M}$ to an $\Lghat$-action on $P$, which preserves $\theta$. Here, $\mathrm{Lift}(X)$ denotes the horizontal lift of the vector field $X$ on $\mathcal{M}$ determined by the connection $\theta$.  If this $\Lghat$-action on $P$ integrates to a $\LGhat$-action, then condition \ref{second} is automatically satisfied. 

Since $\LGhat$ is connected and simply connected, the $\Lghat$-action on $P$ integrates to an $\LGhat$-action, provided the generating vector fields $(\xi,a)^\sharp$ on $P$ are complete.  To see that this is the case, by Proposition 4.1.19 in \cite{AMR}, it suffices to check that for every maximal integral curve $\gamma(t)$ on $P$ and every finite interval $(a,b)$ on which $\gamma$ is defined, $\gamma((a,b))$ lies in some compact subset of $P$.  Since the projection of $\pi\circ \gamma$ is an integral curve of a complete vector field on $\mathcal{M}$, $\pi(\gamma((a,b))$ lies in some (closed) piece of a maximal integral curve $c(t)$ on $\mathcal{M}$, say $c([t_1,t_2])$. Therefore $\gamma((a,b))$ is contained in the cylinder $\pi^{-1}(c([t_1,t_2]))$, which is compact. 
\end{proof}

Propositions \ref{prop:preqloopgroup1} and \ref{prop:preqloopgroup2} may be summarized as the main result of this chapter.

\begin{thm} \cite {AM} Under the correspondence in Theorem \ref{thm:equivalence}, a pre-quantization of the quasi-Hamiltonian $G$-space $(M,\omega,\phi)$ corresponds to a pre-quantization of the Hamiltonian $LG$-space $(\mathcal{M},\sigma,\varphi)$. 
\end{thm}

\appendix

\chapter{Cohomology of compact simple Lie groups}\label{chapter:app}

This appendix lists the cohomology Hopf algebras for compact simple Lie groups $G$ that are not simply connected with coefficients in $\Z_p$, where $p$ is a prime that divides the order $\pi_1(G)$.  Note that  $\hat{\phantom{x}}$ denotes omission, and that the degree of a cohomology class listed as $x_i$ is given by the subscript $i$.

Also, recall that the \emph{reduced coproduct} $\bar\mu^*$ is defined by the equation $\mu^*(x) = x\otimes 1 + 1 \otimes x + \bar\mu^*(x)$ in $H^*(G)\otimes H^*(G) \cong H^*(G\times G)$.

\begin{appthm} \label{thm:BB} \cite{BB}
Let $p$ be a prime dividing $l$.  Let $n=p^rn'$ and $l=p^sl'$
where $p$ is relatively prime to both $n'$ and $l'$. If $p\neq 2$
or $p=2$ and $s>1$, 
  $$H^*(SU(n)/\Z_l;\Zp)\cong \Zp[y_2]/(y_2^{p^r})  \otimes \Lambda(x_1, \ldots, \hat{x}_{2p^r-1}, \ldots, x_{2n-1})$$ as an algebra,
 and $\displaystyle \bar\mu^*(x_{2i-1})=\delta_{rs} x_1\otimes y_2^{i-1} +
     \sum_{j=2}^{i-1}\binom{i-1}{j-1} x_{2j-1}\otimes y_2^{i-j},\quad i\geq 2$,
and $\bar\mu^*(x_1)=\bar\mu^*(y_2)=0$.
If $p=2$ and $s=1$ then the cohomology with $\Ztwo$ coefficients
is as above, with the additional relation that $y_2=x_1^2$.
\end{appthm}

\begin{appthm} \cite{BB} \label{thm:PSp} Write $n=2^r n'$ where $n'$ is odd. 
 $$H^*(PSp(n);\Ztwo) \cong\Ztwo[v_1]/(v_1^{2^{r+2}} ) \otimes \Lambda(b_3,\ldots,\hat{b}_{2^{r+2}-1}, \ldots , b_{4n-1})  $$ as an algebra, and
$\displaystyle \bar\mu^*b_{4k+3}=\sum_{i=1}^{k-1} \binom{k}{i} b_{4i+3} \otimes v_1^{4k-4i}
$ for $k\geq 2$, while $\bar\mu^* b_3=0$, and $\bar\mu^* b_7=b_3\otimes v_1^4$.
\end{appthm}

Let $V(x_{i_1}, \ldots x_{i_k})$ be the commutative, associative algebra over $\Z_2$ such that the monomials $x_{i_1}^{\epsilon_1} \cdots x_{i_k}^{\epsilon_k}$ ($\epsilon_j=0$ or $1$)  form an additive basis, and $x_{i_j}^2=x_{2i_j}$ if $2i_j\in \{i_1, \ldots, i_k\}$ and zero otherwise.

\begin{appthm} \cite{BB} \label{thm:SO} $H^*(SO(n);\Z_2) \cong V(x_1, \ldots, x_{n-1})$ as an algebra, and each $x_j$ is primitive. 
\end{appthm}

\begin{appthm} \cite{BB} \label{thm:PO} Write $n=2^rn'$  where $n'$ is odd. Then
 $$H^*(PO(n);\Z_2)\cong \Z_2[v_1]/(v^{2^r}) \otimes V(u_1, \ldots, \hat{u}_{2^{r}-1} \ldots u_{n-1}) $$ as an algebra, and  $\displaystyle \bar\mu^*u_k=\sum_{i=1}^{k-1} \binom{k}{i} u_i\otimes v_1^{k-i}$, $k\geq 2$.
\end{appthm}

\begin{appthm} \cite{IKT} \label{thm:Ss} Write $n=2^rn'$ where $n'$ is odd, and $r\geq 2$.  Let $s$ be defined by the inequalities $2^{s-1}<n\leq 2^s$, and let $N=\{2^r-1, 1, 2, 2^2, 2^3, \ldots, \}$. Then 
$$H^*(Ss(n);\Z_2)\cong \Z_2[y_1]/(y_1^{2^r}) \otimes V(x_i, 0<i<n, i\notin N, z_{2^{s-1}}) 
$$ as an algebra, and 
$$ \bar\mu^*x_{2j}=\sum_{k=1}^{j-1} \binom{j}{k} y_1^{2k}\otimes x_{2j-2k}, \quad \bar\mu^*x_{2j-1}=x_{2j-2}\otimes y_1 + \sum_{k=1}^{j-1} \binom{j-1}{k}y_1^{2k} \otimes x_{2j-2k-1}.
$$
\end{appthm}

\begin{appthm}\cite{Ko} \label{thm:PEsix} As an algebra,
$$H^*(PE_6; \Z_3) \cong  \Z_3[y_2, y_8]/(y_2^9,y_8^3) \otimes  \Lambda(x_1, x_3, x_7, x_9, x_{11}, x_{15}).$$  The non-trivial reduced coproducts are given by 
\begin{align*}
\bar\mu^*x_3&=y_2\otimes x_1, \quad \bar\mu^*x_7=y^3_2\otimes x_1, \quad \bar\mu^*y_8=y_2^3\otimes y_2 \\
  \bar\mu^*x_9&=y_2\otimes x_7-y_2^3\otimes x_3 +y_8\otimes x_1 + y_2^4\otimes x_1 \\
  \bar\mu^*x_{11}&=y_2\otimes x_9 -y_2^2\otimes x_7 +y_8\otimes x_3-y_2^4 \otimes x_3 + y_8y_2\otimes x_1 -y_2^5\otimes x_1 \\
  \bar\mu^*x_{15}&=y_2^3\otimes x_9 +y_8\otimes x_7 + y_2^6\otimes x_3 + y_8y_2^3\otimes x_1.
 \end{align*}
\end{appthm}

\begin{appthm}\cite{IKT} \label{thm:PEseven} As an algebra,
$$
H^*(PE_7;\Z_2)\cong \Z_2[x_1, x_5, x_9]/(x_1^4, x_5^4, x_9^4) \otimes \Lambda(x_6, x_{15}, x_{17}, x_{23}, x_{27}).
$$
The non-trivial reduced coproducts are given by $\bar\mu^*x_{15}=x_5^2\otimes x_5$, $\bar\mu^*x_{23}=x_9^2\otimes x_5 + x_6\otimes x_{17}$, and $\bar\mu^*x_{27}=x_9^2\otimes x_9 + x^2_5\otimes x_{17}$.
\end{appthm}

\addcontentsline{toc}{chapter}{Bibliography}


\end{document}